\documentclass[12pt]{amsart}
\usepackage{mathrsfs,amscd,amssymb,amsthm,amsmath,csquotes,graphicx,array,wasysym,textcomp,tabularx,comment,epigraph}
\usepackage{supertabular,epic,rotating,fancyvrb,pdflscape,upgreek,longtable,pifont,tikz}
\usepackage{adjustbox}
\usepackage[matrix,arrow,curve]{xy}
\usepackage[noadjust]{cite}
\usepackage[english]{babel}
\usepackage[utf8]{inputenc}
\usepackage[margin=2cm]{geometry}

\newtheorem{theorem}[equation]{Theorem}

\newtheorem{proposition}[equation]{Proposition}
\newtheorem{lemma}[equation]{Lemma}
\newtheorem{corollary}[equation]{Corollary}
\newtheorem{conjecture}[equation]{Conjecture}
\newtheorem{problem}[equation]{Problem}
\newtheorem{question}[equation]{Question}

\theoremstyle{definition}
\newtheorem{example}[equation]{Example}
\newtheorem{definition}[equation]{Definition}

\theoremstyle{remark}
\newtheorem{remark}[equation]{Remark}

\makeatletter\@addtoreset{equation}{section} \makeatother

\makeatletter
\ifcase \@ptsize \relax
  \newcommand{\miniscule}{\@setfontsize\miniscule{3}{7}}
    \newcommand{\stiny}{\@setfontsize\miniscule{5}{7}}
\or
  \newcommand{\miniscule}{\@setfontsize\miniscule{3}{7}}
   \newcommand{\stiny}{\@setfontsize\miniscule{5}{7}}
\or
  \newcommand{\miniscule}{\@setfontsize\miniscule{3}{7}}
    \newcommand{\stiny}{\@setfontsize\miniscule{5}{7}}
\fi
\makeatother

\newcommand{\PGL}{\operatorname{PGL}}
\newcommand{\Ga}{{\mathbb G}_{\mathrm{a}}}
\newcommand{\Gm}{{\mathbb G}_{\mathrm{m}}}
\newcommand{\Gr}{\operatorname{Gr}}

\newcommand{\Spec}{\operatorname{Spec}}

\newcommand{\Pic}{\operatorname{Pic}}

\newcommand{\Supp}{\operatorname{Supp}}

\newcommand{\Aut}{\operatorname{Aut}}

\newcommand{\GL}{\operatorname{GL}}
\newcommand{\SL}{\operatorname{SL}}
\newcommand{\PSL}{\operatorname{PSL}}

\newcommand{\chit}{\chi_{\mathrm{top}}}
\newcommand{\g}{\operatorname{g}}
\newcommand{\pr}{\operatorname{pr}}

\newcommand{\Cl}{\operatorname{\Cl}}
\newcommand{\Pf}{\operatorname{\Pf}}

\newcommand{\SAut}{\operatorname{SAut}}

\newcommand{\mumu}{{\boldsymbol{\mu}}}

\newcommand{\CC}{\mathbb{C}}
\newcommand{\ZZ}{\mathbb{Z}}
\newcommand{\PP}{\mathbb{P}}
\newcommand{\QQ}{\mathbb{Q}}
\newcommand{\FF}{\mathbb{F}}

\def\A{{\mathbb A}}

\newcommand{\OOO}{{\mathscr{O}}}
\newcommand{\NNN}{{\mathscr{N}}}

\renewcommand{\phi}{\varphi}
\renewcommand{\emptyset}{\varnothing}

\begin{document}
\title{Cylinders in Fano varieties}

\author {Ivan Cheltsov, Jihun Park, Yuri Prokhorov, and Mikhail Zaidenberg}

\address{\emph{Ivan Cheltsov}\newline
\textnormal{School of Mathematics, The University of Edinburgh,  Edinburgh, Scotland}
\newline
\textnormal{Laboratory of Mirror Symmetry, NRU HSE, 6 Usacheva street, Moscow, Russia}}
\email{I.Cheltsov@ed.ac.uk}

\address{ \emph{Jihun Park}\newline
\textnormal{Center for Geometry and Physics, Institute for Basic Science
\newline\medskip
77 Cheongam-ro, Nam-gu, Pohang, Gyeongbuk, 37673, Korea
\newline
Department of Mathematics, POSTECH
\newline
77 Cheongam-ro, Nam-gu, Pohang, Gyeongbuk,  37673, Korea}}
\email{wlog@postech.ac.kr}

\address{\emph{Yuri Prokhorov}\newline
\textnormal{Steklov Mathematical Institute, 8 Gubkina street, Moscow, Russia}
\newline
\textnormal{Laboratory of Algebraic Geometry, NRU HSE, 6 Usacheva street, Moscow, Russia}}
\email{prokhoro@mi-ras.ru}

\address{\emph{Mikhail Zaidenberg}\newline
\textnormal{Universit\'e Grenoble Alpes, CNRS, Institut Fourier, F-38000 Grenoble, France}}
\email{Mikhail.Zaidenberg@univ-grenoble-alpes.fr}

\begin{abstract}
This paper is a survey about cylinders in Fano varieties and related problems.
\end{abstract}

\maketitle

\tableofcontents

Throughout this paper except for Section~\ref{subsection:compactifications},
we always assume that all varieties are defined over an~algebraically closed field $\Bbbk$ of characteristic $0$.

\section{Introduction}
\label{section:Intro}
A \emph{cylinder} in a~projective variety $X$ is a~Zariski open subset $U\subset X$ such that
$$
U\cong\mathbb{A}^1\times Z
$$
for an affine variety $Z$.
If $X$ contains a~cylinder, we say that $X$ is \emph{cylindrical}.
Since cylindrical varieties have negative Kodaira dimension,
we will focus our attention on cylindrical Fano varieties,
because they are building blocks of projective varieties with negative Kodaira dimension.

\begin{example}
\label{example:Grassmannian} For positive integers $m, n$ with $m<n$,
let $X$ be the~Grassmannian $\mathrm{Gr}(m,n)$  of
$m$-dimensional subspaces of an $n$-dimensional vector space over $\Bbbk$.
Then $X$ is a smooth projective variety of dimension~$m(n-m)$,
and $-K_X\sim nH$, where $H$ is an ample generator of the~group $\mathrm{Pic}(X)$.
Since $X$ contains an open Schubert cell isomorphic to $\A^{m(n-m)}$, it is a cylindrical Fano variety.
\end{example}

However, not all Fano varieties are cylindrical,
e.g. smooth cubic threefolds and smooth quartic threefolds do not contain cylinders,
because they are irrational \cite{CG,IM}.
On the~other hand, every smooth rational projective surface contains a~cylinder (see, for example,
\cite[Proposition~3.13]{KPZ11}).
In~particular, all smooth del Pezzo surfaces (two-dimensional Fano varieties) are also cylindrical.
Therefore,~one can expect that all rational Fano varieties are cylindrical.
However, the~following example shows that this is not the~case:

\begin{example}
\label{example:dP1-1-D4}
Let $X$ be a~hypersurface of degree $6$ in $\PP(1,1,2,3)$ that is given by
$$
x_3^2=x_2(x_2+x_0x_1)(x_2+\lambda x_0x_1),
$$
for some $\lambda\in\Bbbk\setminus\{0,1\}$, where $x_0$, $x_1$, $x_2$ and $x_3$
are  coordinates of weights $1$, $1$, $2$ and $3$ respectively.
Then $X$ is a~del Pezzo surface that has exactly two Du Val singular points of type $\mathrm{D}_4$,
it is rational,  has Picard number $1$, and does not contain cylinders by \cite[Theorem~1.5]{CPW16b}, see also Theorem~\ref{theorem:dP-du-Val-cylinders} of the present survey and its proof.
\end{example}

The surface in Example~\ref{example:dP1-1-D4} is singular.
There are other examples of singular non-cylindrical rational surfaces
(see Examples~\ref{example:dP1-2A32A1-4A2}, \ref{example:Ueno-Campana}, \ref{example:Kollar} below).
What about \emph{smooth} rational varieties?

\begin{question}
\label{question:main}
Does every smooth rational Fano variety contain a~cylinder?
\end{question}

We do not know the~answer to this question even in dimension three despite the~fact
that smooth three-dimensional Fano varieties (Fano threefolds) are completely classified and well studied \cite{IP99}.
Nevertheless, we believe that the~answer to Question~\ref{question:main} is negative (see Conjectures~\ref{conjecture:g-9-10} and \ref{conjecture:g-7}).
In fact, we do not know the~answer to the~following generalization of Question~\ref{question:main}:

\begin{question}[\cite{CDP18}]
\label{question:cylinders}
Is it true that any smooth rational variety is cylindrical?
\end{question}

A cylindrical variety $X$ is birationally equivalent to a~product $\A^1\times Z$.
Thus, if $X$ is rationally connected, then $Z$ is also rationally connected.
In particular, if $X$ is a cylindrical Fano threefold with Kawamata log terminal singularities, then $X$ must be rational \cite{Zhang2006}.
Moreover, we have

\begin{proposition}
\label{prop:cyl-ruled}
Let $X$ be a~cylindrical smooth Fano variety with $\uprho(X)=1$.
Then $X$ is birational to the~product $Y\times \A^2$ for some rationally connected variety $Y$.
\end{proposition}

\begin{proof}
Let $U$ be a cylinder in the~Fano variety $X$. Then $U\cong Z\times \A^1$ for some affine variety $Z$.
Let $\overline{Z}$ be a~projective completion of the~variety $Z$. Consider the~natural completion
$$
\overline Z\times \A^1\subset\overline Z \times \PP^1,
$$
let $D=(\overline Z \times \PP^1)\setminus (\overline Z\times \A^1)$,
and let $\psi\colon\overline Z \times \PP^1\dasharrow X$ be the~birational map induced by the~open embedding $Z\times \A^1\subset X$.
Since $\uprho(X)=1$ by assumption, the~divisor $D$ must be $\psi$-exceptional,
which implies that $D$ is birational to $Y\times \A^1$ for some variety $Y$.
Then $X$ is birational to $Y\times\A^{2}$.
Since $X$ is rationally connected (see \cite{Kollar-Miyaoka-Mori-1992b,Campana1992}), the~variety $Y$ is rationally connected as well.
\end{proof}

\begin{corollary}
\label{corollary:Fano-fourfolds}
Let $X$ be a~cylindrical smooth Fano fourfold with $\uprho(X)=1$. Then $X$ is rational.
\end{corollary}

However, we do not know cylindricity of many rational smooth Fano fourfolds of Picard rank~$1$.
For instance, we do not know whether any smooth rational cubic fourfold in $\mathbb{P}^5$ is cylindrical or not (see Question~\ref{question:cubic-fourfolds} and Remark~\ref{remark:Fermat-cubics}).
Keeping in mind Corollary~\ref{corollary:Fano-fourfolds}, we ask

\begin{question}
\label{question:cylindrical-rational}
Is it true that all~cylindrical smooth Fano varieties of Picard rank one are rational?
\end{question}

In the~paper \cite{Gromov}, Gromov asked whether every smooth rational variety is uniformly rational?
Recall from  \cite{Bogomolov-Bohning,LiendoPetitjean,Petitjean} that a smooth rational variety is said to be \emph{uniformly rational} if its every point has a Zariski open neighborhood isomorphic to an open subset of the~space $\A^n$ (cf.~\cite{Bodnar2008}).
Similarly, a~smooth cylindrical projective variety is said to be \textit{uniformly cylindrical} if its every point is contained in a~(Zariski open) cylinder (see~Section~\ref{subsection:flexibility} for the~motivation and examples).
It is easy to see that all smooth rational surfaces are uniformly rational and uniformly cylindrical.
On~the~other hand, we do not know the~answer to Gromov's question for varieties of higher dimensions,
and we do not know the~answer to

\begin{question}
\label{question:Gromov-cylinders}
Is it true that any cylindrical smooth projective variety is uniformly cylindrical?
\end{question}

In  Section~\ref{section:Fanos-cylinders}, we will present several cylindrical smooth Fano threefolds whose Picard groups are generated by their anticanonical divisors.
We do not know such examples in any other dimension.
The~counter-examples to \cite[Conjecture~5.1]{Pukhlikov2004} found in~\cite{Castravet}
made us believe that such examples should exist in any dimension $\geqslant 4$.
Therefore, we pose

\begin{problem}
\label{problem:cylindrical-Fano-index-one}
Find a~cylindrical smooth Fano variety of~dimension $\geqslant 4$ whose Picard group is generated by its anticanonical divisor.
\end{problem}

One can also define cylindricity and uniform cylindricity for affine varieties in the~same way we did this for projective varieties.
Note that \cite[Definition~3.4]{KPZ11} asks that the~cylinder should be principal, that is, its complement should be a principal divisor, which is not automatic.

\begin{remark}[{cf.~Question~\ref{question:Gromov-cylinders}}]
\label{remark:Koras-Russell}
There are cylindrical smooth affine varieties that are not uniformly cylindrical.
Indeed, let $V$ be the~Koras-Russell cubic threefold in $\A^4$ that is given by
$$
x_1+x_1^2x_2+x_3^2+x_4^3=0,
$$
where $x_1$, $x_2$, $x_3$ and $x_4$ are coordinates on $\A^4$.
Then $V$ is a cylindrical smooth affine variety \cite{KorasRussell}. Since $\Pic(V)=0$,
every cylinder in $V$ is principal, that is, the complement of a principal divisor.
It follows from \cite[Corollary~4.5]{DMP2010}
that $(0,0,0,0)$ is fixed by any element of $\mathrm{Aut}(V)$,
which implies that this point is not contained in any cylinder in $V$.
Indeed, otherwise the origin would be moved by a suitable $\Ga$-action on $V$, cf. Theorem~\ref{thm:affine-cylindricity} below.
\end{remark}

Like in the~projective case, every cylindrical affine variety $X$ has negative log Kodaira dimension.
Moreover, a~smooth affine surface contains a~cylinder if and only if its log Kodaira dimension is negative \cite[Ch.~2, Theorem~2.1.1]{Miyanishi2000}, cf.~\cite{MS80}.
However, this is no longer true in higher dimensions:

\begin{example}
\label{example:DK15}
Let $X$ be a~smooth hypersurface in $\PP^n$ of degree $n\geqslant 3$.
 Then $\PP^n\setminus X$ is a~smooth affine $n$-fold of negative Kodaira dimension that does not contain cylinders \cite{CDP18, DK15}.
\end{example}

The problem of existence of cylinders in projective varieties is closely related to
unipotent actions on the~affine cones over them.
To illustrate this link, consider the~following

\begin{question}[{\cite[Question~2.22]{FZ03}}]
\label{question:FZ}
Let $V$ be the~affine cone in $\A^4$ over the~Fermat cubic surface, which is given by
$$
x_1^3+x_2^3+x_3^3+x_4^3=0,
$$
where $x_1$, $x_2$, $x_3$ and $x_4$ are coordinates on $\A^4$.
Does $V$ admit an~effective $\Ga$-action?
\end{question}

The answer to this question is negative \cite{CPW16a}, see also \cite[Theorem~7.1]{DFMJ17} for a~purely algebraic proof.
The geometric proof of this fact is based on the~following result:

\begin{theorem}[{\cite[Proposition~3.1.5]{KPZ11}}]
\label{thm:affine-cylindricity}
An affine variety $V$ admits an~effective $\Ga$-action if and only if $V$ contains
a principal effective divisor $D$ such that $V\setminus \Supp(D)$ is a~cylinder.
\end{theorem}

Using this criterion, we can formulate the~corresponding criterion for projective varieties,
which requires the~following refined notion of cylindricity:

\begin{definition}
\label{definition:polar-cylinder}
Let $X$ be a~projective normal variety that contains a~Zariski open cylinder $U$,
and let $H$ be an~ample $\mathbb{Q}$-Cartier $\mathbb{Q}$-divisor on $X$.
The cylinder $U$ is said to be \emph{$H$-polar} if
$$
U=X\setminus \Supp(D)
$$
for some effective $\mathbb{Q}$-divisor $D$ on the~variety $X$ such that $D\sim_{\mathbb{Q}} H$.
\end{definition}

Now, we are in position to state the~following criterion discovered in \cite{KPZ13}.

\begin{theorem}
\label{theorem:criterion}
Let $X$ be a~projective normal variety, let $H$ be an~ample Cartier divisor on it, let
$$
V=\Spec\Bigg(\bigoplus_{n\geqslant 0} H^0\Big(\mathcal{O}_X\big(nH\big)\Big)\Bigg).
$$
Then $V$ admits an~effective $\Ga$-action $\iff$ $X$ contains an~$H$-polar cylinder.
\end{theorem}

\begin{corollary}
\label{corollary:rational-surfaces}
Let $X$ be a~smooth rational projective surface. Then there is an~embedding~\mbox{$X\hookrightarrow\PP^n$} such that the~affine cone in $\A^{n+1}$ over $X$ admits an~effective $\Ga$-action.
\end{corollary}

\begin{corollary}
\label{corollary:criterion}
Let $X$ be a~projective normal variety in $\PP^n$ whose divisor class group is of rank~$1$.
Then the~affine cone in $\A^{n+1}$ over $X$ admits an~effective $\Ga$-action $\iff$ $X$ is cylindrical.
\end{corollary}

\begin{remark}
\label{remark:Gorenstein}
Let $X$, $H$ and $V$ be as in Theorem~\ref{theorem:criterion}.
If $V$ is $\QQ$-Gorenstein and admits an~effective action of the~additive group $\Ga$,
then $X$ is a Fano variety and $H\sim_{\QQ} -\lambda K_X$ for some $\lambda\in\mathbb{Q}_{>0}$ \cite[(3.18)]{KPZ11}.
This explains our primary interest in the~affine cones over Fano varieties.
\end{remark}

The problem of existence of an~effective $\Ga$-action on affine varieties is interesting on its own.
If~an affine variety $V$ admits a non-trivial $\Ga$-action and $\mathrm{dim}(V)\geqslant 2$,
then $\mathrm{Aut}(V)$ is infinite dimensional and non-algebraic \cite{Fr06}.
On the~other hand, if it does not admit non-trivial $\Ga$-actions,
then $\mathrm{Aut}(V)$ contains a unique maximal torus~$\mathbb{T}$,
and $\mathrm{Aut}(V)$ is an extension of its centralizer by a discrete subgroup in $\GL_r(\mathbb{Z})$ (see \cite{AG17} for details).

\begin{example}
\label{example:cones-in-A3}
Let $V$ be the~Pham--Brieskorn surface in $\A^3$, which is given by
$$
x_1^{a_1}+x_2^{a_2}+x_3^{a_3}=0,
$$
where $a_1$, $a_2$, $a_3$ are integers such that $2\leqslant a_1\leqslant a_2\leqslant a_3$, and $x_1$, $x_2$, $x_3$ are coordinates on~$\A^3$.
By~\cite[Lemma~4]{KZ00}, the~affine variety $V$ admits an~effective $\Ga$-action $\iff$ $a_1=a_2=2$.
\end{example}

Affine varieties that do not admit effective $\Ga$-actions are often called \emph{rigid} \cite{Ar18a,Ar18b,BorovikGaifullin,Ga19,Fr06,KZ00}.
Applying \cite[Corollary~2.1.4]{KPZ11} and \cite[Proposition~4.1]{AG17} to affine cones over projective varieties,
we~obtain the~following result:

\begin{theorem}
\label{theorem:Arzhantsev}
Let $V$ be the~affine cone in $\A^{n+1}$ over a~projectively normal subvariety $X\subset \PP^n$.
Suppose that $V$ is rigid and $\mathrm{Aut}(X)$ is finite.
Then there exists an~exact sequence of groups
$$
1\longrightarrow\Gm\longrightarrow\Aut(V)\longrightarrow\mathrm{Aut}(X),
$$
so that $\mathrm{Aut}(V)$ is a~finite extension of the~torus $\Gm$ by a~finite subgroup in $\Aut(X)$.
\end{theorem}

In particular, combining this result with the~negative answer to Question~\ref{question:FZ}, we obtain

\begin{corollary}
\label{corollary:Fermat-cubic}
If $V$ is the~affine hypersurface from Question~\ref{question:FZ}, then  $\Aut(V)=\Gm\times(\mumu_3^3\rtimes \mathfrak{S}_4)$.
\end{corollary}

Both Question~\ref{question:FZ} and Example~\ref{example:cones-in-A3}
are very special cases of the~following old conjecture,
which has been confirmed in many cases (see \cite{ChitayatDaigle}~and~Remark~\ref{remark:Fermat-cubics}).

\begin{conjecture}[{\cite{KZ00,FZ03}}]
\label{conjecture:Pham--Brieskorn}
Let $V$ the~Pham--Brieskorn hypersurface in $\A^n$ with $n\geqslant 3$ given by
$$
x_1^{a_1}+x_2^{a_2}+\cdots+x_n^{a_n}=0,
$$
where $a_1,\ldots,a_n$ are integers such that $2\leqslant a_1\leqslant \cdots\leqslant a_n$, and $x_0,x_1,\ldots,x_n$ are coordinates on~$\A^n$.
Suppose that $a_2\geqslant 3$. Then the~affine hypersurface $V$ is rigid.
\end{conjecture}

In fact, using Theorem~\ref{theorem:criterion}, we can restate Question~\ref{question:FZ} as follows:

\begin{question}
\label{question:FZ-geometric}
Let $X$ be the~Fermat cubic surface. Does $X$ contain $(-K_X)$-polar cylinder?
\end{question}

As we already mentioned, this question has a~negative answer.
Moreover, we will see later that the~answer is also negative
for any smooth cubic surface  (cf. Theorem~\ref{theorem:dP-du-Val}).
This brings~us~to

\begin{problem}
\label{problem:polar-cylinders}
Describe Fano varieties that do not contain anticanonical polar cylinders.
\end{problem}

This problem has been solved for del Pezzo surfaces with Du Val singularities in \mbox{\cite{KPZ11,CPW16a,CPW16b}.}
However, it is still open for smooth Fano threefolds and singular del Pezzo surfaces with quotient singularities.
For Fano varieties whose divisor class groups is of rank $1$, Problem~\ref{problem:polar-cylinders} is~equivalent to the~cylindricity problem (the problem of existence of cylinders).

\begin{remark}
\label{remark:cylinders-any-field}
One can consider Problem~\ref{problem:polar-cylinders} for Fano varieties defined over an arbitrary possibly algebraically non-closed field.
In Section~\ref{subsection:MFS}, we will give a motivation for doing this.
\end{remark}

Let us present one obstruction for the~existence of anticanonical polar cylinders in Fano varieties.
Recall from \cite{Ti87,CheltsovShramovUMN} that the~$\alpha$-invariant of Tian of the~Fano variety $X$ is the~number
$$
\alpha(X)=\mathrm{sup}\left\{\lambda\in\mathbb{Q}\ \left|\ \aligned
&\text{the log pair}\ \left(X,\lambda D\right)\ \text{is log canonical}\\
&\text{for any effective $\mathbb{Q}$-divisor}\ D\sim_{\mathbb{Q}} -K_{X}\\
\endaligned\right.\right\}.
$$
This number plays an important role in K-stability of Fano varieties, since $X$ is K-stable if
$$
\alpha(X)>\frac{\mathrm{dim}(X)}{\mathrm{dim}(X)+1}.
$$
On the~other hand, we have the~following result.

\begin{theorem}
\label{theorem:alpha}
Let $X$ be a Fano variety that has at most Kawamata log terminal singularities.
If $\alpha(X)\geqslant 1$, then $X$ does not contain $(-K_X)$-polar cylinders.
\end{theorem}

\begin{proof}
Suppose  $X$ contains a $(-K_X)$-polar cylinder. Then $U\cong Z\times \A^1$ for an~affine variety $Z$, and
$$
U=X\setminus \Supp(D)
$$
for some effective $\mathbb{Q}$-divisor $D$ on $X$ such that $D\sim_{\mathbb{Q}} -K_X$.
Arguing as in the~proof Corollary~\ref{corollary:alpha-surfaces},
we see that the~log pair $(X,D)$ is not log canonical in this case, so that $\alpha(X)<1$.
\end{proof}

Let us show how to use this obstruction.

\begin{example}
\label{example:dP1-2A32A1-4A2}
Let $X$ be a del Pezzo surface with Du Val singularities of degree $K_X^2=1$
such that one of the~following two conditions holds:
\begin{enumerate}
\item either $X$ has $2$ singular points of type $\mathrm{A}_3$ and $2$ singular points of type $\mathrm{A}_1$;
\item or the~surface $X$ has $4$ singular points of type $\mathrm{A}_2$.
\end{enumerate}
By \cite[Theorem~1.2]{Ye2002}, the~surface $X$ exists, and it is uniquely determined by its singularities.
Moreover, it follows from \cite[Table~4.1]{Ye2002} that the~pencil $|-K_X|$ contains exactly $4$ singular fibers.
They are singular fibers of types $\mathrm{I}_4$ and $\mathrm{I}_2$ (in the~first case) or of types $\mathrm{I}_2$ (in the~second~case).
This~gives $\alpha(X)=1$ by \cite[Theorem~1.25]{ChK14},
so that $X$  contains no anticanonical polar cylinders.
Since~the group $\mathrm{Cl}(X)$ is of rank $1$, the~surface $X$ contains no cylinders at all.
\end{example}

\begin{remark}
\label{remark:MiyanishiKeelMcKernan}
Implicitly, Theorem~\ref{theorem:alpha} has been already used by many people for quite some~time.
For~instance, Miyanishi conjectured in \cite{GZh94} that the~smooth locus of a~del Pezzo surface with quotient
singularities and Picard rank~$1$ admits a~finite unramified covering  that contains~a~cylinder.
It turned out to be wrong.
Namely, in \cite[Example 21.3.3]{KeelMcKernan}, Keel and McKernan have constructed
a~singular del Pezzo surface $X$ with quotient singularities such that $\uprho(X)=1$ and $\alpha(X)\geqslant 1$,
but its~smooth locus has trivial algebraic fundamental group.
Thus, its smooth locus does not admit non-trivial unramified coverings,
and $X$ does not contain cylinders by Theorem~\ref{theorem:alpha}.
\end{remark}

Using Theorem~\ref{theorem:alpha}, we can create many rational Fano varieties without anticanonical polar cylinders.
Indeed, if $X$ and $Y$ are Fano varieties that have Kawamata log terminal singularities,
then it follows from  \cite[Lemma~2.29]{CheltsovShramovUMN} and \cite[Proposition~8.11]{KovacPatakfalvi} that
$$
\alpha\big(X\times Y\big)=\mathrm{min}\big\{\alpha(X),\alpha(Y)\big\}.
$$
Thus, if $S$ is a general smooth del Pezzo surface with $K_S^2=1$, then $\alpha(S)=1$ by \cite[Theorem~1.7]{CheltsovGAFA},
which implies that we also have  $\alpha(X)=1$ for the~$2n$-dimensional smooth Fano variety
$$
X=\underbrace{S\times S\times \cdots\times S}_{n\ \text{times}},
$$
so that $X$ does not contain $(-K_X)$-polar cylinders, but $X$ is cylindrical, because $S$ is cylindrical.
We can construct many similar examples using \cite{CheltsovParkWon2014,CheltsovShramovPark2010,CheltsovShramovZoo,ChK14,Pukhlikov2005}.

\begin{example}
\label{example:dP1-double-cover}
Let $S$ be a general smooth del Pezzo surface with $K_S^2=1$,
and let $Y$ be a general smooth hypersurface in $\mathbb{P}(1^{n+1},n)$ of degree $2n$ for $n\geqslant 3$.
Then $\alpha(S)=1$ by \cite[Theorem~1.7]{CheltsovGAFA},
and $\alpha(Y)=1$ by \cite[Theorem~2]{Pukhlikov2005} (see also \cite{CheltsovParkWon2014}).
Let $X=S\times Y$. Then $\mathrm{dim}(X)=2+n\geqslant 5$ and
$$
\alpha\big(X\big)=\mathrm{min}\big\{\alpha(S),\alpha(Y)\big\}=1,
$$
so that $X$ contains no $(-K_X)$-polar cylinder by Theorem~\ref{theorem:alpha}. But $X$ is cylindrical.
\end{example}

Surprisingly, we do not know a single example of a cylindrical smooth Fano threefold that
contains no anticanonical polar cylinder (cf. Examples~\ref{example:blow-up-dP-n}, \ref{example:blow-up-P3-two-cubic}, \ref{example:blow-up-P3-g-3-d-6} and \ref{example:blow-up-Q-three-quadrics}).

\begin{problem}
\label{problem:anticanonical-polar-cylinders-threefolds}
Find a cylindrical smooth Fano threefold without anticanonical polar cylinder.
\end{problem}

Note that there are Fano varieties without cylinders whose $\alpha$-invariant of Tian is smaller than $1$.
For instance, if $X$ is the~del Pezzo surface from Example~\ref{example:dP1-1-D4}, then $\alpha(X)=\frac{1}{2}$ by \cite[Theorem~1.25]{ChK14}.
On the~other hand, this surface does not contain cylinders \cite{CPW16b}. Note that it is K-polystable \cite{OdakaSpottiSun}.
Surprisingly, all known K-unstable Fano varieties are also cylindrical.

\begin{example}[{\cite{Fujita1980,Fujita1981,Fujita1984,IP99}}]
\label{example:V5}
Let $X$ be a smooth Fano variety of dimension $n\geqslant 2$ such that
$$
-K_{X}\sim (n-1)H,
$$
where $H$ is an ample divisor such that $H^n=5$. Then $n\in\{2,3,4,5,6\}$, and $X$ is unique for each~$n$.
The divisor $H$ is very ample, and the~linear system $|H|$ gives an embedding $X\hookrightarrow\mathbb{P}^{n+3}$
such that the~image is a section of the~Grassmannian $\mathrm{Gr}(2,5)\subset\mathbb{P}^9$ by a linear subspace of dimension $3+n$.
Moreover, if $n\ne 2$, then $\mathrm{Pic}(X)=\mathbb{Z}[H]$. Furthermore, the~following assertions hold.
\begin{itemize}
\item The variety $X$ contains a~Zariski open subset isomorphic to $\A^n$, so that it is cylindrical.
If~$n\ne 5$, this follows from Example~\ref{example:Grassmannian} and Theorems~\ref{theorem:V5-A3} and \ref{theorem:V5-A4} (see also \cite{Fujita1981,  PZ16}).
If $n=5$, then $X$ contains a plane $\Pi$ such that there exists the~following Sarkisov link:
$$
\xymatrix{
&\widetilde{X}\ar[dl]_{\alpha}\ar[dr]^{\beta}\\
X&&\mathbb{P}^5}
$$
where $\alpha$ is the~blowup of the~plane $\Pi$, and $\beta$ is the~blowup of a smooth cubic scroll in~$\PP^5$.
This easily implies that $X$ contains a~Zariski open subset isomorphic to $\A^5$.

\item If $n\in\{2,3,6\}$, then $X$ is known to be K-polystable (see, for example, \cite{CheltsovGAFA,CheltsovShramovDP,ParkWon2018,Ti90,Zhuang}).
On the~other hand, if $n\in\{4,5\}$, then $X$ is K-unstable by \cite{FujitaV5}.
\end{itemize}
\end{example}

Keeping in mind Theorem~\ref{theorem:alpha} and  examples of K-stable Fano varieties without anticanonical polar cylinders (for example, smooth del Pezzo surfaces of degree $1$, $2$ and $3$),
we pose

\begin{conjecture}
\label{conjecture:K-stability}
Let $X$ be a Fano variety that has at most Kawamata log terminal singularities.
If $X$ does not contain $(-K_X)$-polar cylinders, then $X$ is K-polystable.
\end{conjecture}

For a~projective variety $X$, consider the~following subset of the~cone of ample $\mathbb{Q}$-divisors on~$X$:
$$
\mathrm{Amp}^{cyl}\big(X\big)=\Big\{ H\in \mathrm{Amp}(X)\ \big\vert\ \text{there is an~$H$-polar cylinder on $X$}\Big\}.
$$
Let us call it the~\emph{cone of cylindrical ample divisors} of the~variety $X$.
We have seen in Examples~\ref{example:dP1-1-D4} that $\mathrm{Amp}^{cyl}(X)$ can be empty even if $X$ is a Fano variety.
Thus, we can enhance Problem~\ref{problem:polar-cylinders} by

\begin{problem}
\label{problem:cylindrical-cone}
For a~given Fano variety $X$, describe the~cone $\mathrm{Amp}^{cyl}(X)$.
\end{problem}

This problem is not yet solved even for smooth del Pezzo surfaces.
However, we know the~answer for many of them (see \cite{CPW17}).
Namely, if $X$ is a~smooth del Pezzo surface such that $K_X^2\geqslant 4$, then
$$
\mathrm{Amp}^{cyl}(X)=\mathrm{Amp}(X).
$$
On the~other hand, if $K_X^2\leqslant 3$, then $-K_X\not\in\mathrm{Amp}^{cyl}(X)$.
This gives an evidence for

\begin{conjecture}
\label{conjecture:anticanonical-cylinder}
If $X$ is a~Fano variety, then $-K_{X}\in\mathrm{Amp}^{cyl}(X)\iff\mathrm{Amp}^{cyl}(X)=\mathrm{Amp}(X)$.
\end{conjecture}

Let us describe the~structure of this survey.
In Section~\ref{section:del-Pezzo} we review results about polar cylinders in rational surfaces.
In Section~\ref{section:Fanos-cylinders}, we describe results about cylinders in smooth Fano threefolds, smooth Fano fourfolds, and del Pezzo fibrations.
In Section~\ref{section:related}, we survey results on three topics that are closely related to
the main topic of this survey: flexibility of affine varieties
with a special accent on the~flexibility of affine cones over Fano varieties,
cylinders in the~complements to hypersurfaces in weighted projective spaces, and compactifications of $\CC^n$.
Finally, in Appendix~\ref{section:log-pairs}, we present some results about singularities of two-dimensional log pairs,
which are used in Section~\ref{section:del-Pezzo} to prove the~absence of polar cylinders in some del Pezzo surfaces.

\subsection*{Notations.}
Throughout this paper, we will use the~following notation:
\begin{itemize}
\item $\mumu_n$ is a~cyclic subgroup of order $n$;
\item $\Ga$ is a~one-dimensional unipotent additive group;
\item $\Gm$ is a~one-dimensional algebraic torus;
\item $\FF_n$ is the~Hirzebruch surface;
\item $\PP^n$ is the~$n$-dimensional projective  space over $\Bbbk$;
\item $\A^n$ is the~$n$-dimensional affine space over $\Bbbk$;
\item $\PP(a_1,\dots, a_n)$ is the~weighted projective space;
\item for a~variety $X$, we denote by $\uprho(X)$ the~rank of its Picard group.
\end{itemize}

\subsection*{Acknowledgments.}
This work was supported by the~Royal Society grant No. IES\textbackslash R1\textbackslash 180205 and by the~Russian Academic Excellence Project 5-100.
The second author has been supported by IBS-R003-D1, Institute for Basic Science in Korea.

The authors would like to thank Adrien Dubouloz and Sasha Perepechko for useful comments.

\section{Cylinders in del Pezzo surfaces}
\label{section:del-Pezzo}
In this section,  we review results about cylinders in del Pezzo surfaces.
A~\emph{del Pezzo surface} means here a~two-dimensional Fano variety with at most quotient singularities.
Recall that a~smooth del Pezzo surface is either $\mathbb{P}^1\times\mathbb{P}^1$,
or a~blowup of $\mathbb{P}^2$ in at most $8$ points such that
\begin{itemize}
\item at most $2$ points are contained in a~line;

\item at most $5$ points are contained in a~conic;

\item there is no singular cubic in $\mathbb{P}^2$ that contains $8$ points and is singular in one of them.
\end{itemize}
A Gorenstein  del Pezzo surface is a~del Pezzo surface whose anticanonical divisor is Cartier,
equivalently a~del Pezzo surface with only Du Val singularities.
Such surface is either a~quadric, or its minimal resolution of singularities
can be obtained by blowing up $\mathbb{P}^2$ in at most $8$ points such that
at most $3$ of them are contained in a~line, and at most $6$ of them are contained in a~conic.

First, let us go over basic facts about cylinders in rational surfaces.

\subsection{Cylinders in rational surfaces}
\label{subsection:rational-surfaces}
Observe that every smooth rational surface is cylindrical.
This immediately follows from the~fact that $\mathbb{P}^2$ contains a cylinder and the~following

\begin{lemma}
\label{lemma:Fn}
Let $C$ be an irreducible curve in $\mathbb{F}_n$ that is a section of the~natural projection  $\mathbb{F}_n\to\mathbb{P}^1$,
and let $F_1,\ldots,F_r$ be fibers of this projection, where $r\geqslant 1$.
Then $\mathbb{F}_n\setminus (C\cup F_1\cup\cdots F_r)$ is a cylinder.
\end{lemma}

\begin{proof}
Performing appropriate elementary birational transformations, we may assume that $C^2=0$, so that $n=0$.
In this case, the~required assertion is obvious.
\end{proof}

However, as we have seen already in Example~\ref{example:dP1-1-D4}, there are singular rational surfaces that contain no cylinders.
Let us explain how to find many such rational surfaces and provide an obstruction for the~existence of cylinders (see Remark~\ref{remark:obstruction} below),
which will be used in Section~\ref{subsection:del-Pezzo-no-cylinders} to show the~absence of anticanonical polar cylinders in smooth del Pezzo surfaces of degree $1$, $2$ and $3$.

Let $S$ be a~rational surface with quotient singularities
and  suppose that $S$ contains a~cylinder~$U$.
Then $U$ is a~Zariski open subset in $S$ such that $U\cong\mathbb{A}^1\times Z$ for some affine curve $Z$. We then have the~following commutative diagram
$$
\xymatrix{\mathbb{P}^1\times\mathbb{P}^1\ar[ddrr]^{\overline{p}_{2}}&\mathbb{A}^1\times\mathbb{P}^1\ar@{_{(}->}[l]\ar[ddr]^{p_{2}}&~\mathbb{A}^1\times Z\cong U\ar@{_{(}->}[l]\ar[d]^{p_Z}\ar@{^{(}->}[r] &S\ar@{-->}^{\psi}[ddl]& &\widetilde{S}\ar[ll]_{\pi}\ar[ddlll]^{\phi}& \\
&&Z\ar@{_{(}->}[d]&&&& \\
& &\mathbb{P}^1& & &&}
$$
where $p_Z$, $p_{2}$ and $\overline{p}_2$ are the~natural projections to the~second factors,
$\psi$ is the~ rational map induced by $p_Z$,
$\pi$ is a~birational morphism resolving the~indeterminacy of $\psi$ and $\phi$ is a~morphism.
By construction, a~general fiber of $\phi$ is $\mathbb{P}^1$.
Let $C_1,\ldots,C_n$ be the~irreducible curves in $S$ such that
$$
S\setminus U=\bigcup_{i=1}^{n}C_i.
$$
The~curves $C_1,\cdots,C_n$ generate the~divisor class group $\mathrm{Cl}(S)$ of the~surface $S$,
because $\mathrm{Cl}(U)=0$. In particular, one has
\begin{equation}\label{equation:KPZ-r}
\mathrm{rank}\,\mathrm{Cl}(S)\leqslant n.
\end{equation}

Let $E_1,\ldots, E_r$ be all exceptional curves of the~morphism $\pi$ (if any),
and let $\Gamma=\mathbb{P}^1\times\mathbb{P}^1\setminus\mathbb{A}^1\times\mathbb{P}^1$.
Denote by $\widetilde{C}_1,\ldots,\widetilde{C}_n$ and $\widetilde{\Gamma}$ the~proper transforms $\widetilde{S}$ of the~curves $C_1,\ldots,C_n$ and $\Gamma$, respectively.
Then $\widetilde{\Gamma}$ is a~section of the~conic bundle $\phi$, and $\widetilde{\Gamma}$ is one of the~curves $\widetilde{C}_1,\ldots,\widetilde{C}_n$ and $E_1,\ldots, E_r$.
Moreover, all other curves among $\widetilde{C}_1,\ldots,\widetilde{C}_n$ and $E_1,\ldots, E_r$ are components of some fibers of $\phi$.
Thus, we may assume that either $\widetilde{\Gamma}=\widetilde{C}_1$ or $\widetilde{\Gamma}=E_r$.
Then $\psi$ is a~morphism $\iff$ $\widetilde{\Gamma}=\widetilde{C}_1$.

Let $\lambda_1,\ldots,\lambda_n$ be arbitrary rational numbers, and let $D=\lambda_1C_1+\cdots+\lambda_nC_n$. Then
$$
K_{\widetilde{S}}+\sum_{i=1}^n\lambda_i\widetilde{C}_i+\sum_{i=1}^r\mu_iE_i\sim_{\mathbb{Q}}\pi^*\left(K_{S}+D\right)
$$
for some real numbers $\mu_1,\ldots,\mu_r$. Let $\widetilde{F}$ be a~general fiber of $\phi$.
Then $K_{\widetilde{S}}\cdot\widetilde{F}=-2$ by the~adjunction formula.
Put $F=\pi(\widetilde{F})$.
 If $\widetilde{\Gamma}=E_r$, then
$$
-2+\mu_r=\left(K_{\widetilde{S}}+\sum_{i=1}^n\lambda_i\widetilde{C}_i+\sum_{i=1}^r\mu_iE_i\right)\cdot\widetilde{F}=\pi^*\left(K_{S}+D\right)\cdot\widetilde{F}=\left(K_{S}+D\right)\cdot F
$$
Similarly, if $\widetilde{\Gamma}=C_1$, then
$$
-2+\lambda_1=\left(K_{\widetilde{S}}+\sum_{i=1}^n\lambda_i\widetilde{C}_i+\sum_{i=1}^r\mu_iE_i\right)\cdot\widetilde{F}=\pi^*\left(K_{S}+D\right)\cdot\widetilde{F}=\left(K_{S}+D\right)\cdot F.
$$
On the~other hand, if $K_S+D$ is pseudo-effective, then $\left(K_{S}+D\right)\cdot F\geqslant 0$.

\begin{remark}
\label{remark:obstruction}
We are therefore able to draw the~following conclusions:
\begin{itemize}
\item if $K_S+D$ is pseudo-effective, then $(S,D)$ is not log canonical;

\item if $K_S+D$ is pseudo-effective and $\lambda_i<2$ for each $i\in\{1,\ldots,n\}$, then $\psi$ is not a~morphism.
\end{itemize}
\end{remark}

\begin{corollary}
\label{corollary:K-nef}
A~rational surface with quotient singularities and pseudo-effective canonical divisor cannot contain any cylinder.
\end{corollary}

Now we  present two examples of rational singular surfaces with nef canonical divisors,
which do not contain cylinders by Corollary~\ref{corollary:K-nef}.
For more examples, see \cite{HwangKeum,LN13,LP07,OZh96,OguisoZhang,PPSh09a,PPSh09b,Wang}.

\begin{example}[cf.~\cite{OguisoTruong}]
\label{example:Ueno-Campana}
Let $E$ be the~Fermat cubic curve in $\mathbb{P}^2$.
Take $\sigma\in\mathrm{Aut}(E)$ of order~$6$~that fixes a point in $E$.
Let $S=E\times E/\langle\sigma\rangle$, where $\sigma$ acts on $E\times E$ diagonally.
Then $S$ is rational. Moreover, it has quotient singularities and $6K_S\sim 0$.
Then $S$ contains no cylinder  by Corollary~\ref{corollary:K-nef}.
\end{example}

\begin{example}[{\cite{Kollar}}]
\label{example:Kollar}
Let $a_0$, $a_1$, $a_2$, $a_3$,  $w_0$, $w_1$, $w_2$, $w_3$  be positive integers such that
\begin{itemize}
\item $a_0\geqslant 4$, $a_1\geqslant 4$, $a_2\geqslant 4$, $a_3\geqslant 4$;
\item $a_0w_0+w_1=a_1w_1+w_2=a_2w_2+w_3=a_3w_3+w_0$;
\item  $\gcd ( w_0,w_2)=1$,  $\gcd ( w_1,w_3)=1$.
\end{itemize}
From the~first condition above we obtain
$$
\left\{\aligned%
&w_0=a_1a_2a_3-a_2a_3+a_3-1,\\
&w_1=a_0a_2a_3-a_0a_3+a_0-1,\\
&w_2=a_0a_1a_3-a_0a_1+a_1-1,\\
&w_3=a_0a_1a_2-a_1a_2+a_2-1.\\
\endaligned
\right.
$$
Let $S$ be the~hypersurface in $\mathbb{P}(w_0,w_1,w_2,w_3)$ defined by the~following equation:
$$
x_0^{a_0}x_1+x_1^{a_1}x_2+x_2^{a_2}x_3+x_3^{a_3}x_0=0,
$$
where $x_0$, $x_1$, $x_2$ and $x_3$ are coordinates of weights $w_0$, $w_1$, $w_2$, $w_3$, respectively.
Then
$$
K_{S}=\mathcal{O}_S\big(a_0a_1a_2a_3-w_0-w_1-w_2-w_3-1\big)
$$
and $a_0a_1a_2a_3-w_0-w_1-w_2-w_3-1>0$, so that $K_{S}$ is ample.
But $S$ is rational by \cite[Theorem~39]{Kollar}.
By Corollary~\ref{corollary:K-nef}, the~surface $S$ cannot contain any cylinder.
\end{example}

We are mostly interested in cylinders in del Pezzo surfaces.
Applying our Remark~\ref{remark:obstruction} to them, we obtain the~following special case of Theorem~\ref{theorem:alpha},
which we already applied in Example~\ref{example:dP1-2A32A1-4A2}.

\begin{corollary}
\label{corollary:alpha-surfaces}
Suppose that $-K_S$ is ample, and $U$ is a $(-K_S)$-polar cylinder. Then $\alpha(S)<1$.
\end{corollary}

\begin{proof}
There exists an effective $\mathbb{Q}$-divisor $D^\prime$ on the~surface $S$ such that $D^\prime\sim_{\mathbb{Q}} -K_S$ and
$$
D^\prime=\sum_{i=1}^na_iC_i,
$$
for some positive rational numbers $a_1,\ldots,a_n$.
Let $D=D^\prime$. Then $K_S+D\sim_{\mathbb{Q}} 0$ is pseudo-effective,
so~that $(S,D)$ is not log canonical by Remark~\ref{remark:obstruction}, which implies that $\alpha(S)<1$.
\end{proof}

Now, we state main result of this section, which implies negative answer to Question~\ref{question:FZ}.

\begin{theorem}[{\cite{KPZ11,KPZ14a,CPW16a,CPW16b}}]
\label{theorem:dP-du-Val}
Let $S$ be a del Pezzo surface that has at most Du Val singularities.
Then $S$ does not contain $(-K_{S})$-polar cylinders exactly when
\begin{itemize}
\item $K_S^2=1$ and $S$ allows at most singular points of types $\mathrm{A}_1$, $\mathrm{A}_2$, $\mathrm{A}_3$, $\mathrm{D}_4$ if any;
\item $K_S^2=2$ and $S$ allows at most singular points of type $\mathrm{A}_1$ if any;
\item $K_S^2=3$ and $S$ is smooth.
\end{itemize}
\end{theorem}

\begin{corollary}
\label{corollary:dP-du-Val}
A smooth del Pezzo surface $S$ contains a $(-K_{S})$-polar cylinder $\iff$ $K_S^2\geqslant 4$.
\end{corollary}

In the~next two subsections, we will explain how to prove Theorem~\ref{theorem:dP-du-Val}.
Now let us use this result to find all del Pezzo surfaces with Du Val singularities that contain no cylinder.

\begin{theorem}[{\cite[Theorem~1.6]{Be17}}]
\label{theorem:dP-du-Val-cylinders}
Let $S$ be a del Pezzo surface that has Du Val singularities.
Then $S$ contains no cylinder $\iff$ it is one of the~surfaces described in Examples~\ref{example:dP1-1-D4} and \ref{example:dP1-2A32A1-4A2}.
\end{theorem}

\begin{proof}
If $S$ is one of the~surfaces from Examples~\ref{example:dP1-1-D4} and \ref{example:dP1-2A32A1-4A2},
then $\rho(S)=1$, so that it does not contain cylinders by Theorem~\ref{theorem:dP-du-Val}.
To prove the converse assume that $S$ contains no cylinder.
Let~us show that $S$ is one of the~singular del Pezzo surfaces described in Examples~\ref{example:dP1-1-D4} and \ref{example:dP1-2A32A1-4A2}.
If $\uprho(S)=1$, this follows from Theorem~\ref{theorem:dP-du-Val} and \cite[Theorem~1.2]{Ye2002}.

We may assume that $\uprho(S)\geqslant 2$. Let us seek for a contradiction.
Since every smooth rational surface contains a cylinder, we see that $S$ is singular.
Then $K_S^2\leqslant 2$ by Theorem~\ref{theorem:dP-du-Val}.

Let $\pi\colon S\to Y$ be the~contraction of an extremal ray of the~Mori cone $\overline{\mathrm{NE}}(S)$ of the~surface $S$.
Then it follows from \cite{Morrison-1985} that one of the~following  cases hold:
\begin{itemize}
\item either $\pi$ is a conic bundle, $Y=\mathbb{P}^1$ and $\uprho(S)=2$;
\item or $\pi$ is birational, $Y$ is a del Pezzo surface with Du Val singularities, $\uprho(Y)=\uprho(S)+1$;
the~morphism $\pi$ is a weighted blowup of a smooth point in $Y$ with weights $(1,k)$ for $k\geqslant 1$, and $K_Y^2=K_S^2+k$.
\end{itemize}

Suppose that $\pi$ is a conic bundle.
Then we have the~following commutative diagram:
$$
\xymatrix{
&\widetilde{S}\ar[dl]_{\alpha}\ar[dr]^{\beta} \\
S\ar[d]_{\pi}&&\mathbb{F}_n\ar[d]\\
\mathbb{P}^1\ar@{=}[rr]&&\mathbb{P}^1}
$$
where $\alpha$ is a minimal resolution of singularities, $\beta$ is a birational map,
and $\mathbb{F}_n\to\mathbb{P}^1$ is a natural projection.
On the~other hand, it follows from Tsen's theorem that $S$ contains a smooth irreducible curve $Z$ that is a section of the~conic bundle $\pi$.
Let $C$ be its proper transform on $\mathbb{F}_n$. Then
$$
S\setminus\Big(Z\cup T_1\cup\cdots\cup T_r\Big)\cong S\setminus\Big(C\cup F_1\cup\cdots\cup F_r\Big)
$$
where $T_1,\ldots,T_r$ are fibers of $\pi$ that contain singular points of the~surface $S$,
and $F_1,\ldots,F_r$ are fibers of the~projection $\mathbb{F}_n\to\mathbb{P}^1$
over the~points $\pi(T_1),\ldots,\pi(T_r)$, respectively.
Then $S$ contains a~cylinder by Lemma~\ref{lemma:Fn}, which is a contradiction.

We see that $\pi$ is birational. Let $E$ be the~$\pi$-exceptional curve.
If $Y$ contains a cylinder $U$, then it also contains a cylinder $U^\prime\subset U$ such that $\pi(E)\not\in U^\prime$,
so that its preimage in $S$ is a cylinder as well.
Thus, the~surface $Y$ does not contain cylinders.
Then $Y$ is singular and $K_Y^2\leqslant 2$ by Theorem~\ref{theorem:dP-du-Val}.

We see that $K_Y^2=2$ and $\pi$ is a blowup of a smooth point in $Y$.
If $\uprho(Y)\geqslant 2$, then we can apply the~same arguments to $Y$ to show that it contains a cylinder.
Hence, we conclude that $\uprho(Y)=1$.
On the~other hand, all singularities of the~surface $Y$ are ordinary double points by Theorem~\ref{theorem:dP-du-Val}.
We see that $K_Y^2=2$ and $Y$ has $7$ singular points of type $\mathrm{A}_1$.
But such a surface does not exist.
\end{proof}

Let us conclude this subsection by presenting few results about polar cylinders in arbitrary rational surfaces.
To do this, fix an ample $\mathbb{Q}$-divisor $H$ on the~surface $S$.
If $S$ contains an $H$-polar cylinder, we~say that $H$ is \emph{cylindrical}.
The cylindrical ample $\mathbb{Q}$-divisors on $S$ form a cone, which we denoted earlier by $\mathrm{Amp}^{cyl}(S)$.
To investigate this cone, consider the~following number:
$$
\mu_H=\mathrm{inf}\left\{\lambda\in\mathbb{R}_{>0}\ \Big|\ \text{the $\mathbb{R}$-divisor}\ K_{S}+\lambda H\ \text{is pseudo-effective}\right\}.
$$

\begin{remark}
\label{remark:Fujita-invariant}
The number $\mu_H$ is known as the~Fujita invariant of the~divisor $H$,
because it was implicitly used by Fujita in \cite{F87,F92,F96,F97}.
It plays an essential role in Manin's conjecture (see \cite{BM90,HTT}).
\end{remark}

Let $\Delta_{H}$ be the~smallest extremal  face   of the~Mori cone $\overline{\mathbb{NE}}(S)$ that contains the~divisor $K_{S}+\mu_H H$.
Put $r_H=\dim(\Delta_{H})$. Observe that $r_H=0$ if and only if $S$ is a del Pezzo surface and $\mu_H H\sim_{\mathbb{Q}}-K_S$.

\begin{theorem}[{\cite{CheltsovSbornik}}]
\label{theorem:Sbornik}
Suppose that $S$ is smooth, $r_H+K_S^2\leqslant 3$, and the~self-intersection of every smooth rational curve in $S$ is at least $-1$.
Then $S$ does not contain $H$-polar cylinders.
\end{theorem}

Note that if $S$ is smooth del Pezzo surface, then the~self-intersection of every smooth rational curve in $S$ is at least $-1$.
Moreover, it follows from \cite[Proposition~2.4]{deFernex} that this condition also holds if $S$ is obtained by blowing up $\mathbb{P}^2$ at any number of points in general position.

\begin{corollary}[{\cite{CPW17}}]
\label{corollary:del-Pezzo-3}
If $S$ is a smooth del Pezzo surface and $r_H+K_S^2\leqslant 3$, then $H\not\in\mathrm{Amp}^{cyl}(S)$.
\end{corollary}

On the~other hand, we have the~following complimentary result:

\begin{theorem}[{\cite{CPW17,MW18}}]
\label{theorem:Amp-cyl-del-Pezzo}
Suppose that $S$ is a smooth rational surface. If $K_S^2\geqslant 4$, then
$$
\mathrm{Amp}^{cyl}(S)=\mathrm{Amp}(S).
$$
If $K_S^2=3$ and $-K_S$ is not ample, then $\mathrm{Amp}^{cyl}(S)=\mathrm{Amp}(S)$.
If $K_S^2=3$ and $-K_S$ is ample, then
$$
\mathrm{Amp}^{cyl}(S)=\mathrm{Amp}(S)\setminus\mathbb{Q}_{>0}[-K_S].
$$
\end{theorem}

If $S$ is a smooth rational surface and $K_S^2\leqslant 2$,
then  $\mathrm{Amp}^{cyl}(S)$ is poorly understood (see \cite{CPW17}).

\subsection{Absence of polar cylinders}
\label{subsection:del-Pezzo-no-cylinders}
Now, we show that smooth del Pezzo  surfaces of degree~$\leqslant 3$ does not contain any anticanonical polar cylinders,
which is one way implication of Corollary~\ref{corollary:dP-du-Val}.
For~singular del Pezzo surfaces of degree $\leqslant 2$ with types of singular points listed in~Theorem~\ref{theorem:dP-du-Val},
the~same implication can be verified in a~similar way (see \cite{CPW16b} for the~details).

Let $S$ be a smooth del Pezzo surface of degree $K_S^2=d\leqslant 3$,
and let $D$ be an~effective $\mathbb{Q}$-divisor on the~surface $S$, i.e., we have
$$
D=\sum_{i=1}^ra_iC_i,
$$
where every $C_i$ is an~irreducible curve on $S$, and every $a_i$ is a~non-negative rational number.
Suppose that $D\sim_{\mathbb{Q}} -K_S$.
If  $d\in\{2,3\}$, then each $a_i$ does not exceed $1$ by Lemmas~\ref{lemma:Pukhlikov} and~\ref{lemma:double-plane}.
Similarly, if $d=1$, we have
$$
1=d=K_S^2=D\cdot (-K_S)=\sum_{i=1}^ra_iC_i\cdot(-K_S)\geqslant a_iC_i\cdot(-K_S),
$$
which immediately implies that $a_i\leqslant 1$ for each $i$.

\begin{theorem}
\label{theorem:del-Pezzo-degree-1-2-3}
Let $P$ be a point in $S$. Suppose that the~log pair $(S,D)$ is not log canonical at $P$.
Then there exists a~curve $T\in |-K_{S}|$ such that
\begin{itemize}
 \item the~curve $T$ is singular at $P$;
\item the~log pair $(S,T)$ is not log canonical at  $P$;
\item $\mathrm{Supp}(T)\subseteq\mathrm{Supp}(D)$.
\end{itemize}
\end{theorem}
\begin{proof}
We consider the~cases $d=1$, $d=2$ and $d=3$ separately.
See the~proof of \cite[Theorem~1.12]{CPW16a} for an~alternative proof in the~case $d=3$.

Suppose that $K_S^2=1$.
Let $C$ be a~curve in $|-K_S|$ that passes through $P$. Then $C$ is irreducible.
If $C$ is not contained in the~support of $D$, then it follows from  Lemma~\ref{lemma:Skoda} that
$$
1=d\geqslant K_S^2=D\cdot C\geqslant\mathrm{mult}_P(D)>1.
$$
This shows that $C\subset\mathrm{Supp}(D)$.
If $(S,C)$ is not log canonical at $P$, then we can put $T=C$ and we are done.
Thus, we may assume that $(S,C)$ is log canonical at~$P$.
Then Remark~\ref{remark:convexity} implies the~existence
of an~effective $\mathbb{Q}$-divisor $D^\prime$ such that $D^\prime\sim_{\mathbb{Q}} -K_S$,
the curve $C$ is not contained in the~support of $D^\prime$, and $(S,D^\prime)$ is not log canonical at $P$.
Now Lemma~\ref{lemma:Skoda} implies that
$$
1=d\geqslant K_S^2=D^\prime\cdot C\geqslant\mathrm{mult}_P(D^\prime)>1,
$$
which is absurd.

Now, we suppose that $K_S^2=2$.
In this case there exists a~double cover $\tau\colon S\to\mathbb{P}^2$ branched over a~smooth quartic curve $C$.
Moreover, we have
$$
D\sim_{\mathbb{Q}} -K_S\sim \tau^*(L),
$$
where $L$ is a~line in $\mathbb{P}^2$.
By Lemma~\ref{lemma:double-plane}, we have $\tau(P)\in C$.
Now let us choose $L$ to be the~tangent line to $C$ at the~point $\tau(P)$,
and let $R$ be the~curve in $|-K_S|$ such that $\tau(R)=L$. Then~$\mathrm{mult}_P(R)=2$.
If $R$ is irreducible and is not contained in the~support of $D$, then Lemma~\ref{lemma:Skoda} gives
$$
2=d\geqslant K_S^2=D\cdot R\geqslant\mathrm{mult}_P(D)\mathrm{mult}_P(R)\geqslant 2\mathrm{mult}_P(D)>2.
$$
Note that either $R$ is irreducible or $R$ consists of two $(-1)$-curves that both pass through $P$.
Therefore, if one component of the~curve $R$ is not contained in the~support of the~divisor $D$,
then~we obtain a~contradiction in a~similar way by intersecting $D$ with this irreducible component.
Thus,~we may assume that all irreducible component of the~curve $R$ are contained in $\mathrm{Supp}(D)$.
Now we can use  Remark~\ref{remark:convexity} as in the~case $d=1$ to conclude that $(S,R)$ is not
log canonical at~$P$. Hence, we can let $T=R$.

Finally, we suppose that $K_S^2=3$.
Then $S$ is a~smooth cubic surface in $\mathbb{P}^3$,
and $-K_S$ is rationally equivalent to its hyperplane section.
Let $T_P$ be the~intersection of the~surface $S$ with the~hyperplane that is tangent to $S$ at the~point $P$.
Then $T_P$ is a~reduced cubic curve that is singular at $P$.
If~$(S,T_P)$~is not log canonical at $P$ and $\mathrm{Supp}(T_P)\subseteq\mathrm{Supp}(D)$,
we can let $T=T_P$ and we are done.
Therefore, we may assume that at least one of the~following two conditions hold:
\begin{enumerate}
\item the~log pair $(S,T_P)$~is log canonical at $P$;
\item $\mathrm{Supp}(D)$ does not contain at least one irreducible components of the~curve $T_P$.
\end{enumerate}
To obtain a~contradiction, we may assume by Remark~\ref{remark:convexity}
that at least one irreducible component of the~curve $T_P$ is not contained in $\mathrm{Supp}(D)$.

If $L_P$ is a~line that passes through $P$, then $L_P\subseteq\mathrm{Supp}(D)$,
since otherwise we would get
$$
1\geqslant D\cdot L_P\geqslant\mathrm{mult}_P(D)\mathrm{mult}_P(L_P)\geqslant \mathrm{mult}_P(D)>1
$$
by Lemma~\ref{lemma:Skoda}. Thus, we see that $\mathrm{mult}_P(T_P)=2$.

Let $f\colon \widetilde{S}\to S$ be the~blowup of the~point $P$, let $E$ be the~exceptional curve of the~blowup~$f$,
and let $\widetilde{D}$ be the~proper transform on $\widetilde{S}$ of the~$\mathbb{Q}$-divisor $D$.
Then $\mathrm{mult}_P(D)>1$ by Lemma~\ref{lemma:Skoda}.
Moreover, if follows from Lemma~\ref{lemma:blow-up} that the~log pair
$$
\left(\widetilde{S},\widetilde{D}+\left(\mathrm{mult}_P(D)-1\right)E\right)
$$
is not log canonical at some point $Q\in E$.
Moreover, there is a~commutative diagram
$$
\xymatrix{
\widetilde{S}\ar[d]_f\ar[rr]^g&&\overline{S}\ar[d]^h \\
S\ar@{-->}^{\psi}[rr]&&\mathbb{P}^2,}
$$
where $\psi$ is a~projection from $P$, the~morphism $g$ is a~contraction of the~proper transforms of all lines in $S$ that pass through $P$,
and $h$ is a~double cover branched over a~quartic curve.
This quartic curve has at most two ordinary double points, because $\mathrm{mult}_P(T_P)\ne 3$.

Let $\widetilde{T}_P$ be the~proper transform on $\widetilde{S}$ of the~curve $T_P$.
Then $Q\in E\cap\widetilde{T}_P$ by Lemma~\ref{lemma:double-plane}.

Note that $T_P$ is one of the~following curves: an~irreducible cubic curve,
a union of a~conic and a~line, a~union of three lines.
Let us consider these cases separately.

Suppose that $T_P$ is a~union of a~conic and a~line, so that $T_P=L_P+C_P$, where $L_P$ is a~line, and~$C_P$ is an~irreducible conic.
Then $L_P\subset\mathrm{Supp}(D)$, so that $C_P$ is not contained in $\mathrm{Supp}(D)$.
Thus, we write $D=aL_P+\Omega$, where $a\in\mathbb{Q}_{>0}$, and $\Omega$ is an~effective $\mathbb{Q}$-divisor on $S$
whose support contains none of the~curves $L_P$ and $C_P$. Put
$m=\mathrm{mult}_P(\Omega)$. Then $\mathrm{mult}_P(D)=m+a$ and
$$
2-2a=\Omega\cdot C_P\geqslant m,
$$
which gives $m+2a\leqslant 2$.
Similarly, we obtain $1+a\geqslant m$ by using
$$
1+a=L_P\cdot D=\Omega\cdot L_P\geqslant m.
$$
Denote by $\widetilde{C}_P$ the~proper transform of the~conic $C_P$
on the~surface $\widetilde{S}$, denote by $\widetilde{L}_P$  the~proper
transform of the~line $L_P$ on the~surface $\widetilde{S}$, and denote by $\widetilde{\Omega}$ the~proper transform of the~divisor $\Omega$
on the~surface $\widetilde{S}$. Put $\widetilde{m}=\mathrm{mult}_{Q}(\widetilde{\Omega})$.
Then the~log pair
\begin{equation}
\label{equation:cubic-surface-L-C}
\left(\widetilde{S},a\widetilde{L}_P+\widetilde{\Omega}+\left(m+a-1\right)E\right)
\end{equation}
is not log canonical at $P$. Now, applying Lemma~\ref{lemma:Skoda} to this log pair, we obtain $2a+m+\widetilde{m}>2$.
On the~other hand, if $Q\in\widetilde{C}_P$, then
$$
2-2a-m=\widetilde{\Omega}\cdot\widetilde{C}_P\geqslant\widetilde{m},
$$
so that $Q\not\in\widetilde{C}_P$.
Since $Q\in\widetilde{T}_P$, we see that $Q\in\widetilde{L}_P$. Then we have
$$
1+a-m=\widetilde{\Omega}\cdot\widetilde{L}_P\geqslant\widetilde{m},
$$
so that $2\geqslant 1+a\geqslant m+\widetilde{m}\geqslant 2\widetilde{m}$, which gives $\widetilde{m}\leqslant 1$.
Thus, we can apply Theorem~\ref{theorem:Vanya} to the~log pair \eqref{equation:cubic-surface-L-C} at the~point $Q$.
This gives
$$
m=\widetilde{\Omega}\cdot E\geqslant\big(\widetilde{\Omega}\cdot E\big)_Q>2(2-a-m)
$$
or
$$
1+a-m=\widetilde{\Omega}\cdot\widetilde{L}\geqslant\big(\widetilde{\Omega}\cdot\widetilde{L}\big)_Q>2(1-a),
$$
so that we get $3a+m>3$ or $2a+m>2$, which is impossible since $a\leqslant 1$ and $m+2a\leqslant 2$.

Therefore, we conclude that the~curve $T_P$ a~union of three lines.
Hence, we have $T_P=L_1+L_2+L_3$, where $L_1$, $L_2$, $L_3$ are lines in $S$ such that $P=L_1\cap L_2$ and $P\not\in L_3$.
Then $L_1\subset\mathrm{Supp}(D)\supset L_2$.
Therefore, we can write $D=a_1L_1+a_2L_2+\Delta$, where $a_1$ and $a_2$ are some positive rational numbers,
and $\Delta$ is an~effective $\mathbb{Q}$-divisor whose support does not contain $L_1$ and $L_2$.
Put $\mathbf{m}=\mathrm{mult}_P(\Delta)$. Then
$$
\mathbf{m}\leqslant\Delta\cdot L_1=\left(H-a_1L_1-a_2L_2\right)\cdot L_1=1+a_1-a_2,
$$
because $L_1\cdot L_2=1$ and $L_1^2=-1$ on the~surface $S$. Similarly, we see that
$$
\mathbf{m}\leqslant\Delta\cdot L_2=\left(H-a_1L_1-a_2L_2\right)\cdot L_2=1-a_1+a_2.
$$
This gives $\mathbf{m}\leqslant 1$.
Thus, we can apply Theorem~\ref{theorem:Vanya} to the~log pair $(S,D)$ at the~point $P$.
Then
$$
1+a_1-a_2=\Delta\cdot L_1\geqslant\big(\Delta\cdot L_1\big)_P>2(1-a_2)
$$
or
$$
1-a_1+a_2=\Delta\cdot L_2\geqslant\big(\Delta\cdot L_2\big)_P>2(1-a_1),
$$
which implies that $a_1+a_2>1$.
On the~other hand, we have
$$
0\leqslant\Delta\cdot L_3=\left(H-a_1L_1-a_2L_2\right)\cdot L_3=1-a_1-a_2,
$$
which implies that $a_1+a_2\leqslant 1$. The obtained contradiction completes the~solution.
\end{proof}

We now claim that a~smooth del Pezzo surface of degree $d\leqslant 3$ cannot contain a~$(-K_S)$-cylinder.
If $d\leqslant 2$, the~claim is \cite[Proposition~5.1]{KPZ14a}.
Similarly, if $d=3$, then the~claim is \cite[Theorem~1.7]{CPW16a}.
Let us show how to derive the~claim from from Theorem~\ref{theorem:del-Pezzo-degree-1-2-3} and Remark~\ref{remark:obstruction}.

Suppose that $S$ contains a $(-K_S)$-polar cylinder $U$. Then
$$
S\setminus U=C_1\cup\cdots\cup C_n
$$
for some irreducible curves $C_1,\ldots,C_n$ in  $S$,
and there are positive rational numbers $\lambda_1,\ldots,\lambda_n$ such that
$$
\sum_{i=1}^n\lambda_iC_i\sim_{\mathbb{Q}} -K_{S}.
$$
Put $D=\lambda_1C_1+\cdots+\lambda_nC_n$.
Then $(S,D)$ is not log canonical at some point $P\in S$ by Remark~\ref{remark:obstruction}.
Hence, by Theorem~\ref{theorem:del-Pezzo-degree-1-2-3},
there exists a~curve $T\in |-K_{S}|$ such that
\begin{itemize}
\item the~log pair $(S, T)$ is not log canonical at $P$;
\item and $\mathrm{Supp}(T)\subseteq\mathrm{Supp}(D)$.
\end{itemize}
Then $D\ne T$, because $n>3$ by \eqref{equation:KPZ-r}, and $T$ does not have more than $d\leqslant 3$ irreducible components.
Thus, there exists a~rational number $\mu>0$ such that $(1+\mu)D-\mu T$ is effective,
and its support does not contain at least one irreducible component of the~curve $T$.
Then $(S,(1+\mu)D-\mu T)$ is not log canonical at $P$ by Remark~\ref{remark:obstruction},
which contradicts to Theorem~\ref{theorem:del-Pezzo-degree-1-2-3}, since
$$
(1+\mu)D-\mu T\sim_{\mathbb{Q}} -K_S.
$$

\subsection{Construction of polar cylinders}
\label{subsection:cylinders-constructions}
Now, we show how to construct anticanonical polar cylinders in singular del Pezzo surfaces with Du Val singularities.
We start with

\begin{lemma}[{\cite[Theorem~3.19]{KPZ11}}]
\label{lemma:KPZ}
Let $S$ is a smooth del Pezzo surface. Suppose that $K_S^2\geqslant 4$. Then the~surface $S$ contains a $(-K_S)$-polar cylinder.
\end{lemma}

\begin{proof}
We may assume that $S\ne\mathbb{P}^1\times\mathbb{P}^1$.
Then there exists a birational map $\sigma\colon S\to\mathbb{P}^2$ that blows up $k\leqslant 5$ distinct points.
Let  $E_1,\ldots,E_k$ be the~$\sigma$-exceptional curves,
let~$C$ be an irreducible conic in $\mathbb{P}^2$ that contains all points $\sigma(E_1),\ldots,\sigma(E_k)$,
and let~$L$ be a line in $\mathbb{P}^2$ that is tangent to the~conic $C$ at some point that is different from $\sigma(E_1),\ldots,\sigma(E_k)$.
Denote by $\widetilde{C}$ and $\widetilde{L}$ the~proper transforms on $S$ of the~curves $C$ and $L$, respectively.
Then
$$
-K_{S}\sim\sigma^*(-K_{\mathbb{P}^2})-\sum_{i=1}^{k} E_i\sim_{\mathbb{Q}} (1+\varepsilon)\widetilde{C}+(1-2\varepsilon)\widetilde{L}+\varepsilon\sum_{i=1}^{k} E_i
$$
for every positive $\varepsilon<\frac{1}{2}$.
On the~other hand, we have
$$
S\setminus(\widetilde{C}\cup\widetilde{L}\cup E_1\cup\cdots\cup E_k)\cong\mathbb{P}^2\setminus(C\cup L)\cong\Big(\mathbb{A}^1\setminus\big\{0\big\}\Big)\times\mathbb{A}^1,
$$
so that the~surface $S$ contains a $(-K_S)$-polar cylinder.
\end{proof}

Now let us present an example of a singular del Pezzo surface of degree $2$ that
has one singular point of type $\mathrm{A}_2$ and contains an anticanonical polar cylinder.

\begin{example}
\label{example:dp2-A2}
Let $h\colon\widehat{S}\to\mathbb{P}^2$ be a~composition of $10$ blowups,
let $E_1,\ldots,E_{10}$ be the~exceptional curves of the birational morphism $h$,
let $L_1$ and $L_2$ be two distinct lines in~$\mathbb{P}^2$,
and let $\widehat{L}_1$ and $\widehat{L}_2$ be their proper transforms on $\widehat{S}$, respectively.
Now, let us choose $h$ such that the~intersections of these twelve curves are depicted as follows:
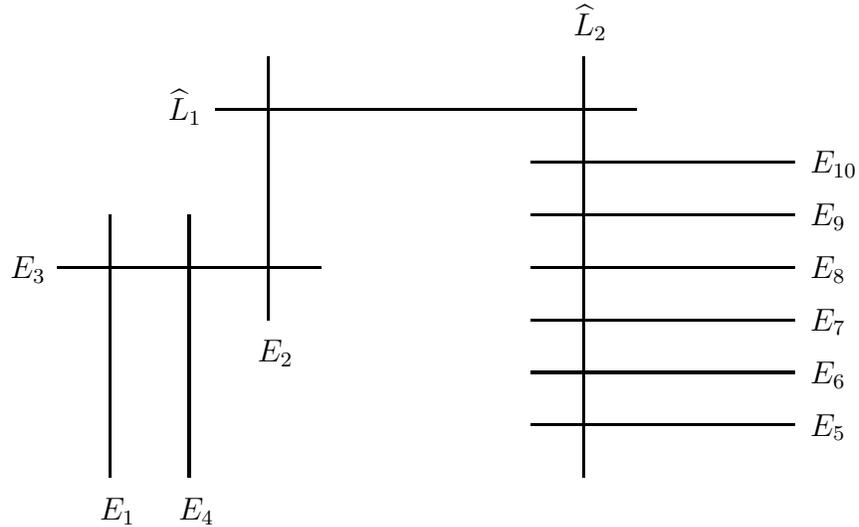
\begin{center}
\setlength{\unitlength}{0.70mm}
\begin{picture}(170,110)(-45,-47)
\thicklines
\put(-9,38){\mbox{$\widehat{L}_1$}}\put(0,40){\line(1,0){80}}
\put(10,0){\line(0,1){50}}\put(8,-8){$E_2$}
\put(-30,10){\line(1,0){50}}\put(-39,8){$E_3$}
\put(-20,-30){\line(0,1){50}}\put(-22,-38){$E_1$}
\put(-5,-30){\line(0,1){50}}\put(-7,-38){$E_4$}
\put(113,-22){$E_5$}\put(60,-20){\line(1,0){50}}
\put(113,-12){$E_6$}\put(60,-10){\line(1,0){50}}
\put(113,-2){$E_7$}\put(60,0){\line(1,0){50}}
\put(113,8){$E_8$}\put(60,10){\line(1,0){50}}
\put(113,18){$E_9$}\put(60,20){\line(1,0){50}}
\put(113,28){$E_{10}$}\put(60,30){\line(1,0){50}}
\put(68,54){\mbox{$\widehat{L}_2$}} \put(70,-30){\line(0,1){80}}
\end{picture}
\end{center}
Note that the $10$ blowups are arranged in the order indicated by the indices of their exceptional curves $E_i$.
To describe the~intersection form of the~curves $\widehat{L}_1$, $\widehat{L}_2$, $E_1,\ldots,E_{10}$, observe that
$$
\widehat{L}_1^2=-1,\quad \widehat{L}_2^2=-5,\,\,\, E_1^2=-3,\quad E_2^2=-2,\quad E_3^2=-2,\,\,\, E_4^2=
\ldots=E_{10}^2=-1.
$$
Let $g\colon\widehat{S}\to \widetilde{S}$ be the~contraction of the~curves $\widehat{L}_1$, $E_2$, $E_3$,
and let $\widetilde{L}_2$, $\widetilde{E}_1$, $\widetilde{E}_4,\ldots,\widetilde{E}_{10}$ be the~proper transforms on $\widetilde{S}$ of the~curves $\widehat{L}_2$, $\widehat{E}_1$, $\widehat{E}_4,\ldots,\widehat{E}_{10}$, respectively.
Then $\widetilde{S}$ is smooth and $K_{\widetilde{S}}^2=2$.
Moreover, the~divisor $-K_{\widetilde{S}}$ is nef.
To show this, fix an arbitrary positive rational number~$\epsilon<\frac{1}{3}$, let $D_{\widehat{S}}$ be the~following $\mathbb{Q}$-divisor:
$$
(2-\epsilon)\widehat{L}_1+(1+\epsilon)\widehat{L}_2+(1-\epsilon)E_{1}+(2-2\epsilon)E_{2}+(2-3\epsilon)E_{3}+(1-3\epsilon)E_{4}+\epsilon\Big(E_{5}+E_{6}+E_{7}+E_{8}+E_{9}+E_{10}\Big),
$$
and denote by $D_{\widetilde{S}}$ its proper transform on $\widetilde{S}$.
Then $D_{\widehat{S}}$ is effective, $D_{\widehat{S}}\sim_{\mathbb{Q}} -K_{\widehat{S}}$ and $D_{\widetilde{S}}\sim_{\mathbb{Q}} -K_{\widetilde{S}}$.
Moreover, we have $\widetilde{L}_2^2=\widetilde{E}_1=-2$, $\widetilde{E}_4^2=0$ and $\widetilde{E}_5^2=\cdots=\widetilde{E}_{10}^2=-1$, so that
$$
-K_{\widetilde{S}}\cdot\widetilde{L}_2=-K_{\widetilde{S}}\cdot\widetilde{E}_1=0, -K_{\widetilde{S}}\cdot\widetilde{E}_4=2, -K_{\widetilde{S}}\cdot\widetilde{E}_{5}=\cdots=-K_{\widetilde{S}}\cdot\widetilde{E}_{10}=1.
$$
This shows that $-K_{\widetilde{S}}$ is nef. Moreover, we also see that $\widetilde{L}_2$ and $\widetilde{E}_1$ are the~only $(-2)$-curves in~$\widetilde{S}$.
Let $f\colon\widetilde{S}\to S$ be the~birational contraction of these two $(-2)$-curves.
Then $S$ is a~del Pezzo surface with one singular point of type $\mathrm{A}_2$ such that $K_S^2=2$.
Let $D_{S}=f\circ g\big(D_{\widetilde{S}})$.
Then $D_S\sim_{\mathbb{Q}} -K_{S}$ and
$$
S\setminus \mathrm{Supp}(D_S)\cong\mathbb{P}^2\setminus \mathrm{Supp}(D_{\mathbb{P}^2})\cong \mathbb{A}^1\times\Big(\mathbb{A}^{1}\setminus\big\{0\big\}\Big),
$$
so that $S$ contains $(-K_S)$-polar cylinder.
\end{example}

One can use the~construction in Example~\ref{example:dp2-A2} to construct an~anticanonical polar cylinder
in \emph{every} del Pezzo surface of degree $2$ that has a~single singular point of type $\mathrm{A}_2$ (see \mbox{\cite[\S 4.3]{CPW16b}}).
Similarly, we can prove the~existence part of Theorem~\ref{theorem:dP-du-Val}.
However, there is an alternative proof, which is more algebraic. Let us describe it following \cite{CDP18}.

Let $S$ be a~singular del Pezzo surface  of degree $K_S^2\leqslant 3$ that has at most Du Val singularities, and let $P$ be its singular point.
Suppose, in addition, that the~following conditions hold:
\begin{itemize}
\item the~singular point $P$ is not of type $\mathrm{A}_1$ if $K_S^2=2$;
\item the~singular point is not of types $\mathrm{A}_1$, $\mathrm{A}_2$, $\mathrm{A}_3$, $\mathrm{D}_4$ if $K_S^2=1$.
\end{itemize}
Now, let us prove that $S$ contains a $(-K_S)$-polar cylinder (cf. Theorem~\ref{theorem:dP-du-Val}).

Denote by~$\mathbb{P}$ the~three-dimensional weighted projective space in which $S$ sits as a~hypersurface.
Note that~$\mathbb{P}=\mathbb{P}^3$ (respectively, $\mathbb{P}(1,1,1,2)$,  $\mathbb{P}(1,1,2,3)$) if $K_S^2=3$ (respectively, $K_S^2=2$,  $K_S^2=1$).
For the~quasi-homogenous coordinate system for $\mathbb{P}$, we use $[x:y:z:w]$.
By a~coordinate change, we may assume that $P=[1:0:0:0]$.
Then the~equation of $S$ can be described as follows:
\begin{itemize}
\item if $K_S^2=3$, then $S$ is given by
\begin{equation}
\label{eq:degree3}
xf_2(y,z,w)+f_{3}(y,z,w)=0,
\end{equation}
where $f_2$ and $f_3$ are polynomials of degrees $2$ and $3$, respectively;
\item if $K_S^2=2$, then $S$ is given by
\begin{equation}
\label{eq:degree2}
w^2+x\big(ayw+f_3(y,z)\big)+f_{4}(y,z)=0,
\end{equation}
where $f_3$ and $f_4$ are polynomials of degrees $3$ and $4$, respectively, and $a\in\Bbbk$;
\item if $K_S^2=1$, then $S$ is given by
\begin{equation}
\label{eq:degree1-y}
w^2+x\big(ay^2w+f_5(y,z)\big)+f_{6}(y,z)=0
\end{equation}
or
\begin{equation}
\label{eq:degree1-z}
w^2+x\big(zw+f_5(y,z)\big)+f_{6}(y,z)=0,
\end{equation}
where $f_5$ and $f_6$ are polynomials of degrees $5$ and $6$, respectively, and~$a\in\Bbbk$.
\end{itemize}

Let $\Pi$ be the~hyperplane in $\mathbb{P}$ defined by~$x=0$, and let $\pi\colon S\dasharrow \Pi$ be the~map given by
$$
\big[x:y:z:w\big]=\big[0:y:z:w\big].
$$
The hyperplane $\Pi$ is isomorphic to $\mathbb{P}^2$, $\mathbb{P}(1,1,2)$, $\mathbb{P}(1,2,3)$ according to $K_S^2=3$, $2$, $1$, respectively.
We denote by $g(y,z,w)$ the~coefficient of $x$ in each of equations \eqref{eq:degree3}, \eqref{eq:degree2}, \eqref{eq:degree1-y} and  \eqref{eq:degree1-z}.
Namely, if $K_S^2=3$, then $g(y,z,w)=f_2(y,z,w)$.
Similarly, if $K_S^2=2$, then
$$
g(y,z,w)=ayw+f_3(y,z).
$$
Finally, if $K_S^2=1$, then $g(y,z,w)=zw+f_5(y,z)$ or
$$
g(y,z,w)=ay^2w+f_5(y,z).
$$
Let $D$ be the~divisor on $S$ that is cut out by $g(y,z,w)=0$.
If $K_S^2=3$, then $D$ consists of the~lines that contains $P$.
There are at most six such lines and they are defined in $\mathbb{P}^3$ by
$$
\left\{\aligned%
&g(y,z,w)=0,\\
&f_3(y,z,w)=0.\\
\endaligned
\right.
$$
Similarly, if $K_S^2=2$, then the~divisor $D$ consists of at most six curves passing through the~point~$P$.
They are defined in $\mathbb{P}(1,1,1,2)$ by
$$
\left\{\aligned%
&g(y,z,w)=0,\\
&w^2+f_4(y,z)=0.\\
\endaligned
\right.
$$
Finally, if $K_S^2=1$, then the~divisor $D$ consists of at most five curves passing through the~point~$P$, which are defined in $\mathbb{P}(1,1,2,3)$ by
$$
\left\{\aligned%
&g(y,z,w)=0,\\
&w^2+f_6(y,z)=0.\\
\endaligned
\right.
$$
In each case, the~number of curves in $D$ is the~same as the~number of points determined by the~corresponding system of equations in $\Pi$.
We denote these curves by $L_1,\ldots, L_r$ in each case.
The~map $\pi$ contracts each curve $L_i$ to a~point on $\Pi$.

The equations \eqref{eq:degree3}, \eqref{eq:degree2}, \eqref{eq:degree1-y} and  \eqref{eq:degree1-z}
immediately imply that $\pi$ is a~birational map.
Moreover, it induces an~isomorphism
$$
\widetilde{\pi}\colon S\setminus \Big(L_1\cup\cdots\cup L_r\Big)\cong\mathrm{Im}\big(\widetilde{\pi}\big)\subset\Pi.
$$
Let $\mathcal{C}$ be the~curve on $\Pi$ defined by $g(y,z,w)=0$.
Then $\mathcal{C}$ can be reducible or non-reduced.

\begin{lemma}
\label{lemma:cylinders-d-3}
Suppose that $K_S^2=3$.
Then there is a~hyperplane section $H$ of the~surface $S$ such that the~complement $S\setminus\left( H\cup L_1\cup\cdots\cup L_r\right)$ is a $(-K_S)$-polar cylinder.
\end{lemma}

\begin{proof}
Observe that $\mathrm{Im}(\widetilde{\pi})=\Pi\setminus \mathcal{C}$.
Let $\varphi\colon\overline{S}\to S$ be the~blowup of the~point $P$.
Then there exists a~commutative diagram
$$
\vcenter{\xymatrix{
&\overline{S}\ar@{->}[ld]_{\varphi}\ar@{->}[rd]^{\psi}&\\%
S\ar@{-->}[rr]^{\pi}&&\Pi}}
$$ %
where $\psi$ is the~birational morphism that contracts  the~proper transforms of the~lines $L_1,\ldots, L_r$.
Let~$E$ be the~exceptional curve of the~blowup $\varphi$.
Then $\psi(E)=\mathcal{C}$, and $\mathcal{C}$ contains each point~$\pi(L_i)$.

If $P$ is an~ordinary double point of the~cubic surface $S$, then the~curve $\mathcal{C}$ is a~smooth conic.
Similarly, if $P$ is a~singular point of type $\mathrm{A}_n$ for $n\geqslant 2$, then $\mathcal{C}$ splits as a union of two distinct~lines.
Finally, if $P$ is either of type $\mathrm{D}_n$ or of type $\mathrm{E}_6$, then $\mathcal{C}$ is a~double line.

If $\mathcal{C}$ is smooth, let $\ell$ be a~general line in $\Pi$ that is tangent to $\mathcal{C}$.
If $\mathcal{C}$ is singular, let $\ell$ be a~general line in $\Pi$ that passes through a~singular point of the~conic $\mathcal{C}$.
By a~suitable coordinate change, we may assume that $\ell$ is defined by $x=y=0$.
Let $H$ be the~curve in $S$ cut out by $y=0$.
Then
$$
S\setminus\Big(H\cup L_1\cup\cdots\cup L_r\Big)\cong\Pi\setminus\Big(\mathcal{C}\cup\ell\Big)\cong\left\{\aligned%
&\Big(\mathbb{A}^1\setminus\big\{0, 1\big\}\Big)\times\mathbb{A}^1 \mbox{ if $\mathcal{C}$ is a union of two distinct lines,}
\\
&\Big(\mathbb{A}^1\setminus\big\{0\big\}\Big)\times\mathbb{A}^1 \mbox{ otherwise.}
\\
\endaligned
\right.
$$
Therefore, $S\setminus(H\cup L_1\cup\cdots\cup L_r)$ is a cylinder. But $H+D\sim -3K_S$ and $L_1\cup\cdots\cup L_r=\mathrm{Supp}(D)$.
Thus, the~complement $S\setminus(H\cup L_1\cup\cdots\cup L_r)$ is a $(-K_S)$-polar cylinder.
\end{proof}

To deal with the~cases $K_S^2=1$ and $K_S^2=2$, let $\ell_y$ be the~curve in $\mathbb{P}$ that is given by $x=y=0$,
and let $H_y$ be the~curve in the~surface $S$ that is cut out by $y=0$.

\begin{lemma}
\label{lemma:cylinders-d-1-2}
If $K_S^2=2$ or if $K_S^2=1$ and the~surface $S$ is defined by~equation~\eqref{eq:degree1-y},
then the~complement $S\setminus\left(H_y\cup L_1\cup\cdots\cup L_r\right)$ is a $(-K_S)$-polar cylinder.
\end{lemma}

\begin{proof}
Observe that the~morphism $\widetilde{\pi}$ gives an isomorphism $S\setminus\left(H_y\cup L_1\cup\cdots\cup L_r\right)\cong\Pi\setminus\left(\mathcal{C}\cup\ell_y\right)$.
But $\pi$ maps $S\setminus H_y$ onto $\Pi\setminus \ell_y\cong\mathbb{A}^2$.
Thus, if $K_S^2=2$, then $S\setminus\left(H_y\cup L_1\cup\cdots\cup L_r\right)$ is isomorphic to
the~complement in $\A^2$ of the~curve defined by
$$
aw+f_3(1,z)=0.
$$
Similarly, if $K_S^2=1$ and $S$ is defined by~\eqref{eq:degree1-y},
then $S\setminus\left(H_y\cup L_1\cup\cdots\cup L_r\right)$ is isomorphic to
the~complement in $\A^2$ of the~curve defined by
$$
aw+f_5(1,z)=0.
$$
Therefore, in both cases, the~complement $S\setminus\left(H_y\cup L_1\cup\cdots\cup L_r\right)$ is a cylinder.
Now, arguing as in~the~proof of Lemma~\ref{lemma:cylinders-d-3}, we see that $S\setminus\left(H_y\cup L_1\cup\cdots\cup L_r\right)$ is a $(-K_S)$-polar cylinder.
\end{proof}

Finally, to deal with the~remaining case, let $\ell_z$ be the~curve in $\mathbb{P}$ that is given by $x=z=0$,
and~let $H_z$ be the~hyperplane section of $S$ that is cut by $z=0$.

\begin{lemma}
\label{lemma:cylinders-d-1}
Suppose that $K_S^2=1$ and the~del Pezzo surface $S$ is given by equation~\eqref{eq:degree1-z}.
Then the~complement $S\setminus\left( H_z\cup L_1\cup\cdots\cup L_r\right)$ is a $(-K_S)$-polar cylinder.
\end{lemma}

\begin{proof}
Observe that the~morphism $\widetilde{\pi}$ gives an isomorphism $S\setminus\left(H_z\cup L_1\cup\cdots\cup L_r\right)\cong\Pi\setminus\left(\mathcal{C}\cup\ell_z\right)$.
But $\pi$ maps $S\setminus H_z$ onto $\Pi\setminus\ell_z$.
Then $\Pi\setminus\left(\mathcal{C}\cup\ell_z\right)$  is the~complement of the~curve defined by
$$
w+f_5(y,1)=0
$$
in $\Pi\setminus \ell_z\cong\mathbb{A}^2/ \boldsymbol\mu_2$, where the~$\mumu_2$-action is given by $(y,w)\mapsto (-y,-w)$.

Since $f_5(y,1)$ is an~odd polynomial in $y$, the~isomorphism $\mathbb{A}^2\to\mathbb{A}^2$ defined by
$$
\big(y,w\big)\mapsto\big(y, w+f_5(y,1)\big)
$$
is $\mumu_2$-equivariant and gives an isomorphism between the~complement $\Pi\setminus\left(\mathcal{C}\cup\ell_z\right)$
and the~complement in $\mathbb{A}^2/\mumu_2$ of the~image of the~curve defined by $w=0$,
which is isomorphic to $\mathbb{A}^1\setminus\{0\}\times \mathbb{A}^1$.

We see that  $S\setminus\left(H_z\cup L_1\cup\cdots\cup L_r\right)$ is a cylinder.
Now, arguing as in the~proof of Lemma~\ref{lemma:cylinders-d-3}, we conclude that $S\setminus\left(H_z\cup L_1\cup\cdots\cup L_r\right)$ is a $(-K_S)$-polar cylinder.
\end{proof}

\section{Cylinders in higher-dimensional varieties}
\label{section:Fanos-cylinders}

In this section, we describe known results about cylinders in smooth Fano threefolds and fourfolds, and varieties fibred into del Pezzo surfaces.
Let us say few words about Fano varieties \cite{Iskovskikh-1980-Anticanonical,Mu89,IP99}.

Let $V$ be a~smooth Fano variety of dimension~$n\geqslant 3$.
The number $(-K_V)^n$ is known as the~degree of the~Fano variety $V$. Put
$$
\iota(V)=\max\Big\{t\in\mathbb{N}\ \big| -K_V \sim tH\ \text{for}\ H\in \Pic(V)\Big\}.
$$
Then $\iota(V)$ is known as the~(Fano) index of the~variety $V$.
It is well known that $1\leqslant \iota (V)\leqslant n+1$.
Moreover, one has
$$
\iota (V)=n+1\ \iff V\cong\PP^n.
$$
Similarly, we have $\iota (V)=n$ if and only if $V$ is a quadric (see \cite{Kobayashi-Ochiai-1973,IP99}).

\begin{remark}[{\cite{Fujita1980,Fujita1981,Fujita1984,IP99}}]
\label{remark:del-Pezzo-n-folds}
Suppose that $\iota(V)=n-1$. Then
$$
-K_V\sim (n-1)H
$$
for some ample divisor $H\in\Pic(V)$. In this case, the~variety $V$ is usually called a \textit{del Pezzo variety}.
If $\uprho(V)=1$, then there are just the~following possibilities:
\begin{itemize}
\item
$H^n=1$ and $V=V_6$ is a~weighted hypersurface in $\PP(1^n,2,3)$ of degree $6$;
\item
$H^n=2$ and $V=V_4$ is a~weighted hypersurface in $\PP(1^{n+1},2)$ of degree $4$;
\item
$H^n=3$ and $V=V_3$ is a~cubic hypersurface in $\PP^{n+1}$;
\item
$H^n=4$ and $V=V_{2\cdot 2}$ is a~complete intersection of two quadrics in $\PP^{n+2}$;
\item
$H^n=5$, $n\in\{3,4,5,6\}$ and $V$ is described in Example~\ref{example:V5}.
\end{itemize}
If $\dim (V)=3$ and $\uprho(V)=1$, then the~values of the~Hodge number $h^{1,2}(V)$ are given in
\begin{center}\renewcommand{\arraystretch}{1.4}
\begin{tabular}{|c||c|c|c|c|c|}
\hline
\quad$H^3$\quad \quad&\quad $1$\quad\quad &\quad  $2$\quad\ \quad &\quad $3$\quad\quad &\quad $4$\quad\quad &\quad $5$\quad\quad\\
\hline
$h^{1,2}(V)$&$21$&$10$&$5$&$2$&$0$\\
\hline
\end{tabular}
\end{center}
\end{remark}

Let us prove cylindricity of any higher-dimensional smooth intersection of two quadrics.

\begin{lemma}[{\cite{KPZ11}}]
\label{lemma:dP4}
Let $V$ be a smooth complete intersection of two quadric hypersurfaces in $\PP^{n+2}$.
Then $V$ is cylindrical.
\end{lemma}

\begin{proof}
Let $\ell$ be a line in $V$, let $D$ be an irreducible divisor in $X$ swept out by  lines meeting~$\ell$,
let~$\sigma\colon\widetilde{V}\to V$ be the~blowup of the~line $\ell$, let $E$ be its exceptional divisor,
and let $\widetilde{D}$ be  the~proper transform on $\widetilde{V}$ of the~divisor $D$.
There exists the~following commutative diagram:
$$
\xymatrix{
&\widetilde{V}\ar[dl]_{\sigma}\ar[dr]^{\varphi}&\\
V\ar@{-->}[rr]^{\psi}&& \PP^n}
$$
where $\varphi$ is a~birational morphism that contracts $\widetilde{D}$, and $\psi$ is  the~projection from $\ell$.
Thus, we have
$$
V\setminus D\cong \PP^n\setminus\varphi(E).
$$
But $\varphi(E)$ is a~quadric that contains a~one-parameter family of linear subspaces of dimension~\mbox{$n-2$}.
Hence, this quadric is singular,
so that $\PP^n\setminus\varphi(E)$ contains a cylinder.
\end{proof}

Smooth Fano varieties of dimension $n\geqslant 3$ and index $n-2$ are known as \textit{Fano--Mukai varieties}.
If $V$ is a Fano--Mukai variety and $H\in\Pic(V)$ such that $-K_V\sim (n-2)H$, then the~number
$$
\g(V)=\frac{1}{2}H^n+1
$$
is integral and is called the~\textit{genus} of the~Fano--Mukai variety $V$.
The possible values of the~genus are given in the~following table:
\begin{center}\renewcommand{\arraystretch}{1.4}
\begin{tabular}{|c||c|c|c|c|c|c|c|}
\hline
\quad$\g(V)$\quad\quad & \quad$2\leqslant\g(V)\leqslant 5$\quad\quad & \quad$6$\quad\quad & \quad$7$\quad\quad & \quad$8$\quad\quad & \quad$9$\quad\quad &\quad$10$\quad\quad &\quad$12$\quad\quad\\
\hline
$\dim(V)$&any&$\leqslant 6$&$\leqslant 10$&$\leqslant 8$&$\leqslant 6$ &$\leqslant 5$&$3$\\
\hline
\end{tabular}
\end{center}
Moreover, the~following result has been recently proved in \cite{KP2020}.

\begin{theorem}
\label{theorem:g-7-8-9-10}
Let $V$ be a smooth Fano--Mukai variety such that $\uprho(V)=1$ and $\g(V)\in\{7,8,9,10\}$.
Suppose that $\dim(V)\geqslant 5$. Then $V$ is cylindrical.
\end{theorem}

In Subsection~\ref{subsection:Fano-3-folds}, we will outline several known results about cylindrical smooth Fano threefolds.
Then, in Subsection~\ref{subsection:Fano-4-folds}, we will present constructions of cylinders in some smooth Fano fourfolds.
In particular, we will explain how to prove the~following result:

\begin{theorem}
\label{theorem:g-7-8-9-10-fourfolds}
For every $g\in\{7,8,9,10\}$, there is a cylindrical Fano--Mukai fourfold of genus $g$.
\end{theorem}

Finally, in Subsection~\ref{subsection:MFS}, we will present results about cylinders in Mori fibrations.

\subsection{Cylindrical Fano threefolds}
\label{subsection:Fano-3-folds}

Let $X$ be a smooth Fano variety that has dimension three.
Then $X$ belongs to one of $105$ families, which have been explicitly described in  \cite{Iskovskikh77, Iskovskikh78, Iskovskikh1990, Iskovskikh-1980-Anticanonical, MM81, Mori-Mukai1983, Mori-Mukai-1986, Mori-Mukai2003}.
Their~automorphism groups have been studied in \cite{CheltsovPrzyjalkowskiShramov,Kuznetsov-Prokhorov,Kuznetsov-Prokhorov-Shramov,Mukai-Umemura-1983,Prokhorov-1990c}.
In particular, we have

\begin{theorem}
\label{theorem:Fano-threefolds-Aut-infinite}
Let $X$ be a~smooth Fano threefold such that $\uprho(X)=1$ and $\Aut(X)$ is infinite.
Then $X$ and $\Aut(X)$ can be described as follows:
\begin{enumerate}
\item $X=\PP^3$ and $\Aut(X)\cong \PGL_4(\Bbbk)$;
\item $X$ is a smooth quadric in $\PP^4$ and $\Aut(X)\cong \mathrm{PSO}_5(\Bbbk)$;
\item $X$ is the~quintic del Pezzo threefold described in Example~\ref{example:V5} and $\Aut(X)\cong \PGL_2(\Bbbk)$;
\item $X$ is one of the~following Fano threefolds in $\PP^{13}$ of degree $22$ and genus $12$:
\begin{enumerate}
\item the~Mukai--Umemura threefold $X_{22}^{\mathrm{mu}}$ with $\Aut(X_{22}^{\mathrm{mu}})\cong \PGL_2(\Bbbk)$;
\item the~unique special threefold $X_{22}^{\mathrm{a}}$ with $\Aut(X_{22}^{\mathrm{a}})\cong \Ga\rtimes \mumu_4$;
\item a threefold $X_{22}^{\mathrm{m}}$ in one-parameter family with $\Aut(X_{22}^{\mathrm{m}})\cong \Gm\rtimes \mumu_2$.
\end{enumerate}
\end{enumerate}
\end{theorem}

Before we describe some cylindrical smooth Fano threefolds, observe that we have the~following implications:
\begin{center}
$X$ contains $(-K_X)$-polar cylinder $\Longrightarrow$ $X$ is cylindrical  $\Longrightarrow$ $X$ is rational.
\end{center}
Moreover, the~rationality problem for smooth Fano threefolds is \emph{almost} completely solved (see~\cite{IP99}).
In particular, for \emph{general} member of every family, we know whether it is rational or irrational.
It is expected that the~same answer holds for \emph{every} smooth member in each family.

If $\iota(X)\geqslant 3$, then either $X\cong\mathbb{P}^3$ or $X$ is a smooth quadric in $\mathbb{P}^4$, so that $X$ is cylindrical.

If~$\iota(X)=2$, then $-K_X\sim 2H$ for $H\in\mathrm{Pic}(X)$, and we have the~following possibilities:
\begin{itemize}
\item $H^3=1$ and $X=V_1$ is a sextic hypersurface in $\mathbb{P}(1,1,1,2,3)$;

\item $H^3=2$ and $X=V_2$ is quartic hypersurface in $\mathbb{P}(1,1,1,1,2)$;

\item $H^3=3$ and $X=V_3$ is a cubic hypersurface in $\mathbb{P}^4$;

\item $H^3=4$ and $X=V_4$ is an intersection of two quadrics in $\mathbb{P}^5$;

\item $H^3=5$ and $X=V_5$ is the~quintic del Pezzo threefold described in Example~\ref{example:V5};

\item $H^3=6$ and $X$ is a divisor in $\mathbb{P}^2\times\mathbb{P}^2$ of degree $(1,1)$;

\item $H^3=6$ and $X=\mathbb{P}^1\times\mathbb{P}^1\times\mathbb{P}^1$;

\item $H^3=7$ and $X=V_7$ is a blowup of $\mathbb{P}^3$ at a point.
\end{itemize}
In this case, if $H^3\leqslant 3$, then $X$ is irrational (see \cite{ArtinMumford,CheltsovPrzyjalkowskiShramov2019,CG,Grinenko03,Grinenko04,Voi88}), so that it is not cylindrical.
On~the other hand, if $H^3\geqslant 4$, then $X$ contains a $(-K_X)$-polar cylinder.
Indeed, if $H^3=4$, this follows from Lemma~\ref{lemma:dP4}.
If $H^3\geqslant 6$, this is obvious.
Finally, if $H^3=5$, this follows from

\begin{theorem}
\label{theorem:V5-A3}
Let $V_5$ be the~quintic del Pezzo threefold in $\mathbb{P}^6$ that is described in Example~\ref{example:V5}.
Then $V_5$ contains a hyperplane section $H$ such that $V_5\setminus H\cong\A^3$.
\end{theorem}

\begin{proof}
Let us give two constructions of the~required hyperplane section.
First, let $L$ be a line in~$X$.
Let $\alpha\colon\widetilde{V}_5\to V_5$ be the~blowup of the~line $L$.
Then we have the~following commutative diagram:
$$
\xymatrix{&
\widetilde{V}_5\ar[dl]_{\alpha} \ar[dr]^{\beta}&&     \\
V_5&& Q\\}
$$
where $Q$ is a smooth quadric in $\mathbb{P}^4$,
and $\beta$ is the blowup of a twisted cubic curve $C$ contained~in~$Q$.
Let~$H_C$ be the~unique hyperplane section of $Q$ that contains $C$,
and let $H_L$ be the~unique hyperplane section of $V_5$ that is singular along $L$.
Then $H_L$ is the~proper transform of the~$\beta$-exceptional surface, and $H_C$ is the~proper transform of the~$\alpha$-exceptional surface.
Note that $H_L$ is swept out by the~lines that intersects the~line $L$.
Moreover, it follows from \cite{IP99,Kuznetsov-Prokhorov-Shramov} that
$$
\NNN_{L/V_5}\cong
\begin{cases}
\OOO_L\oplus \OOO_L & \text{$L$ is a line of type $(0,0)$},\\
\OOO_L(1) \oplus \OOO_L(-1) & \text{$L$ is a line of type $(1,-1)$}.
\end{cases}
$$
The lines in $V_5$ are parameterized by $\mathbb{P}^2$,
and the~lines of the~type $(1,-1)$ are parameterized by a~smooth conic in this plane (see \cite{Iskovskikh-1980-Anticanonical,Furushima1989a,Kuznetsov-Prokhorov-Shramov}).
Furthermore, the~surface $H_C$ is smooth if and only if $L$ is a~line of type $(1,-1)$.
Thus, if we choose $L$ to be a~line of type $(1,-1)$ and put~$H=H_L$, then  $V_5\setminus H\cong Q\setminus H_C\cong\A^3$ as required.

To present the~second construction, let $P$ be a point in $V_5$.
Recall that $\mathrm{Aut}(V_5)\cong\PGL_2(\Bbbk)$.
Moreover, it follows from \cite{CheltsovShramov2016,Furushima1989a,Kuznetsov-Prokhorov-Shramov,Iskovskikh-1980-Anticanonical,Mukai-Umemura-1983}
that $\mathrm{Aut}(V_5)$ has exactly three orbits on $V_5$:
\begin{enumerate}
\item a closed one-dimensional orbit $\mathcal{C}$, which is a~twisted rational sextic curve in $\PP^6$;
\item a two-dimensional orbit $\mathring{\mathcal{S}}$ whose closure is a~surface $\mathcal{S}\sim -K_V$ which is singular along $\mathcal{C}$;
\item an open orbit $V_5\setminus \mathcal{S}$.
\end{enumerate}
Furthermore, let $k_P$ be the~number of lines in $V_5$ passing through $P$. Then
$$
k_P=\left\{\aligned%
&1\ \text{if $P\in\mathcal{C}$},\\
&2\ \text{if $\mathcal{S}\setminus\mathcal{C}$},\\
&3\ \text{if $V_5\setminus\mathcal{S}$}.\\
\endaligned
\right.
$$
Observe also that $\mathcal{S}$ is swept out by the~lines of type $(1,-1)$.

Let $\sigma\colon\widehat{V}_5\to V_5$ be the~blowup of the~point $P$.
Then it follows from \cite{Furushima1989} that there exists the~following Sarkisov link:
$$
\xymatrix{
&\widehat{V}_5\ar[dl]_{\sigma}\ar@{-->}[rr]^{\chi}&&\overline{V}_5\ar[dr]^{\varphi}&\\
V_5\ar@{-->}[rrrr]^{\psi}&&&& \PP^2}
$$
where $\chi$ is a~composition of flops of the~proper transforms of lines in $V_5$ that pass through $P$,
the~morphism $\varphi$ is a~$\PP^1$-bundle, and $\psi$ is given by the~linear system of hyperplane sections that are singular at the~point $P$.
Now we suppose that $P\in\mathcal{C}$.

Let $E$ be the~$\sigma$-exceptional surface, and let $\overline{E}$ be its proper transform on the~threefold $\overline{V}_5$.
Then~$\overline{E}$ is a~del Pezzo surface of degree $6$ with at most Du Val singularities,
and its singular locus consists of one singular point of type $\mathrm{A}_2$.
Moreover, the~$\PP^1$-bundle $\varphi\colon\overline{V}_5\to\PP^2$ induces a birational map $\overline{E}\to\mathbb{P}^2$
that contracts a single curve $\Gamma\subset\overline{E}$ to a point in $\mathbb{P}^2$.

Let $\mathscr{L}$ be a line in $\PP^2$ that passes through the~point $\varphi(\Gamma)$,
let $\overline{H}$ be its preimage in $\overline{V}_5$ via $\varphi$,
let $\widehat{H}$ be its proper transform on $\widehat{V}_5$, and let $H=\sigma(\widehat{H})$.
Then
$$
V_5\setminus H\cong \overline{V}_5\setminus\big(\overline{E}\cup\overline{H}\big),
$$
and $H$ is a~hyperplane section of the~threefold $V_5$ that is singular at~$P$.
Furthermore, one can show that the~surface $H$ is smooth away from $P$,
and $H$ has Du Val singularity of type $\mathrm{A}_4$ at this point.
Then the~$\PP^1$-bundle $\varphi$ induces a morphism
$\overline{V}_5\setminus\big(\overline{E}\cup\overline{H}\big)\to \PP^2\setminus \mathscr{L}$ that is an~$\mathbb{A}^1$-bundle over $\mathbb{A}^2$.
This implies that $V_5\setminus H\cong \overline{V}_5\setminus(\overline{E}\cup\overline{H})\cong\A^3$ as required.
\end{proof}

Now, we assume that $\iota(X)=1$. This leaves us $95$ families of smooth Fano threefolds \cite{MM81,IP99}.
If $\uprho(X)=1$, $\iota(X)=1$ and $\g(X)\leqslant 6$, then we have the~following possibilities:
\begin{enumerate}
\item $\g(X)=2$ and $X$ is a sextic hypersurface in $\mathbb{P}(1^4,3)$;

\item \label{Fano3-folds:g=3}
 $\g(X)=3$ and $X$ is an intersection of a quadric and a quartic in $\mathbb P(1^5,2)$;

\item \label{Fano3-folds:g=4} $\g(X)=4$ and $X$ is a complete intersection of a quadric and a cubic in $\mathbb{P}^5$;

\item \label{Fano3-folds:g=5} $\g(X)=5$ and $X$ is a complete intersection of three quadrics in $\mathbb{P}^6$;

\item \label{Fano3-folds:g=6} $\g(X)=6$ and $X$ is a section of  the~cone in $\mathbb{P}^8$ over the~smooth quintic del Pezzo fourfold described in Example~\ref{example:V5} by a quadric and a hyperplane.
\end{enumerate}
All these deformation families are irreducible.
General members of the~family~\eqref{Fano3-folds:g=3} are~smooth quartic hypersurfaces in $\mathbb{P}^4$,
and special members are double covers of the~quadric threefold branched over octic surfaces.
Similarly, general members of the~family~\eqref{Fano3-folds:g=6} are sections of the~smooth quintic del Pezzo fourfold
in $\mathbb{P}^7$ by quadrics, and special members are double covers of the~smooth quintic del Pezzo threefold branched over anticanonical surfaces.

In the~first two cases, the~Fano threefold $X$ is known to be
irrational even if we allow mild isolated singularities \cite{CheltsovPark2010,IlievKatzarkovPrzyjalkowski,IM,Iskovskikh1980,Kuznetsova,Mella,PrzyjalkowskiShramov,Shramov2007}.
In the~case \eqref{Fano3-folds:g=5}, the~threefold $X$ is also irrational~\cite{Beauville}.
General threefolds of the families \eqref{Fano3-folds:g=4}  and \eqref{Fano3-folds:g=6} are irrational~\cite{Beauville,Iskovskikh1997,IskovskikhPukhlikov,PukhlikovAMS,HassettTschinkel2019},
and every smooth member is also expected to be irrational.
Therefore, in all these cases, the~threefold $X$ is either non-cylindrical or it is expected to be irrational and, thus,~non-cylindrical.

\begin{remark}
\label{remark:g-6-special}
Let $V_5$ be the~smooth quintic del Pezzo threefold, see Example~\ref{example:V5},
and let $\pi\colon X\to V_5$ be a double cover branched over a surface $S\in |-K_{V_5}|$.
If $S$ has an isolated ordinary double point, then~$X$ is rationally connected \cite{Zhang2006},
it is $\mathbb{Q}$-factorial \cite{Cynk},
and it follows from \cite{Prokhorov2017} that there exists the~following Sarkisov link:
$$
\xymatrix{
&\widetilde{X}\ar[dl]_{\alpha}\ar[dr]^{\beta}&\\
X && \mathbb{P}^2,}
$$
where $\alpha$ is the blow up of the singular point of $X$,
and $\beta$ is a standard conic bundle, whose discriminant curve has degree $6$.
Hence, in this case, the threefold $X$ is irrational by \cite[Theorem~10.2]{Shokurov1983}.
Now, using \cite[Theorem~IV.1.8.3]{Kollar-1996-RC},
we conclude that  $X$ is also irrational if $S$ is a very general surface in the~linear system $|-K_{V_5}|$.
\end{remark}

If $\uprho(X)=1$, $\iota(X)=1$ and $\g(X)\geqslant 7$, then $\g(X)\in\{7,8,9,10,12\}$.
Moreover, if $\g(X)=8$, then the~threefold $X$ is birational to a smooth cubic hypersurface in $\mathbb{P}^4$ (see, for example, \cite{Iskovskikh1980,IP99,Takeuchi}),
so~that it is irrational~\cite{CG}.
On the~other hand, we know that $X$ is rational if
$$
\g(X)\in\big\{7,9,10,12\big\}.
$$
In these cases, the~divisor $-K_X$ is very ample, and $|-K_X|$ gives an embedding $X\hookrightarrow\PP^{\g(X)+1}$.
Moreover, all the known constructions of cylinders in $X$ use the \emph{double projection} from a line in $X$ (see~\cite{Iskovskikh1990}).
Recall from \cite{IP99,KPZ14b,Prokhorov-PhD} that $X$ can contain two types of lines depending on their normal bundles.
Namely, for a~line $\ell\subset X$, we have the~following two possibilities:
$$
\NNN_{\ell/X}\cong
\begin{cases}
\OOO_\ell\oplus \OOO_\ell(-1) & \text{$\ell$ is of type $(0,-1)$},
\\
\OOO_\ell(1) \oplus \OOO_\ell(-2) & \text{$\ell$ is of type $(1,-2)$}.
\end{cases}
$$
If $X$ is a sufficiently general member of one of these three families of smooth Fano threefolds, then~$X$~does not contain lines of type $(1,-2)$.
Moreover, one can show that the~threefolds containing lines of type $(1,-2)$ form a codimension one subset in the~corresponding moduli spaces.
On the~other hand, we have the~following result:

\begin{theorem}[{\cite[Theorem~0.1]{KPZ14b}}]
\label{theorem:g-9-10-12}
Suppose that $\uprho(X)=1$, $\iota(X)=1$, and \mbox{$\g(X)=9$~or~$\g(X)=10$.}
If~$X$ contains a line of type $(1,-2)$, then $X$ is cylindrical.
\end{theorem}

\begin{proof}
Let $\ell$ be a line in the~Fano threefold $X$, and let $\sigma\colon\widetilde{X}\to X$ be the~blowup of the~line $\ell$.
Then it follows from \cite{Iskovskikh-1980-Anticanonical,Iskovskikh1990,Pr91,IP99} that there is the~Sarkisov link:
$$
\xymatrix{
&\widetilde{X}\ar[dl]_{\sigma}\ar@{-->}[rr]^{\chi}&&\widehat{X}\ar[dr]^{\varphi}&\\
X &&&& Y,}
$$
where $Y$ is a~smooth Fano threefold described below,
the~morphism $\varphi$ is the~blowup of a~smooth rational curve $\Gamma$,
and $\chi$ is a~composition of flops of the~proper transforms of the~lines that meet $\ell$. Moreover, we have the~following options:
\begin{itemize}
\item if $\g(X)=9$, then $Y=\PP^3$, and $\Gamma$ is a curve of degree $7$ and genus $3$;
\item if $\g(X)=10$, then $Y$ is a~smooth quadric in $\PP^4$, and $\Gamma$ is a curve of degree $7$ and genus $2$.
\end{itemize}

Let $E$ be the~$\sigma$-exceptional surface, let $\widehat{E}$ be its proper transform on $\widehat{X}$, and let $\mathscr{S}=\varphi(\widehat{E})$.
Then $\mathscr{S}$ is a~(maybe singular or non-normal) del Pezzo surface of degree $\g(X)-3$ that contains~$\Gamma$.
Similarly, let $S$ be the~proper transform of the~$\phi$-exceptional surface on the~Fano threefold $X$.
Then $S$ is a~hyperplane section of $X$ such that $\mathrm{mult}_{\ell}(S)=3$. Using this, we conclude that
$$
X\setminus S\cong Y\setminus \mathscr{S}.
$$
Moreover, if $\ell$ is a line of type $(1,-2)$, then the~surface $\mathscr{S}$ is not normal.
This implies that the~complement $Y\setminus \mathscr{S}$ contains a cylinder, so that $X$ is cylindrical.
\end{proof}

In fact, we believe that the~following is true:

\begin{conjecture}
\label{conjecture:g-9-10}
Let $X$ be a very general smooth Fano threefold such that $\uprho(X)=1$, $\iota(X)=1$, and \mbox{$\g(X)=9$~or~$\g(X)=10$.}
Then $X$ is not cylindrical.
\end{conjecture}

Using a similar Sarkisov link as in the~proof of Theorem~\ref{theorem:g-9-10-12}, we obtain

\begin{theorem}[{\cite{KPZ11}}]
\label{theorem:g-12}
Suppose that $\uprho(X)=1$, $\iota(X)=1$ and $\g(X)=12$. Then $X$ is cylindrical.
\end{theorem}

\begin{proof}
Let $\ell$ be a line in $X$. Then there exists a~unique surface $S\in|-K_X|$ such that $\mathrm{mult}_{\ell}(S)=3$.
Moreover, it follows from \cite{Iskovskikh-1980-Anticanonical,Iskovskikh1990,Pr91,IP99} that there exists the~following Sarkisov link:
\begin{equation}
\label{eq:slV22}
\vcenter{\xymatrix{
&\widetilde{X}\ar[dl]_{\sigma}\ar@{-->}[rr]^{\chi}&&\widehat{X}\ar[dr]^{\varphi}&\\
X &&&& V_5,}}
\end{equation}
where $\sigma$ is the~blowup of the~line $\ell$, the~variety $V_5$ is a~smooth quintic del Pezzo threefold in $\mathbb{P}^6$,
the~morphism $\varphi$ is the~blowup of a~rational quintic curve $\Gamma$,
and $\chi$ is a~composition of flops.

Let $E$ be the~$\sigma$-exceptional surface, let $\widehat{E}$ be its proper transform on $\widehat{X}$, and let $\mathscr{S}=\varphi(\widehat{E})$.
Then $\mathscr{S}$ is a~hyperplane section of the~threefold $V_5$ that contains the~curve $\Gamma$,
and $S$ is the~proper transform of the~$\varphi$-exceptional surface.
Moreover, we have
$$
X\setminus S\cong V_5\setminus \mathscr{S}.
$$
Let us show that $V_5\setminus\mathscr{S}$ contains a cylinder.
In fact, this follows from the~proof of Theorem~\ref{theorem:V5-A3}.
We will use the~notation and assumptions introduced in this proof.

Let $L$ be a line in $V_5$ that is contained in $\mathscr{S}$ (it does exists).
If $\mathscr{S}\ne H_L$, let $\mathrm{S}$ be the~proper transform on $Q$ of the~surface $\mathscr{S}$.
Otherwise, we let $\mathrm{S}=H_C$. Then the~surface $\mathrm{S}$ is a hyperplane section of the~quadric $Q$.
Thus, we see that
$$
V_5\setminus\big(\mathscr{S}\cup H_L\big)\cong Q\setminus\big(\mathrm{S}\cup H_C\big).
$$
Now taking the linear projection $Q\dasharrow\mathbb{P}^3$ from a sufficiently general point in $\mathrm{S}\cap H_C$,
one can easily show that the~complement $Q\setminus(\mathrm{S}\cup H_C)$ contains a cylinder, so~that $X$ is cylindrical.
\end{proof}

\begin{remark}[{\cite{Pr91}}]
\label{remark:g-12}
In the~notation and assumptions of the~proof of  Theorem~\ref{theorem:g-12}, let~$\ell$ be a line of type  $(-1,2)$.
Then $\mathscr{S}$ is a non-normal surface whose singular locus is a line in $V_5$.
Letting $L$ to be this line gives $\mathscr{S}=H_L$, so that
$$
X\setminus S\cong V_5\setminus \mathscr{S}\cong Q\setminus H_C.
$$
Thus, if we also have $\NNN_{L/V_5}\cong\OOO_L(1)\oplus \OOO_L(-1)$,
then $H_C$ is singular (see the~proof of Theorem~\ref{theorem:V5-A3}), so that $X\setminus S\cong\A^3$.
We can always find such $\ell$ and $L$ if $\mathrm{Aut}(X)$ is infinite (see Theorem~\ref{theorem:Fano-threefolds-Aut-infinite}).
\end{remark}

We do not know examples of cylindrical smooth Fano threefolds of Picard rank $1$ and genus~$7$.
In fact, we believe that any such threefold is not cylindrical.

\begin{conjecture}
\label{conjecture:g-7}
Let $X$ be a smooth Fano threefold such that $\uprho(X)=1$, $\iota(X)=1$, and \mbox{$\g(X)=7$.}
Then $X$ is not cylindrical.
\end{conjecture}

Before we close this section, let us mention that most of smooth Fano threefolds with~\mbox{$\uprho(X)\geqslant 2$} are
rational \cite{Iskovskikh1997,IP99,Prokhorov2019},
and many of them are known to be cylindrical.
However, we do not know the~existence of anticanonical polar cylinders in majority of cylindrical smooth Fano threefolds.
Let~us list few examples.

\begin{example}
\label{example:blow-up-dP-n}
Let $Y$ be a smooth Fano threefold such that $Y$ is a del Pezzo threefold or $Y=\mathbb{P}^3$.
Take~$H\in\mathrm{Pic}(Y)$ on $Y$ such that $-K_Y\sim 2H$.
Choose a smooth curve $\mathscr{C}\subset Y$ that is a~complete intersection of two surfaces from~$|H|$.
Suppose that $X$ is a blowup of the~threefold $Y$ along  $\mathscr{C}$.
Then $X$ is a smooth Fano threefold.
Moreover, if $H^3\geqslant 4$, then $X$ is cylindrical.
\end{example}

\begin{example}
\label{example:blow-up-P3-two-cubic}
Suppose that $X$ is a blowup of $\mathbb{P}^3$ along a smooth curve
that is a complete intersection of two cubic surfaces.
Then $X$ is a cylindrical smooth Fano threefold.
\end{example}

\begin{example}
\label{example:blow-up-P3-g-3-d-6}
Suppose that $X$ is a blowup of $\mathbb{P}^3$ along a~smooth curve of degree $6$ and genus~$3$,
which is an~intersection of cubic hypersurfaces. Then $X$ is a cylindrical smooth Fano~threefold.
\end{example}

\begin{example}
\label{example:blow-up-Q-three-quadrics}
Let $Q$ be a smooth quadric threefold in $\mathbb{P}^4$, and let $H$ be its hyperplane section.
Suppose that $X$ is a blowup of $Q$ along a smooth curve
that is a complete intersection of two surfaces from $|2H|$.
Then $X$ is a cylindrical smooth Fano threefold.
 \end{example}

Each smooth Fano threefold described in Examples~\ref{example:blow-up-dP-n},
\ref{example:blow-up-P3-two-cubic}, \ref{example:blow-up-P3-g-3-d-6} and \ref{example:blow-up-Q-three-quadrics} is cylindrical,
but we do not know whether any of these threefolds contains anticanonical polar cylinders or not.

\subsection{Cylindrical Fano fourfolds}
\label{subsection:Fano-4-folds}
Now, let $X$ be a smooth Fano fourfold such that $\uprho(X)=1$.
By Corollary~\ref{corollary:Fano-fourfolds} we have the~following implications:
$$
\text{$X$ is cylindrical}\Longrightarrow\text{$X$ is rational}.
$$
If~$\iota(X)=5$ or $\iota(X)=4$, then $X=\mathbb{P}^4$ or $X$ is a smooth quadric fourfold, so that $X$ is cylindrical.
Similarly, if~$\iota(X)=3$, then it follows from Remark~\ref{remark:del-Pezzo-n-folds} that $X$ is one of the~following fourfolds:
\begin{enumerate}
\item a smooth sextic hypersurface in $\PP(1,1,1,1,2,3)$;
\item a smooth quartic hypersurface in $\PP(1,1,1,1,1,2)$;
\item a smooth cubic fourfold in $\PP^{5}$;
\item a smooth complete intersection of two quadrics in $\PP^6$;
\item the~quintic del Pezzo fourfold described in Example~\ref{example:V5}.
\end{enumerate}
In the~first two cases, we expect that $X$ is always irrational.
In fact, we know that a~very general quartic hypersurface in $\PP(1,1,1,1,1,2)$ is irrational \cite{HassettPirutkaTschinkel},
so that it is definitely not cylindrical.
Similarly,  general cubic fourfold in $\PP^5$ is expected to be irrational.
But there are rational smooth cubic fourfolds (see \cite{Tregub1984,Hassett1999,Hassett,RussoStagliano}), so that it is very natural to ask the~following question:

\begin{question}
\label{question:cubic-fourfolds}
Are there smooth rational cylindrical cubic fourfolds?
\end{question}

\begin{remark}
\label{remark:Fermat-cubics}
Every smooth cubic fourfold in $\mathbb{P}^5$ containing two skew planes is rational (see \cite{Hassett}).
In particular, the~Fermat cubic fourfold is rational.
If it is cylindrical, then the~affine cone over it admits an effective action of the~group $\Ga$ by Theorem~\ref{theorem:criterion},
which contradicts Conjecture~\ref{conjecture:Pham--Brieskorn}.
\end{remark}

By Lemma~\ref{lemma:dP4}, we know that a smooth complete intersection of two quadrics in $\PP^6$ is cylindrical.
Let us prove that the~quintic del Pezzo fourfold described in Example~\ref{example:V5} is cylindrical as well.
To~do this, let us present a~detailed description of this fourfold given in \cite{Pr94}.

Let $V_5$ be the~quintic del Pezzo fourfold in $\PP^7$.
By \cite[Theorem 6.6]{Piontkowski1999}, we have the~following exact sequence of groups:
$$
1\longrightarrow (\Ga)^4\rtimes\Gm \longrightarrow \Aut(V_5)\longrightarrow\PGL_2(\CC)\longrightarrow 1,
$$
so that the~group $\Aut(V_5)$ is not reductive. In particular, the~fourfold $V_5$ is not K-polystable \cite{AlperBlumHalpernLeistnerXu}.
The planes on $V_5$ belong to one of the~following two classes:
\begin{itemize}
\item[(i)] a~unique plane $\Xi$ which is a Schubert variety of type $\upsigma_{2,2}$;
\item[(ii)] a~one-parameter family of planes $\Pi_t$ that are Schubert varieties of type $\upsigma_{3,1}$.
\end{itemize}
We say that $\Xi$ is the~plane of type $\upsigma_{2,2}$, and $\Pi_t$ are planes of type $\upsigma_{3,1}$.
They are distinguished~by the~types of the~normal bundles: $\mathrm{c}_2(\NNN_{\Xi/X})=2$ and $\mathrm{c}_2(\NNN_{\Pi_t/X})=1$.
Moreover, there is a hyperplane section $\mathcal{H}$ of the~fourfold $V_5$ that contains all planes in $V_5$.
Furthermore, one has $\mathrm{Sing}(\mathcal{H})=\Xi$,
the~threefold $\mathcal{H}$ is the~union of all the~$\upsigma_{3,1}$-planes in $V_5$,
and $\Xi$ contains a special conic $\mathcal{C}$ such that
\begin{itemize}
\item the~intersection $\Pi_t\cap\Xi$ is a~tangent line to the~conic $\mathcal{C}$;
\item two distinct $\upsigma_{3,1}$-planes $\Pi_{t_1}$ and $\Pi_{t_2}$ meet in a~point in $\Xi\setminus\mathcal{C}$.
\end{itemize}
The automorphism group $\Aut(V_5)$ has the~following orbits in $V_5$:
\begin{enumerate}
\item the~open orbit $X\setminus \mathcal{H}$;
\item the~three-dimensional orbit $\mathcal{H}\setminus \Xi$;
\item the~two-dimensional orbit $\Xi\setminus \mathcal{C}$;
\item the~one-dimensional closed orbit $\mathcal{C}$.
\end{enumerate}
The Hilbert scheme of lines on the~del Pezzo fourfold $V_5$ is smooth, irreducible, and four-dimensional.
Moreover, if $\ell$ is a line in $V_5$, then $\ell$ belongs to one the~following five classes:
\begin{enumerate}
\renewcommand\labelenumi{(\alph{enumi})}
\renewcommand\theenumi{(\alph{enumi})}
\item\label{line:type:a}
$\ell\not\subset\mathcal{H}$, $\ell\cap \Xi=\emptyset$, and $l\cap\mathcal{H}$ is a point;
\item\label{line:type:b}
$\ell\subset \mathcal{H}$, $l\cap \Xi$ is a point, and $\ell\cap\mathcal{C}=\emptyset$;
\item\label{line:type:c}
$\ell\subset \mathcal{H}$, and $l\cap \Xi=l\cap \mathcal{C}$ is a point;
\item\label{line:type:d}
$\ell\subset \Xi$, and the~intersection $\ell\cap \mathcal{C}$ consists of two points;
\item\label{line:type:e}
$\ell\subset \Xi$ and $\ell$ is tangent to $\mathcal{C}$.
\end{enumerate}
The~group $\Aut(V_5)$ acts transitively on the~lines in each of these classes.
For a line $\ell\subset V_5$, the~lines meeting $\ell$ sweep out a hyperplane section $H_\ell$ of the~fourfold $V_5$ that is singular along the~line $\ell$.
Vice versa, if $H$ is a~hyperplane section of the~quintic del Pezzo fourfold $V_5$ that has non-isolated singularities, then $H=H_\ell$ for some line $\ell\subset V_5$.

\begin{theorem}[\cite{Pr94}]
\label{theorem:V5-A4}
Let  $\ell$ be a line in $V_5$ that is  not a line of type \textup{\ref{line:type:b}}. Then $V_5\setminus H_\ell\cong \A^4$.
\end{theorem}

\begin{proof}
If $\ell$ is a line of type \textup{\ref{line:type:d}} or \textup{\ref{line:type:e}}, then $H_\ell=\mathcal{H}$.
On the~other hand, there exists the~following  $\Aut(V_5)$-equivariant Sarkisov link:
$$
\xymatrix{
&\widetilde{V}_5\ar[dl]_{\sigma}\ar[dr]^{\varphi}&\\
V_5\ar@{-->}[rr]^{\psi}&& \PP^4}
$$
where $\sigma$ is the~blowup of the~plane $\Xi$, $\varphi$ is the~blowup of a~twisted cubic curve $C$, and $\psi$ is the~linear projection from $\Xi$.
Then the~$\varphi$-exceptional divisor is the~proper transform of the~threefold $\mathcal{H}$.
Moreover, if $E$ is the~$\sigma$-exceptional divisor, then $\varphi(E)$ is the~hyperplane in $\PP^4$ that contains $C$.
Thus, if $\ell$ is a line of type \textup{\ref{line:type:d}} or \textup{\ref{line:type:e}}, then $V_5\setminus H_\ell=V_5\setminus \mathcal{H}\cong \PP^4\setminus \varphi(E)\cong\A^5$.

Let $\pi\colon \widehat{V}_5$ be the~blowup of the~line $\ell$.
Then there exists the~following Sarkisov link:
\begin{equation}
\label{equation:V5-Sarkisov-link-2}
\vcenter{\xymatrix{
&\widehat{V}_5\ar[dl]_{\pi}\ar[dr]^{\eta}&\\
V_5\ar@{-->}[rr]^{\zeta}&& Q}}
\end{equation}
where $Q$ is an~irreducible quadric in $\PP^5$,
the map $\zeta$ is the~projection from $\ell$,
and $\eta$ is a birational morphism that contracts the~proper transform of the~hyperplane section $H_\ell$ to a surface of degree~$3$.
Let $\widehat{H}_\ell$ be the~proper transform on $\widehat{V}_5$ of the~threefold $H_\ell$, and let $F$ be the~$\pi$-exceptional divisor.
Then $V_5\setminus H_\ell\cong Q\setminus \eta(F)$,
and $\eta(F)$ is a singular hyperplane section of the~quadric $Q$.

If $\ell$ is a~line of type \textup{\ref{line:type:a}}, then all fibers of $\eta$ are one-dimensional,
so that $Q$ is smooth (see \cite{Ando:85:ext-rays}).
Thus, in this case, we have $V_5\setminus H_\ell\cong Q\setminus \eta(F)\cong \A^4$.

To complete the~proof, we may assume that $\ell$ is of type \textup{\ref{line:type:c}}.
Then $\ell$ is contained in a~plane in $V_5$, so that $\eta$ has a two-dimensional fiber.
Hence, in this case, the~quadric $Q$ can be singular~(cf.~\cite{Andreatta-Wisniewski:contr-1}).
Analyzing the~situation more carefully,  we see that $V_5\setminus H_\ell\cong Q\setminus \eta(F)\cong \A^4$.
\end{proof}

\begin{corollary}
\label{corollary:V5-A4}
The quintic del Pezzo fourfold is cylindrical.
\end{corollary}

In the~remaining part of this subsection, we present known constructions of cylinders in some smooth Fano--Mukai fourfolds.
Basically, our main goal is to explain how to prove Theorem~\ref{theorem:g-7-8-9-10-fourfolds}.
Thus, we suppose that $X$ is a smooth Fano--Mukai fourfold, $\uprho(X)=1$ and $\g(X)\in\{7,8,9,10\}$.

Let $H$ be an ample Cartier divisor on $X$ such that
$$
-K_X\sim 2H.
$$
Then $H^4=2\g(X)-2\in\{12,14,16,18\}$. Moreover, the~divisor $H$ is very ample, and the~linear system $|H|$ gives an embedding $X\hookrightarrow \PP^{\g(X)+2}$.
Let us deal with four cases separately.

If $\g(X)=10$, then $X=X_{18}$ is a hyperplane section of the~homogeneous fivefold $G_2/P\subset \PP^{13}$,
where $G_2$ is the~simple algebraic group of exceptional type $\mathrm{G_2}$,
and $P$ is its~parabolic subgroup that corresponds to a~short root (see \cite{Mukai-1988,Mu89}).
The family $\mathfrak{X}$ of all such fourfolds is one-dimensional.
Moreover, if $X=X_{18}$ is a general member of $\mathfrak{X}$, then $\Aut(X)\cong\Gm^2\rtimes\mumu_2$.
Besides, there are three distinguished fourfolds in this family:
\begin{enumerate}
\item[(0)] $X_{18}^{\mathrm{r}}$ such that $\Aut(X_{18}^{\mathrm{r}})\cong\Gm^2\rtimes\mumu_6$;
\item[(1)] $X_{18}^{\mathrm{s}}$ such that $\Aut(X_{18}^{\mathrm{s}})\cong\GL_2(\CC)\rtimes\mumu_2$;
\item[(2)] $X_{18}^{\mathrm{a}}$ such that $\Aut(X_{18}^{\mathrm{a}})\cong (\Ga\times\Gm)\rtimes\mumu_2$.
\end{enumerate}
See \cite{PZ18} for details, where the~following result has been proved:

\begin{theorem}[{\cite{PZ18}}]
\label{theorem:FM-g-10}
Let $X$ be a smooth Fano--Mukai fourfold in $\PP^{12}$ of genus $10$ with \mbox{$\uprho(X)=1$}.
Then there exists an~$\Aut^0(X)$-invariant hyperplane section $H$ of $X$ such that
the~complement $X\setminus H$ is $\Aut^0(X)$-equivariantly isomorphic to $\A^4$.
\end{theorem}

This theorem implies, in particular, that any smooth Fano--Mukai fourfolds of genus $10$ is cylindrical.
See also Example~\ref{example:Fano-foufolds} for another application of Theorem~\ref{theorem:FM-g-10}.

If  $\g(X)=8$, then $X=X_{14}$ is a~section of the~Grassmannian $\Gr(2,6)\subset\PP^{14}$ by a~linear subspace of dimension $10$ (see \cite{Mukai-1988,Mu89}).
Some of these fourfolds are cylindrical.

\begin{example}[{\cite{PZ16}}]
\label{example:FM-g-8}
Suppose that $\g(X)=8$ and $X$ contains a~plane $\Pi$ which is a Schubert variety of type $\upsigma_{4,2}$,
and $X$ does not contain planes meeting $\Pi$ along a~line.
Such fourfolds do exist and form a~subspace of codimension one in the~moduli space of all Fano--Mukai fourfolds of genus~$8$.
Then it follows from \cite{Prokhorov-1993c} that there exists the~following Sarkisov link:
$$
\xymatrix{
&\widetilde{X}\ar[dl]_{\sigma}\ar[dr]^{\varphi}&\\
X && V_5}
$$
where $V_5$ is the~del Pezzo quintic fourfold in $\PP^7$ (see Theorem~\ref{theorem:V5-A4}),
$\sigma$ is the~blowup of the~plane~$\Pi$, and~$\varphi$ is the~blowup of a smooth rational surface $S$ of degree $7$ such that $K_S^2=3$.
Then
$$
X\setminus H_X\cong V_5\setminus H_{V_5},
$$
where $H_{V_5}$ is the~proper transform on $V_5$ of the~$\sigma$-exceptional divisor,
and $H_{X}$ is the~proper transform on $X$ of the~$\varphi$-exceptional divisor.
On the~other hand, the~divisor $H_{V_5}$ is a~hyperplane section of the~fourfold $V_5$ that contains $S$,
and $H_X$ is a~hyperplane section of $X$ containing $\Pi$.
Thus, the~set $V_5\setminus H_{V_5}$ contains a cylinder by \cite[Theorem~4.1]{PZ16}, so that $X$ is cylindrical.
\end{example}

If $\g(X)=7$, then $X=X_{12}$ is a~section of the~orthogonal Grassmannian $\operatorname{OGr}(4,9)\subset \PP^{15}$
by a~linear subspace of dimension $9$ (see \cite{Mukai-1988,Mu89}).
In this case, we also have cylindrical fourfolds.

\begin{example}[{\cite{PZ16}}]
\label{example:FM-g-7}
Suppose that $\g(X)=7$ and $X$ contains a~plane $\Pi$.
Such fourfolds do exist.
Suppose that $X$ is a~sufficiently general Fano--Mukai fourfold of genus $7$ that contains the~plane~$\Pi$.
Then by \cite{Prokhorov-1993c} there exists the Sarkisov link
$$
\xymatrix{
&\widetilde{X}\ar[dl]_{\sigma}\ar[dr]^{\varphi}&\\
X && V_4}
$$
where $V_4$ is a smooth complete intersection of two quadrics in $\PP^6$, $\sigma$ is the~blowup of the~plane~$\Pi$,
and $\varphi$ is the~blowup of a smooth del Pezzo surface $S$ such that $K_S^2=5$.
Arguing as in Example~\ref{example:FM-g-8}, we conclude that $X$ is cylindrical.
\end{example}

If $\g(X)=9$, then $X=X_{16}$ is a~section of the~Lagrangian Grassmannian $\operatorname{LGr}(3,6)\subset \PP^{13}$
by a~linear subspace of dimension $11$ (see \cite{Mukai-1988,Mu89}).
There are cylindrical fourfolds in this family.

\begin{example}[\cite{PZ17}]
\label{example:FM-g-9}
Suppose that $\g(X)=9$. Then $X_{16}$ contains an~irreducible two-dimensional quadric surface $S$.
Suppose, for simplicity, that $X_{16}$ is a general Fano--Mukai fourfold of genus $9$ that contains $S$.
Then there exists the~following Sarkisov link:
$$
\xymatrix{
&\widetilde{X}\ar[dl]_{\sigma}\ar[dr]^{\varphi}&\\
X && V_5}
$$
where $V_5$ is the~del Pezzo quintic fourfold, $\sigma$ is the~blowup of the~surface~$S$,
and~$\varphi$ is the~blowup along a smooth del Pezzo surface of degree $6$.
Arguing as in Example~\ref{example:FM-g-8}, we see that $X$ is~cylindrical.
\end{example}

The interested reader can consult also the recent preprint~\cite{HHT2021} for further examples of cylindrical Fano fourfolds.

\subsection{Cylinders in Mori fibrations}
\label{subsection:MFS}
This subsection is inspired by the~following

\begin{question}
\label{question:Adrien}
Given a~family of cylindrical varieties, when its total space is cylindrical?
\end{question}

For example, irrational three-dimensional conic bundles are not cylindrical, though their general fibers are.
In general, this question is very subtle and has birational nature, so that it is natural to consider it for Mori fibred spaces first.

Let $V$ be a projective variety with terminal $\mathbb{Q}$-factorial singularities,
let $\pi\colon V\to B$ be a dominant projective non-birational morphism such that $-K_V$ is $\pi$-ample, $\pi_*\mathcal{O}_V=\mathcal{O}_B$ and $\uprho(V)=\uprho(B)+1$.
Let~$X_{\eta}$ be the~fiber of the~morphism $\pi$ over the~(scheme-theoretic) generic point $\eta$ of the~base~$B$.
Then $X_{\eta}$ is a Fano variety that has at most terminal singularities, which is defined over $\mathbb{K}=\Bbbk(B)$, i.e. the~field of rational functions on $B$.
Over the~(algebraically non-closed) field $\mathbb{K}$, the~divisor class group of the~Fano variety $X_{\eta}$ is of rank $1$, because we assume that $\uprho(V)=\uprho(B)+1$.

\begin{definition}[{\cite{DK18}}]
\label{definition:vertical-cylinder}
If the~variety $V$ contains a (Zariski open) cylinder $U=\A^1\times Z$,
we say that the~cylinder $U$ is \emph{vertical} (with respect to~$\pi$)
if there is a~morphism $h\colon Z\to B$ such that the~restriction $\pi|_U\colon U\to B$
is a~composition $h\circ \pr_Z$, where $\pr_Z\colon U \to Z$ is the~natural projection.
In this case, we have commutative diagram:
\begin{equation}
\label{equation:vertical-cylinder}
\vcenter{\xymatrix{
\mathbb{A}^1\times Z=U\ar[d]^{p_Z}\ar@{^{(}->}[rr]&&V\ar@{->}^{\pi}[d]\\
Z\ar@{->}[rr]^h&& B}}
\end{equation}
A cylinder in $V$ which is not vertical is called \emph{twisted}.
\end{definition}

If $V$ contains a vertical cylinder  $U=\A^1\times Z$, then the~Fano variety $X_{\eta}$ contains a cylinder
$$
U_\eta=\A^1\times Z_\eta,
$$
where $U_\eta$ and $Z_\eta$ are generic (scheme) fibers of the~morphisms $h\circ \pr_Z$ and $h$ in \eqref{equation:vertical-cylinder}, respectively.
Vice versa, if the~Fano variety $X_{\eta}$ contains a cylinder defined over the~field $\mathbb{K}$, then $V$ does contain a~vertical cylinder by \cite[Lemma~3]{DK18}.
This gives a motivation to study cylinders in Fano varieties defined over arbitrary fields (cf. \cite{BW2019,HT2019,KP2019,KP2020})
The first step in this direction is

\begin{theorem}[{\cite{DK18}}]
\label{theorem:DP-cylinders}
Let $S$ be a geometrically irreducible smooth del Pezzo surface defined over a~field $\mathbb{F}$ of characteristic~$0$.
Suppose that $\uprho(S)=1$.
Then the~following conditions are equivalent:
\begin{itemize}
\item[(i)] the~surface $S$ contains a cylinder defined over $\mathbb{F}$;
\item[(ii)] the~surface $S$ is rational over $\mathbb{F}$;
\item[(iii)] $K_S^2\geqslant 5$ and $S$ has an $\mathbb{F}$-point.
\end{itemize}
\end{theorem}

\begin{proof}
It is commonly known that the~conditions (ii) and (iii) are equivalent (see, for example,~\cite{IskovskikhUMN}).
Moreover, the~implication (iii)$\Rightarrow$(i) can be shown using well-known Sarkisov links that start at~$S$, which are described in \cite{IskovskikhUMN}.
For details, see the~proof of \cite[Proposition~12]{DK18}.
Thus, we just have to show that (i) implies (iii).
This can also be shown using Sarkisov links, but we present another proof.

Suppose that $S$ contains a cylinder $U$ which is defined over~$\mathbb{F}$.
Then $U\cong\mathbb{A}^1\times Z$ for some affine curve $Z$ defined over $\mathbb{F}$.
Let $\overline{Z}$ be the~completion of the~curve~$Z$.
Then $\overline{Z}$ is a geometrically irreducible curve.
Moreover, we have the~following commutative diagram
$$
\xymatrix{\mathbb{P}^1\times\overline{Z}\ar[ddrr]^{\overline{p}_{2}}&\mathbb{A}^1\times\overline{Z}\ar@{_{(}->}[l]\ar[ddr]^{p_{2}}&~\mathbb{A}^1\times Z\cong U\ar@{_{(}->}[l]\ar[d]^{p_Z}\ar@{^{(}->}[r] &S\ar@{-->}^{\psi}[ddl]& &\widetilde{S}\ar[ll]_{\pi}\ar[ddlll]^{\phi}& \\
&&Z\ar@{_{(}->}[d]&&&& \\
& &\overline{Z}& & &&}
$$
where $p_Z$, $p_{2}$ and $\overline{p}_2$ are the~natural projections to the~second factors,
$\psi$ is the~ rational map induced by $p_Z$,
$\pi$ is a~birational morphism resolving the~indeterminacy of $\psi$ and $\phi$ is a~morphism.
By construction, a~general fiber of $\phi$ is isomorphic to $\mathbb{P}^1$.

Let $\Gamma$ be the~section of $\overline{p}_2$ that is the~complement of $\mathbb{A}^1\times\overline{Z}$ in $\mathbb{P}^1\times\overline{Z}$,
and let $\widetilde{\Gamma}$ be the~proper transform on $\widetilde{S}$ of the~curve $\Gamma$.
Then $\widetilde{\Gamma}\cong\Gamma\cong\overline{Z}$, the~curve $\widetilde{\Gamma}$ is a~section of $\phi$,
and the~curve $\widetilde{\Gamma}$ is $\pi$-exceptional, because $\uprho(S)=1$.
Let $P=\pi(\widetilde{\Gamma})$. Then $P$ is an~$\mathbb{F}$-point.

Now, we can proceed in two (slightly different) ways.
First, as in the~proof of \cite[Theorem 1]{DK18}, we can let $\mathcal{M}$ to be the~linear system on $S$ that gives the~map $\psi$.
Then, arguing as in Section~\ref{subsection:del-Pezzo-no-cylinders}, we conclude that $(S,\lambda\mathcal{M})$ is not log canonical at $P$
for a some $\lambda\in\mathbb{Q}_{>0}$ such that $\lambda\mathcal{M}\sim_{\mathbb{Q}} -K_S$.
Such number exists, since $\uprho(S)=1$.
Let $M_1$ and $M_2$ be two general curves in $\mathcal{M}$. Then
$$
\frac{K_S^2}{\lambda^2}=M_1\cdot M_2\geqslant\big(M_1\cdot M_2\big)_P>\frac{4}{\lambda^2}
$$
by \cite[Theorem~3.1]{Corti2000}. This gives $K_S^2\geqslant 5$, so that (i) implies (iii).

Alternatively, we can use Corollary~\ref{corollary:dP-du-Val}.
Let $C_1,\ldots,C_n$ be the~irreducible curves in $S$ that lie in the~complement $S\setminus U$.
Then we put $D=\lambda(C_1+\cdots+C_n)$ for $\lambda\in\mathbb{Q}_{>0}$ such that $D\sim_{\mathbb{Q}} -K_S$.
Therefore, we conclude that $S$ contains a $(-K_S)$-polar cylinder, so that $K_S^2\geqslant 4$ by Corollary~\ref{corollary:dP-du-Val}.
Thus, we may assume that $K_S^2=4$.
Then our point $P$ is not contained in any $(-1)$-curve in $S\otimes_{\mathbb{F}}\overline{\mathbb{F}}$, where $\overline{\mathbb{F}}$ is an algebraic closure of the~field $\mathbb{F}$.
Indeed, otherwise the~$\mathrm{Gal}(\overline{\mathbb{F}}/\mathbb{F})$-orbit of this curve
would consist of at least four $(-1)$-curves that all pass through the~point $P$, which is impossible.
Let $\xi\colon\widehat{S}\to S$ be the~blowup of the~point $P$, and let $E$ be the~exceptional curve of the~blowup~$\xi$.
Then $\widetilde{S}$ is a smooth del Pezzo surface of degree $K_{\widetilde{S}}^2=3$ and
$$
\widetilde{D}+\big(\mathrm{mult}_{P}(D)-1\big)E\sim_{\mathbb{Q}}-K_{\widehat{S}},
$$
where $\mathrm{mult}_{P}(D)>1$ by Remark~\ref{remark:obstruction} and Lemma~\ref{lemma:Skoda}. Then $\widetilde{S}$ contains a $(-K_{\widetilde{S}})$-polar cylinder,
which is impossible by Corollary~\ref{corollary:dP-du-Val}. This again shows that (i) implies (iii).
\end{proof}

\begin{corollary}[{\cite[Theorem 1]{DK18}}]
\label{corollary:DP-cylinders}
Suppose that $X_\eta$ is a del Pezzo surface.
Then $V$ contains a~vertical cylinder $\iff$ $K_{X_\eta}^2\geqslant 5$ and $\pi$ has a~rational section.
\end{corollary}

Note that if $\Bbbk$ is uncountable and the~general fiber of $\pi$ contains a cylinder,
then it follows from \cite{DK19b,Ka18} that the~total space of the~family $V\times_B B^\prime\to B$ contains a vertical cylinder for an appropriate finite base change $B^\prime\to B$.
This basically means that $X_\eta\otimes_{\mathbb{K}}\mathbb{K}^\prime$ contains a cylinder defined over $\mathbb{K}^\prime$ for an appropriate finite extension of fields $\mathbb{K}\subset\mathbb{K}^\prime$.

\begin{remark}
\label{remark:twisted-cylinders}
If  $X_\eta$ is a del Pezzo surface and $K_{X_\eta}^2\leqslant 4$, then  $V$ can contain twisted cylinders.
In~fact, there are three-dimensional examples constructed in \cite{DK17,DK18} such that $K_{X_\eta}^2\leqslant 3$,
$B=\mathbb{P}^1$, and $V$ contains a Zariski open subset isomorphic to $\A^3$.
See also {color{red}\cite{GMM14,DK16,Saw19a, Saw19b}.}
\end{remark}

Now let us mention one relevant result about forms of the~quintic del Pezzo threefold defined over a~non-algebraically closed field (cf. \cite[Theorem 3.3]{KP2019}).

\begin{theorem}[{\cite{DK19a}}]
\label{theorem:V5}
Let $X$ be a smooth Fano threefold defined over a~field $\mathbb{F}$ of characteristic~$0$.
Suppose that $X\otimes_{\mathbb{F}}\overline{\mathbb{F}}\cong V_5$, where $V_5$ is the~quintic del Pezzo threefold described in Example~\ref{example:V5},
where $\overline{\mathbb{F}}$ is the~algebraic closure of the~field $\mathbb{F}$. Then the~following assertions hold:
\begin{itemize}
\item $X$ contains a Zariski open subset $U\cong\mathbb{A}^2\times Z$ for some affine curve $Z$;
\item $X$ contains a Zariski open subset isomorphic to $\A^3$ if and only if $X$ contains a smooth rational curve $\ell$ defined over $\mathbb{F}$
such that $-K_{V_5}\cdot\ell=2$ and $\NNN_{\ell/X} \cong \OOO_{\ell} (-1) \oplus \OOO_{\ell}(1)$.
\end{itemize}
\end{theorem}

Let us conclude this section with the~following generalization of Theorem~\ref{theorem:g-7-8-9-10}.

\begin{theorem}[{\cite{KP2020}}]
\label{theorem:g-7-8-9-10-non-closed}
Let $X$ be a smooth Fano threefold defined over a~field $\mathbb{F}$ of characteristic~$0$.
Suppose that $X\otimes_{\mathbb{F}}\overline{\mathbb{F}}\cong X_{2g-2}$, where $X_{2g-2}$ is a Fano--Mukai variety of genus $g$ with $\uprho(X_{2g-2})=1$,
where $\overline{\mathbb{F}}$ be the~algebraic closure of $\mathbb{F}$.
Suppose  that the~following conditions hold:
\begin{enumerate}
\item $\dim(X)\geqslant 5$;
\item  $g\in\{7,8,9,10\}$;
\item $X$ has an $\mathbb{F}$-point.
\end{enumerate}
Then $X$ is cylindrical over $\mathbb{F}$.
\end{theorem}

\section{Beyond cylindricity}
\label{section:related}

\subsection{Flexible affine varieties}
\label{subsection:flexibility}
Let $X$ be an~affine variety.
Given a~$\Ga$-action on $X$, it induces a~representation of the~group $\Ga$ on the~structure $\Bbbk$-algebra $\OOO(X)$ of the~form
$$
(t,f)\longmapsto \exp(t\partial)(f)
$$
for $t\in\Ga$ and $f\in\OOO(X)$, where the~infinitesimal generator $\partial$ of the~$\Ga$-subgroup is a~\emph{locally nilpotent derivation} of $\OOO(X)$,
which means that every element $f\in\OOO(X)$ is annihilated by $\partial^{(m)}$~for some sufficiently large $m$ that depends on the~element $f$.
Conversely, any locally nilpotent derivation of the~$\Bbbk$-algebra $\OOO(X)$ generates a~$\Ga$-action on $X$ (see \cite{Fr06}).

Recall that the~derivations of $\OOO(X)$ correspond to the~regular vector fields on $X$.
We say that a~vector field on $X$ is \emph{locally nilpotent} if the~corresponding derivation is.

 If an open variety $X$ admits a~$\Ga$-action, then the log-Kodaira dimension of $X$ is negative. However, the converse does not hold, in general. Indeed, there are smooth affine surfaces of negative log-Kodaira dimension which admit no effective $\Ga$-action \cite{GMMR18}. Let us stay on this in more detail.

As we mentioned already, any smooth affine surface $X$ of negative log-Kodaira dimension contains a cylinder \cite[Ch. 2, Theorem 2.1.1]{Miyanishi2000}. Moreover, $X$ is affine-ruled, that is, there is a morphism $X\to C$ onto a smooth curve $C$ with general fiber $\A^1$. The base curve $C$ could be affine or projective. However, a smooth affine surface $X$ admits an effective $\Ga$-action if and only if it admits an $\A^1$-ruling $X\to C$ over an affine curve $C$, or, which is equivalent, a \emph{principal} cylinder, see Theorem~\ref{thm:affine-cylindricity}. In \cite{GMMR18}
there are examples of smooth rational affine surfaces $\A^1$-ruled over $\PP^1$ and with no $\A^1$-ruling over an affine curve. Hence, such a surface admits no effective $\Ga$-action.

To construct such a surface $X$ we start with the quadric $\PP^1\times\PP^1$ endowed with the first projection to $\PP^1$. We blow up three distinct points on the section  $S=\PP^1\times\{\infty\}$ and infinitesimally near points in such a way that each of the 3 resulting reducible fibers has a unique $(-1)$-component of multiplicity 2 and the union of the section $S$ and the remaining components of the reducible fibers forms  a connected divisor $D$. The complement of $D$ in the resulting projective surface is a smooth affine surface $X$. It comes equipped with an $\A^1$-fibration $X\to \PP^1$. Each fiber of this fibration is irreducible, and three of them are multiple of multiplicity 2. According to \cite[Theorem~4.1]{GM2005}, such a surface $X$ does not carry any $\A^1$-fibration over an affine curve. Hence, $X$ admits  no $\Ga$-action.

\begin{definition}
\label{definition:flexibility}
A point $P\in X$ is said to be \emph{flexible} if locally nilpotent vector fields on $X$ span the~tangent space $T_{P}X$.
The variety $X$ is said to be \emph{flexible} if every smooth point of $X$ is flexible.
We~also say that $X$ is \emph{generically flexible} if every point in a~non-empty Zariski open subset of $X$ is flexible.
\end{definition}

Let $\SAut(X)$ be the~subgroup of $\Aut(X)$ generated by all the~$\Ga$-subgroups.
The flexibility of $X$ is ultimately related to the transitivity of the~action of the~group $\SAut(X)$.
Indeed, we have the following criteria of flexibility.

\begin{theorem}[{\cite{AFKKZ13}}]
\label{theorem:flexibility}
Suppose that $\mathrm{dim}(X)\geqslant 2$. Then the~following conditions are equivalent:
\begin{enumerate}
\item the~variety $X$ is flexible;
\item the~group $\SAut(X)$ acts transitively on the~smooth locus of $X$;
\item the~group $\SAut(X)$ acts  highly transitively on the~smooth locus of $X$.
\end{enumerate}
\end{theorem}

One says that a group acts \emph{highly transitively} \footnote{Or \emph{infinitely transitively} in another terminology}. on an infinite set if it acts $m$-transitively for any natural number $m$.

\begin{remark}
\label{remark:algebraic-group-transiively}
A~dimension count shows that an~algebraic group cannot act highly transitively on an affine variety.
Moreover, it cannot act even $3$-transitively on an~affine variety \cite{Borel1953,Knop1983}.
\end{remark}

Let us present examples of flexible affine varieties;   see e.g. \cite{AZ21, AKZ12, AKZ19} for further examples.

\begin{example}
\label{example:An-flexible}
Let $X=\A^n$, where $n\geqslant 2$. Then the~subgroup of translations in $\SAut(\A^n)$ acts transitively on the~variety $X$,
so that $X$ is flexible by Theorem~\ref{theorem:flexibility} (cf. \cite{KZ99}).
\end{example}

\begin{example}
\label{example:KuyumzhiyanAndrist}
Let $X$ be the~$n$th Calogero--Moser space defined as follows:
$$
\Big\{(A,B)\in\,\mathrm{Mat}_n(\Bbbk)\times \mathrm{Mat}_n(\Bbbk)\ \big\vert\ \mathrm{rk}\big([A,B]+I_n\big)=1\Big\}/\!\!/\PGL_n(\Bbbk),
$$
where $\PGL_n(\Bbbk)$ acts via $g.(A,B)=(gAg^{-1}, gBg^{-1})$.
Then $X$ is a smooth rational irreducible~affine algebraic variety of dimension~$2n$ \cite{Popov2014,Wilson1998},
and it follows from \cite{BerestEshmatovEshmatov,Kuyumzhiyan2020} that $\mathrm{Aut}(X)$ acts highly transitively on $X$ for every $n\geqslant 1$.
Moreover, the~variety $X$ is flexible by \cite[Proposition~2.9]{Andrist2020}.
\end{example}

There are several constructions producing new flexible varieties from given ones (see~\cite{AFKKZ13,AKZ12,FKZ16,KZ99}).
For~instance, the~product of flexible varieties is flexible. Some further examples of flexible varieties are as follows.

\begin{example}
\label{example:flexible-varieties}
Suppose that $X$ is an~affine $G$-variety of dimension $\geqslant 2$, where $G$ is a~connected linear algebraic group that acts on $X$ with an~open orbit.
Then $X$ is flexible in the~following cases:
\begin{itemize}
\item $X$ is a~normal toric variety with no torus factor \cite[Theorem~0.2.2]{AKZ12};
\item $X=G/H$ is a~homogeneous space and $G$ has no nontrivial character \cite[Theorem~5.4]{AFKKZ13};
\item $X$ is smooth and $G$ is semisimple \cite[Theorem~5.6]{AFKKZ13};
\item $X$ is smooth with only constant invertible functions and $G$ is reductive \cite[Theorem~2]{GS19};
\item
$X$ is normal and $G=\SL_2(\Bbbk)$ \cite[Theorem~5.7]{AFKKZ13};
\item
$X$ is normal horospherical and $G$ is semisimple \cite[Theorem~2]{Sh17};
\item
$X$ is normal horospherical with no non-constant invertible regular function \cite[Theorem~3]{GS19}.
\end{itemize}
\end{example}

See also~\cite{BoldyrevGaifullin, Dub2015, Dub2019, Gi18, Ka2019, KM2012}.

If we replace the~smooth locus of $X$ in Theorem~\ref{theorem:flexibility} by the~open orbit of the~group $\SAut(X)$,
we obtain a criterion for the~generic flexibility \cite{AFKKZ13}.
If $X$ contains $\A^n$ as a~principal Zariski open~set, then $X$ is generically flexible.
Generically flexible varieties are unirational, but they are not always stably rational (see \cite[Proposition~4.9]{Liendo2010} and \cite[Example 1.22]{Popov2011}).

\begin{example}
\label{example:Gizatulin}
Suppose that $X$ is a normal affine surface such that $X$ can be completed by a~simple normal crossing chain of rational curves.
Then $X$ is often called a \emph{Gizatullin surface}.
If $X\not\cong\A^1\times (\A^1\setminus\{0\})$,
then it is generically flexible \cite{Gi71}, but it is not necessarily flexible \cite{Kov15}.
\end{example}

Affine cones over cylindrical Fano varieties often provide examples of flexible affine varieties.

\begin{example}
\label{example:flag-var-cones}
Let $V=G/P$, where $G$ is a~semisimple algebraic group, and $P$ is its~parabolic subgroup.
Then $V$ is a smooth Fano variety.
Let $V\hookrightarrow\mathbb{P}^n$  be any projectively normal embedding,
and let $\widehat{V}$ be the~affine cone in $\A^{n+1}$ over~$V$.
If $\mathrm{dim}(V)\geqslant 2$, then $\widehat{V}$ is flexible by \cite[Theorem~1.1]{AKZ12}.
\end{example}

To~explain why this is the~case, let us present two explicit criteria of flexibility of affine cones.
To do this, fix a~smooth projective variety $V$. Let $H$ be a very ample divisor on the~variety $X$.
Then~the~linear system $|H|$ gives an embedding $V\hookrightarrow\mathbb{P}^n$. Let $\widehat{V}$ be the~affine cone in $\A^{n+1}$ over~$V$.
We are interested in the~case when $V$ is a smooth cylindrical Fano variety.

If the~variety $V$ is uniformly cylindrical, then each point of $V$ is contained in a cylinder, so that the~variety $V$ admits a covering
\begin{equation}
\label{equation:covering}
V=\bigcup_{i\in I} U_i,
\end{equation}
where each $U_i$ is a Zariski open subset in $V$ such that $U_i\cong \A^1\times Z_i$ for some affine variety $Z_i$.
In~this case, a~subset $Y\subset V$ is said to be \emph{invariant} with respect to a~cylinder $U_i$ if
$$
Y\cap U_i=\pi_i^{-1}\big(\pi_i(Y\cap U_i)\big),
$$
where $\pi_i\colon U_i \to Z_i$ is the~natural projection.

\begin{definition}
\label{definition:transversal cylindres}
If $V$ is uniformly cylindrical, then we say that the~covering \eqref{equation:covering} is \emph{transversal}
if no proper subset $Y\subset X$ is invariant with respect to every cylinder $U_i$ in the~covering \eqref{equation:covering}.
\end{definition}

Now, we are ready to state the~first flexibility criterion for affine cones.

\begin{theorem}[{\cite{Pe13}}]
\label{theorem:Perepechko}
Suppose that $V$ is uniformly cylindrical and has a covering \eqref{equation:covering} such~that
\begin{itemize}
\item[(i)] the~covering \eqref{equation:covering} is transversal;
\item[(ii)] each cylinder in the~covering \eqref{equation:covering} is $H$-polar.
\end{itemize}
Then the~affine cone $\widehat{V}$ is flexible.
\end{theorem}

The second useful criterion is given by the~following

\begin{theorem}[{\cite{MPS18}}]
\label{theorem:criterion-MPS}
The~affine cone $\widehat{V}$ is flexible if the~variety $V$ is uniformly cylindrical and admits a covering
$$
V=\bigcup_{j\in J} W_j
$$
where each $W_j$ is a flexible affine Zariski open subset in $V$ such that $W_j=V\setminus\Supp{D_j}$
for some effective $\QQ$-divisor $D_j$ on the~variety $V$ that satisfies $D_j\sim_{\QQ} H$.
\end{theorem}

Using these criteria and the~proof of Lemma~\ref{lemma:KPZ}, one can prove the~following result:

\begin{theorem}[\cite{Pe13,PW16}]
\label{theorem:flexible-DP}
Suppose that $V$ is a~smooth del Pezzo surface such that $K_V^2\geqslant 4$.
Then the~affine cone $\widehat{V}$ is flexible for every very ample divisor $H$ on the~surface $V$.
\end{theorem}

Unfortunately, we cannot apply Theorems~\ref{theorem:Perepechko} and \ref{theorem:criterion-MPS}
to the~affine cone in $\A^4$ over a smooth cubic surface in $\mathbb{P}^3$,
simply because its anticanonical divisor is not cylindrical by Corollary~\ref{corollary:dP-du-Val}.
On the~other hand, in this case, we know from Theorem~\ref{theorem:Amp-cyl-del-Pezzo}
that every ample $\mathbb{Q}$-divisor that is not a~multiple of the~anticanonical divisor is cylindrical.
Using this and the~construction of cylinders given in the~proof of Theorem~\ref{theorem:Amp-cyl-del-Pezzo},
Perepechko very recently proved the~following result:

\begin{theorem}[{\cite{Perepechko2020}}]
\label{theorem:Perepechko-new}
If $V$ is a~smooth cubic surface, then the~affine cone $\widehat{V}$ is generically flexible for every very ample divisor $H$ on the~surface $V$ such that $H\not\in\mathbb{Z}_{>0}[-K_V]$.
\end{theorem}

Now, let us consider the~flexibility of affine cones over some cylindrical smooth Fano threefolds.
Many of them are flexible by Theorem~\ref{theorem:criterion-MPS},
because the~underlying Fano threefolds admit covering like in Theorem~\ref{theorem:criterion-MPS} with each Zariski open subset $W_j$ isomorphic to $\A^3$.
A possibly non-complete list of such smooth Fano threefolds is given in \cite[Proposition~4]{APS14}.
This gives

\begin{corollary}
\label{corollary:Fano-threefods-PSL-2}
Suppose that $V$ is a~smooth Fano threefold admitting an~effective $\PSL_2(\Bbbk)$-action.
If  $\uprho(V)=1$,  then the~affine cone $\widehat{V}$ is flexible.
\end{corollary}

\begin{proof}
If $\uprho(V)=1$, then it follows from Theorem~\ref{theorem:Fano-threefolds-Aut-infinite} that one of the~following four cases occurs:
\begin{itemize}
\item[(i)] $V=\PP^3$;
\item[(ii)] $V$ is the~smooth quadric threefold in $\PP^4$;
\item[(iii)] $V$ is the~smooth quintic del Pezzo threefold $V_5\subset\mathbb{P}^6$ described in Example~\ref{example:V5};
\item[(iv)] $V$ is the~Mukai--Umemura threefold $X=X_{22}^{\mathrm{mu}}\subset \PP^{13}$.
\end{itemize}
We may assume that we are in the~case (iii)~or~(iv), because the~required assertion is clear in the~remaining cases.
Then it follows from the~proofs of Theorems~\ref{theorem:V5-A3} and~\ref{theorem:g-12} that $V$ contains
a~one-parameter family of hyperplane sections $H_\ell$ such that each $H_\ell$ is singular along a line $\ell$ and
$$
V\setminus H_\ell\cong \A^3.
$$
The group $\PSL_2(\Bbbk)$ acts transitively on this family.
So, to apply Theorem~\ref{theorem:criterion-MPS}, we need to check that the~intersection of all these hyperplane sections is empty.
Suppose that this is not the~case. Then this intersection is $\PSL_2(\Bbbk)$-invariant,
so that it contains a~closed $\PSL_2(\Bbbk)$-orbit of minimal dimension.
But the~variety $V$ does not contain $\PSL_2(\Bbbk)$-fixed points,
and the~only one-dimensional closed $\PSL_2(\Bbbk)$-orbit in $V$ is not contained in any hyperplane section singular along a~line.
\end{proof}

For more examples of smooth Fano threefolds with flexible affine cones, see \cite[Theorem~4.5]{MPS18}.
Now, let us present examples of smooth Fano fourfolds with flexible affine cones.

\begin{example}[{\cite{PZ18}}]
\label{example:Fano-foufolds}
It follows from Theorems~\ref{theorem:V5-A4} and~\ref{theorem:FM-g-10}
that the~following smooth cylindrical Fano fourfolds admit coverings by affine charts isomorphic to~$\A^4$:
\begin{enumerate}
\item the~quintic del Pezzo fourfold $V_{5}$ described in Example~\ref{example:V5} (see Theorem~\ref{theorem:V5-A4});
\item the~Fano--Mukai fourfold $X_{18}^{\mathrm{s}}$ of genus $10$ with $\Aut(X_{18}^{\mathrm{s}})\cong\GL_2(\CC)\rtimes\mumu_2$;
\item the~Fano--Mukai fourfolds  $X_{18}$ of genus $10$ with $\Aut^0(X_{18})\cong \Gm^2$ (there is a~one-parameter family of these, up to isomorphism).
\end{enumerate}
Hence, all of them have flexible affine cones.
\end{example}

By Theorem~\ref{theorem:FM-g-10}, every smooth Fano--Mukai fourfold in $\PP^{12}$ of genus $10$ contains a~Zariski open subset isomorphic to $\A^4$.
Moreover, the~following result has been recently proved in \cite{PZ20}.

\begin{theorem}
\label{thm:genus10}
The affine cones over any smooth Fano--Mukai fourfold of genus $10$~are~flexible.
\end{theorem}

For more higher-dimensional examples of flexible affine cones, see \cite{MPS18}.

\subsection{Cylinders in~complements to hypersurfaces}
\label{subsection:complements}
This section is motivated by the~following folklore conjecture that first appeared in 2005 \cite{FR05}.

\begin{conjecture}
\label{conj:Gizatullin}
Let $S$ be a~smooth cubic surface in $\mathbb{P}^3$.
Then any automorphism of the~affine variety $\mathbb{P}^3\setminus S$ is induced by an~automorphism of $\mathbb{P}^3$, i.e., we have
$$
\mathrm{Aut}\big(\mathbb{P}^3\setminus S\big)=\mathrm{Aut}\big(\mathbb{P}^3,S\big).
$$
\end{conjecture}

If $S$ is smooth surface in $\mathbb{P}^3$ of degree $\geqslant 4$, then it is easy to see that $\mathrm{Aut}(\mathbb{P}^3\setminus S)=\mathrm{Aut}(\mathbb{P}^3,S)$.
Vice versa, if $S$ is either a smooth quadric surface or a plane in $\mathbb{P}^3$, then $\mathrm{Aut}(\mathbb{P}^3\setminus S)\ne\mathrm{Aut}(\mathbb{P}^3,S)$.
Moreover, it is not hard to see that Conjecture~\ref{conj:Gizatullin} fails for \textbf{some} singular cubic surfaces.

\begin{example}
\label{example:cubic-surfaces}
Let $S$ be one of the three cubic surfaces with Du Val singularities in $\PP^3$ that
admits an effective $\Ga$-action (see \cite{CheltsovProkhorov,MartinStadlmayr,Sakamaki2010}).
Then $\mathrm{Aut}(\mathbb{P}^3,S)$  contains a subgroup isomorphic~to~$\Ga$,
so that $\mathrm{Aut}(\mathbb{P}^3\setminus S)$ also contains a subgroup isomorphic~to~$\Ga$.
Then $\mathrm{Aut}(\mathbb{P}^3\setminus S)$ must be infinite dimensional (see \cite{Fr06}),
so that $\mathrm{Aut}(\mathbb{P}\setminus S)\ne \mathrm{Aut}(\mathbb{P},S)$, because $\mathrm{Aut}(\mathbb{P},S)$ is algebraic.
\end{example}

Based on the~results in  \cite{CPW16a,CPW16b,CDP18},
we may generalize the~problem to del Pezzo surfaces that are hypersurfaces in weighted projective spaces.
To be precise, let $S$ be a~del Pezzo surface that has at most Du Val singularities such that $K_S^2\leqslant 3$.
Then we have one of the~following three cases:
\begin{enumerate}
\item $K_S^2=1$, and $S$ is a~hypersurface of degree $6$ in $\mathbb{P}(1,1,2,3)$;
\item $K_S^2=2$, and $S$ is a~hypersurface of degree $4$ in $\mathbb{P}(1,1,1,2)$;
\item $K_S^2=3$, and $S$ is a~hypersurface of degree $3$ in $\mathbb{P}^3$.
\end{enumerate}
Denote by $\mathbb{P}$ the~weighted projective space in these three cases:
$\mathbb{P}(1,1,2,3)$, $\mathbb{P}(1,1,1,2)$ or $\mathbb{P}^{3}$.
Then, very surprisingly, we have the~following result:

\begin{theorem}[{\cite{CDP18,Park20}}]
\label{theorem:del-Pezzo-1}
The following three conditions are equivalent:
\begin{itemize}
\item the~surface $S$ contains a~$(-K_S)$-polar cylinder;
\item the~complement $\mathbb{P}\setminus S$ is cylindrical;
\item the~group $\mathrm{Aut}(\mathbb{P}\setminus S)$ contains a~unipotent subgroup.
\end{itemize}
\end{theorem}

Combining this result with  Theorem~\ref{theorem:dP-du-Val}, we obtain

\begin{corollary}[{\cite[Corollary~1.6]{Park20}}]
\label{corollary:del-Pezzo}
The group $\mathrm{Aut}(\mathbb{P}\setminus S)$ contains no unipotent subgroup exactly when $S$ is one of the~surfaces listed in Theorem~\ref{theorem:dP-du-Val}.
\end{corollary}

\begin{corollary}[{\cite[Corollary~4.10]{CDP18}}]
\label{corollary:del-Pezzo-1}
Suppose that the~surface $S$ contains a~$(-K_S)$-polar cylinder. Then $\mathrm{Aut}(\mathbb{P}\setminus S)\ne \mathrm{Aut}(\mathbb{P},S)$.
\end{corollary}

\begin{proof}
By Theorem~\ref{theorem:del-Pezzo-1}, the~group $\mathrm{Aut}(\mathbb{P}\setminus S)$ contains a~unipotent subgroup,
so that it is infinite dimensional, which implies that $\mathrm{Aut}(\mathbb{P}\setminus S)\ne \mathrm{Aut}(\mathbb{P},S)$,
because $\mathrm{Aut}(\mathbb{P},S)$ is algebraic.
\end{proof}

This corollary together with Theorem~\ref{theorem:dP-du-Val} show that Conjecture~\ref{conj:Gizatullin} fails for \textbf{all} singular cubic surfaces
that have Du Val singularities.
On the~other hand, we have

\begin{theorem}[{\cite[Theorem~4.1]{CDP18}}]
\label{theorem:smoothdP}
Suppose that $S$ is smooth. If $K_S^2=1$, then
$$
\mathrm{Aut}\left(\mathbb{P}\setminus S\right)=\mathrm{Aut}\left(\mathbb{P},S\right).
$$
If $K_S^2=2$ or $K_S^2=3$, then $\mathrm{Aut}(\mathbb{P}\setminus S)$ does not contain nontrivial connected algebraic groups.
\end{theorem}

The proof of this result depends on irrationality of some del Pezzo threefolds (see \cite{CG,Grinenko03,Grinenko04,Voi88}).
Taking into account Theorem~\ref{theorem:smoothdP}, Corollary~\ref{corollary:dP-du-Val} and Corollary~\ref{corollary:del-Pezzo-1}, we propose the following

\begin{conjecture}
\label{conjecture:main}
The surface $S$ contains no $(-K_{S})$-polar cylinder $\iff$ $\mathrm{Aut}\left(\mathbb{P}\setminus S\right)=\mathrm{Aut}\left(\mathbb{P},S\right)$.
\end{conjecture}

If $S$ is a~smooth cubic surface, then it does not contain any $(-K_S)$-polar cylinder by Theorem~\ref{theorem:dP-du-Val}.
In this case, Conjecture~\ref{conjecture:main} claims that $\mathrm{Aut}\left(\mathbb{P}\setminus S\right)=\mathrm{Aut}\left(\mathbb{P},S\right)$,
which is Conjecture~\ref{conj:Gizatullin}.

In \cite{Park20}, Theorem~\ref{theorem:del-Pezzo-1} has been generalized as follows.
Let $X$ be a~normal projective variety, and let $D$ be an~ample Cartier divisor on $X$.
Suppose that  the~following conditions are satisfied:
\begin{enumerate}
\item\label{assumption-1}
the section ring of $(X, D)$ is a~hypersurface, i.e., one has
$$
\bigoplus_{m=0}^{\infty}\mathrm{H}^0(X, \mathcal{O}_X(mD))\cong \Bbbk[x_0,x_1,\ldots, x_n]/(F),
$$
where $\Bbbk[x_0,\ldots, x_n]$ is a~polynomial ring in variables $x_0,\ldots, x_n$ with weights
$$
a_0=\mathrm{wt}(x_0)\leqslant a_1=\mathrm{wt}(x_1)\leqslant\ldots\leqslant a_n=\mathrm{wt}(x_n),
$$
and $F$ is a~quasi-homogeneous polynomial of degree $d$,
so~that $X$ is a~hypersurface in the~weighted projective space  $\mathbb{P}(a_0,a_1,\ldots, a_n)=\mathrm{Proj}(\Bbbk[x_0,x_1,\ldots, x_n])$;
\item\label{assumption-2}
the Veronese map $v_d:\mathbb{P}(a_0, a_1,\ldots, a_n)\dasharrow \mathbb{P}^N$ given by
$\left|\mathcal{O}_{\mathbb{P}(a_0, a_1,\ldots, a_n)}\left(d\right)\right|$ is an~embedding.
\end{enumerate}
Recall from \cite[Proposition~3.5]{KPZ11} that the~complement  $\mathbb{P}(a_0, a_1,\ldots, a_n)\setminus X$ admits a~nontrivial $\Ga$-action
if and only if it is cylindrical. On the~other hand, we have the~following result:

\begin{theorem}[{\cite[Theorem~3.1]{Park20}}]
\label{theorem:main-tool}
Suppose that $\mathbb{P}(a_0, a_1,\ldots, a_n)\setminus X$ has a~nontrivial $\Ga$-action. Then $X$ contains a~$D$-polar cylinder.
\end{theorem}

Based on the~results on non-ruledness of  smooth hypersurfaces of low degrees in the~projective spaces such as
\cite{Cheltsov2000,CG,IM,Ko95,Pu87,Pu98,deFe13,Sch19} one can extend Conjecture~\ref{conj:Gizatullin} as follows:

\begin{conjecture}
\label{conjecture:big}
Let $X$ be a~smooth hypersurface in $\mathbb{P}^n$ of degree $d\geqslant 3$.
Then
$$
\mathrm{Aut}\big(\mathbb{P}^n\setminus X\big)=\mathrm{Aut}\big(\mathbb{P}^n,X\big).
$$
\end{conjecture}

The conjecture holds when $d>n$ since the~hypersurface $X$ has non-negative Kodaira dimension.
It~remains true if $d=n\geqslant 4$ and $(n,d)=(4,3)$ due to the~results by  \cite{Cheltsov2000,CG,IM,Pu87,Pu98,deFe13}.

\subsection{Compactifications of $\CC^n$}
\label{subsection:compactifications}
In this subsection, we assume that varieties are defined over~$\CC$.
In this case, the~problem of existence of (Zariski open) cylinders in smooth Fano varieties
is closely related to the~following famous problem posed by Hirzebruch 65 years ago in \cite{Hirzebruch1954}.

\begin{problem}
\label{problem:Hirzebruch}
Find all complex analytic compactifications of $\CC^n$ with second Betti number $1$.
\end{problem}

This problems asks to describe all compact complex manifolds $X$ with $\mathrm{b}_2(X)=1$ that contain
an~open subset $U$ which is biholomorphic to $\CC^n$ and whose complement $A=X\setminus U$ is a~closed complex analytic subspace.
Thus, we call a~\textit{compactification} of $\CC^n$ a~pair $(X,A)$ consisting of
\begin{itemize}
\item a~compact complex manifold $X$ with $\mathrm{b}_2 (X)=1$;
\item and a~closed complex analytic subset $A\subset X$ such that $X\setminus A \underset{\mathrm{bihol}}{\cong} \CC^n$.
\end{itemize}
A~compactification $(X,A)$ of $\CC^n$ is said to be \textit{algebraic} if $X$ is a~smooth projective variety,
and the~biholomorphism $X\setminus A \underset{\mathrm{bihol}}{\cong} \CC^n$ is an~algebraic isomorphism.
Thus, we see that
$$
\text{$(X,A)$ is an algebraic compactification of $\CC^n$}\Longrightarrow \text{$X$ is a cylindrical~Fano~variety}.
$$

\begin{proposition}[{\cite{Ven1962,Brenton1978}}]
\label{proposition:compactificationCn}
Let $(X,A)$ be a~compactification of $\CC^n$. Then the~following holds.
\begin{enumerate}
\item $A$ is purely $1$-codimensional and irreducible;
\item $H^{i}(X,\ZZ)\cong H^{i}(A,\ZZ)$,\ $H_{i}(X,\ZZ)\cong H_{i}(A,\ZZ)$ for every $i\leqslant 2n-2$;
\item $H^1(X,\ZZ)=0$ and $H_1(X,\ZZ)=0$;
\item the~class of $A$ generates the~groups $H^2(X,\ZZ)\cong \ZZ$ and $H^2(A,\ZZ)\cong \ZZ$;
\item if $X$ is Moishezon, then $H^1(X,\OOO_X)=0$ and $H^2(X,\OOO_X)=0$, so that $\Pic(X)\cong H^2(X,\ZZ)$.
\end{enumerate}
\end{proposition}

The following deep result is due to Kodaira \cite[Theorem~3]{Kodaira-MR0301228}:

\begin{theorem}
\label{theorem:Kodaira}
If $(X,A)$ is a~compactification of $\CC^n$, then
$$
h^0\big(X, \upomega_X^{\otimes m}\big)=0
$$
for every $m>0$, where $\upomega_X$ is the~sheaf of holomorphic $n$-forms on $X$.
\end{theorem}

Thus, if $(X,A)$ is a~compactification of $\CC^n$ and $X$ is projective, then $X$ is a~smooth Fano variety,
and $A$ is an ample divisor on $X$ that generates $\mathrm{Pic}(X)$.

\begin{example}
\label{example:compactifications}
Let $(X,A)$ be one of the~following polarized smooth Fano varieties:
\begin{enumerate}
\item $X=\PP^n$ and $A$ is a hyperplane;
\item $X$ is a~smooth quadric in $\mathbb{P}^{n+1}$ and $A$ is its~singular hyperplane section;
\item $X=\Gr(m,k)$ and $A$ is its~Schubert subvariety of codimension $1$, where $n=m(k-m)$;
\item $X=G/P$ and $A$ is its open cell isomorphic to $\CC^n$ (such a cell does exist by \cite{Kollar-1996-RC,Borel1954}), where $G$ is a~semisimple connected complex linear algebraic group, and $P$ is its maximal parabolic subgroup.
\end{enumerate}
Then $(X,A)$ is a compactification of $\CC^n$.
\end{example}

In two-dimensional case, Problem~\ref{problem:Hirzebruch} has an~easy solution:
if $(X,A)$ is a compactification of~$\CC^2$, then $X=\PP^2$ and $A$ is a line in $X$.
In the~three-dimensional case,  Problem~\ref{problem:Hirzebruch} has been solved in the~series of papers \cite{Furushima1986,Furushima1990,Furushima1992,Furushima1993,Fur93,Furushima1989,Peternell1989,Peternell-Schneider-1988,Pr91}.
In particular, we have the~following result:

\begin{theorem}
\label{theorem:compactifications-threefolds}
Let $(X,A)$ be a~compactification of $\CC^3$. Suppose that $X$ is a projective threefold.
Then this compactification is algebraic and $(X,A)$ can be described as follows:
\begin{enumerate}
\item $X=\PP^3$ and $A$ is a plane;
\item $X$ is a~smooth quadric in $\mathbb{P}^{4}$ and $A$ is its~singular hyperplane section;
\item $X$ is the~quintic del Pezzo threefold in $\mathbb{P}^5$ described in Example~\ref{example:V5} and $A$ is its singular
hyperplane section that can be described as follows:
\begin{enumerate}
\item a surface whose singular locus is a~line $L$ with normal bundle $\NNN_{L/X}=\OOO_{L}(1)\oplus \OOO_{L}(-1)$;
\item a normal del Pezzo surface that has a~unique singular point of type~$\mathrm{A_4}$;
\end{enumerate}
\item $X$ is a~smooth Fano threefold of index $1$ and genus $12$ in $\PP^{13}$ and $A$ is its certain hyperplane
section whose singular locus is a~line $\ell$ with normal bundle $\NNN_{\ell/X}=\OOO_{\ell}(1)\oplus \OOO_{\ell}(-2)$.
\end{enumerate}
\end{theorem}

\begin{proof}
We know that $X$ is a smooth Fano threefold, and the~surface $A$ generates $\mathrm{Pic}(X)$,
so that
$$
-K_X\sim\iota(X)A,
$$
where $\iota(X)$ is the~Fano index of the~threefold $X$.
If $\iota(X)=4$, then  $X=\mathbb{P}^3$ and $A$ is a plane.
Similarly, if $\iota(X)=3$, then $X$ is a smooth quadric threefold in $\PP^4$, and $A$ is its hyperplane section.
In this case, the~surface $A$ must be singular, since  $H^2(A,Z)=\ZZ$ by Proposition~\ref{proposition:compactificationCn}.

If $\iota(X)=1$, then the~surface $A$ must be a~non-normal K3 surface,
and the~proof uses a~delicate analysis of its singularities.
As a result, one can show that $X$ is a~Fano threefold of genus $12$ in~$\PP^{13}$,
and $A$ is its hyperplane section that is singular along a line of type $(1,-2)$.
One~construction~of~such compactification is described in Remark~\ref{remark:g-12}.
We will not dwell into further details in this case.

Suppose that $\iota(X)=2$.
Let us show that $X$ is the~quintic del Pezzo threefold in~$\mathbb{P}^5$,
and~$A$~is its singular hyperplane section described above.
Note that in this case $(X,A)$ is indeed a compactification of $\CC^3$,
which follows from the~proof of Theorem~\ref{theorem:V5-A3}.

By Proposition~\ref{proposition:compactificationCn}, we have $H^2(A,Z)=\ZZ$ and
\begin{equation}
\label{eq:comp:iota=2}
4+2h^{1,2}(X)=\chit(X)=\chit(A)+1.
\end{equation}

First, we suppose that the~surface $A$ is normal.
Then $-K_A$ is ample by the~adjunction formula, so that $A$ is a~del Pezzo surface with isolated Gorenstein singularities.
If its~singularities are worse than Du Val, then $A$ must be a~(generalized) cone over an~elliptic curve \cite{Hidaka-Watanabe-1981}, so that $\chit(A)=1$.
The latter contradicts \eqref{eq:comp:iota=2}.
Thus, we see that $A$ is a~del Pezzo surface with Du Val singularities.
Then $\uprho(A)=1$, because $H^2(A,Z)=\ZZ$.
Then $\chit(A)=3$, so that we have $h^{1,2}(X)=0$ by \eqref{eq:comp:iota=2}.
Now, using Remark~\ref{remark:del-Pezzo-n-folds}, we conclude that $X$ is the~quintic del Pezzo threefold in $\mathbb{P}^5$ as required.
Moreover, we have $K_A^2=5$, so that $A$ is a quintic del Pezzo surface that has Du Val singularities.
Since $\uprho(A)=1$, it follows from \cite{Furushima1986,Miyanishi-Zhang-1988} that $A$ has a~unique singular point of type $\mathrm{A}_4$.

Now, we suppose that $A$ is non-normal, so that it has a~singular locus of positive dimension.
It~is easy to show that any hyperplane section of a~smooth complete intersection has only isolated singularities,
and the~same result holds for hyperplane sections of weighed smooth hypersurfaces.
Therefore, using Remark~\ref{remark:del-Pezzo-n-folds}, we conclude again that $X$ is the~quintic del Pezzo threefold in~$\mathbb{P}^5$,
and~$A$ is its hyperplane section.
Using the~adjunction formula, we see that a~general hyperplane section of the~surface $A$
is an~irreducible singular curve of arithmetic genus $1$, so that it has one singular point.
Thus, the~non-normal locus of the~surface $A$ is some line $L$.
Hence, it follows from the~proof of Theorem~\ref{theorem:V5-A3} that $\mathrm{Sing}(A)=L$  and
$$
X\setminus A\cong Q\setminus H
$$
where $Q$ is a smooth quadric threefold in $\mathbb{P}^4$, and $H$ is its hyperplane section.
Since $X\setminus A\cong\mathbb{C}^3$, we conclude that the~surface $H$ is singular.
As we already mentioned in the~proof of Theorem~\ref{theorem:V5-A3},
this implies that $\NNN_{L/X}=\OOO_{L}(1)\oplus \OOO_{L}(-1)$ as required.
\end{proof}

\begin{corollary}
\label{corollary:compactifications-threefolds}
Let $(X,A)$ be a~compactification of $\CC^3$. Suppose that $X$ is a projective threefold.
Then $H^k(X,\ZZ)\cong H^k(\PP^3,\ZZ)$ for all $k$.
\end{corollary}

It would be interesting to find an alternative proof of Theorem~\ref{theorem:compactifications-threefolds}
that does not heavily rely on the~classification of smooth Fano threefolds.

\begin{remark}
\label{remark:V22-A3}
Let $X$ be a~smooth Fano threefold such that $\uprho(X)=1$, $\iota(X)=1$ and $\g(X)=12$.
If~$X$ is a compactification of $\CC^4$, then $X$ contains a line $\ell$ such that $\NNN_{\ell/X}=\OOO_{\ell}(1)\oplus \OOO_{\ell}(-2)$.
However, this condition does not always guarantee that $X$ is a compactification of $\CC^4$ (see \cite{Pr91}).
\end{remark}

\begin{remark}
\label{remark:the-list}
The~list  in Theorem~\ref{theorem:compactifications-threefolds} is similar to the~list in Theorem~\ref{theorem:Fano-threefolds-Aut-infinite}.
\end{remark}

In higher dimensions, we know very few results on Problem~\ref{problem:Hirzebruch}.
Let us present one of them, which follows from Theorem~\ref{theorem:V5-A4} and its proof.
We use here the~notation introduced in Section~\ref{subsection:Fano-4-folds}.

\begin{theorem}[{\cite{Pr94}}]
\label{thm:comp:C4:i=3}
Let $(X,A)$ be a~compactification of $\CC^4$, where $X$ is a~smooth Fano fourfold.
Suppose that $\iota(X)=3$. Then $X$ is the~quintic del Pezzo fourfold in $\PP^7$ and
\begin{enumerate}
\item either $A=H_\ell$, where $\ell$ is a line in $X$ that is not a line of type \textup{\ref{line:type:b}};
\item or $A$ is a~singular hyperplane section of the~del Pezzo fourfold $X$ such that its singular locus consists of a single ordinary double point that is not contained in the~divisor $\mathcal{H}$.
\end{enumerate}
Each of these compactifications is algebraic and unique up to isomorphism.
\end{theorem}

\begin{proof}
We prove the~existence part only.
In the~first case, the~existence follows from Theorem~\ref{theorem:V5-A4}.
To deal with the~second case, let us use the~notation introduced in the~proof of Theorem~\ref{theorem:V5-A4}.
Consider the~Sarkisov link \eqref{equation:V5-Sarkisov-link-2} with $\ell$ being a~line of type \ref{line:type:a}.
We~already know that $Q$ is smooth, and so we may assume that it is given in $\PP^4$ by
$$
x_2x_3+x_1x_4+x_{0}x_5=0.
$$
Similarly, we may assume that $\eta(F)$ is cut out by $x_0=0$.
Moreover, the~surface $\eta(\widehat{H}_\ell)$ is a smooth cubic scroll in this case.
Hence, we may assume that it is cut out on $Q$ by
$$
\left\{\aligned%
&x_0=0,\\
&x_2x_4+x_1x_5=0,\\
&x_4^2-x_3x_5=0.\\
\endaligned
\right.
$$
Let $D$ be the~hyperplane section of the~quadric $Q$ that is cut out by $x_3=0$, and let $\widehat{D}$ be its proper transform on $\widehat{V}_5$.
Then $D$ is singular. We claim that $\widehat{V}_5\setminus(\widehat{D}\cup\widehat{H}_\ell)\cong\A^4$.
Indeed, let $U=Q\setminus D$. Then $U\cong\A^4$ with coordinates $y_0=\frac{x_0}{x_3}$, $y_1=\frac{x_1}{x_3}$, $y_4=\frac{x_4}{x_3}$, $y_5=\frac{x_5}{x_4}$,
so that $\widehat{V}_5\setminus\widehat{D}$ is given  by
$$
y_0z_0=(y_5-y_4^2)z_1
$$
in $\A^4\times \PP^1$, where $z_0$ and $z_1$ are coordinates on $\PP^1$. Then  $\widehat{V}_5\setminus(\widehat{D}\cup\widehat{H}_\ell)$ is given in $\A^4\times \mathbb{A}^1$ by
$$
y_0z=y_5-y_4^2,
$$
where $z=\frac{z_0}{z_1}$.
This implies that $\widehat{V}_5\setminus(\widehat{D}\cup\widehat{H}_\ell)\cong\A^4$.
Now, observe that $\pi(\widehat{D})$ is a hyperplane section of $V_5$
whose singular locus consists of a single ordinary double point not contained in $\mathcal{H}$.
\end{proof}

In dimension $4$, we know very few compactifications $(X,A)$  of $\CC^4$.
They can be listed as follows:
\begin{itemize}
\item $X=\mathbb{P}^4$ and $A$ is a hyperplane;
\item $X$ is a smooth quadric and $A$ is its singular hyperplane section;
\item $X$ is the~del Pezzo quintic fourfold and $A$ is described in Theorem~\ref{thm:comp:C4:i=3};
\item $X$ is a smooth Fano--Mukai fourfold of genus $10$ and $A$ is described in Theorem~\ref{theorem:FM-g-10}.
\end{itemize}
In particular, in every known example of a compactification $(X,A)$ of $\CC^4$ with $X\not\cong\PP^4$, one has
$$
H^k\big(X,\ZZ\big)\cong H^k\big(Q,\ZZ\big)
$$
for all $k$, where $Q$ is a~smooth quadric in $\PP^5$. We wonder whether this is just a coincidence.

\begin{question}
\label{question:Fano-C4-index-1}
Does there exist a smooth Fano fourfold of index $1$ that is a compactification~of~$\CC^4$?
\end{question}

Before we conclude this survey, let us set the~following question:

\begin{question}
\label{question:Cn-rational}
Is it true that any compactification of $\CC^n$ is rational?
\end{question}

Note that the~answer to this question is not obvious,
since the~isomorphism $X\setminus A\cong \CC^n$ in the~definition of a compactification of $\CC^n$
is a~biholomorphism, which is not necessarily algebraic.

\appendix

\section{Singularities of pairs}
\label{section:log-pairs}
Let $S$ be a~surface with at most quotient singularities, let $D$ be an~effective non-zero
$\mathbb{Q}$-divisor on $S$, let $P$ be a~point of $S$, and let
$$
D=\sum_{i=1}^{r}a_iC_i,
$$
where $C_1,\ldots,C_r$ are distinct irreducible curves on $S$, and each $a_i$ is a~non-negative
rational number.
We call $(S,D)$ a~log pair.

Let $\pi\colon\widetilde{S}\to S$ be a~birational morphism such that $\widetilde{S}$ is smooth.
For each $C_i$, denote by $\widetilde{C}_i$ its proper transform on the~surface $\widetilde{S}$.
Let $F_1,\ldots, F_n$ be $\pi$-exceptional curves.
Then
$$
K_{\widetilde{S}}+\sum_{i=1}^{r}a_i\widetilde{C}_i+\sum_{j=1}^{n}b_jF_j\sim_{\mathbb{Q}}\pi^{*}\left(K_{S}+D\right)
$$
for some rational numbers $b_1,\ldots,b_n$. Suppose that $\widetilde{C}_1+\cdots+\widetilde{C}_2+F_1+\cdots+F_n$ is a divisor with simple normal crossings.
Then we say that $\pi\colon\widetilde{S}\to S$ is a log resolution of the~log pair $(S,D)$.

\begin{definition}
\label{definition:lct}
The log pair $(S,D)$ is said to be log canonical at the~point $P$ if the~following two conditions are satisfied:
\begin{itemize}
\item $a_i\leqslant 1$ for every $C_i$ such that $P\in C_i$;
\item $b_j\leqslant 1$ for every $F_j$ such that $\pi(F_j)=P$.
\end{itemize}
The log pair $(S,D)$ is called log canonical if it is log canonical at every point of $S$.
\end{definition}

This definition does not depend on the~choice of the~log resolution $\pi\colon\widetilde{S}\to S$.

\begin{remark}
\label{remark:convexity}
Let $R$ be an~effective $\mathbb{Q}$-divisor on $S$ such that $R\sim_{\mathbb{Q}} D$.
For a~rational number $\epsilon$, let
$$
D_{\epsilon}=(1+\epsilon) D-\epsilon R.
$$
Then $D_{\epsilon}\sim_{\mathbb{Q}} D$.
Suppose that $R\ne D$. Then there exists the~greatest rational number $\epsilon_0\geqslant 0$ such that the~divisor $D_{\epsilon_0}$ is effective.
By construction, the~support of the~divisor $D_{\epsilon_0}$ does not contain at least one curve contained in the~support of the~divisor $R$.
Moreover, if $(S,D)$ is not log canonical at $P$, but $(S,R)$ is log canonical at~$P$,
then $(S,D_{\epsilon_0})$ is not log canonical at $P$.
\end{remark}

Now, we suppose that the~surface $S$ is smooth at $P$.

\begin{lemma}
\label{lemma:Skoda}
Suppose that $(S,D)$ is not log canonical at $P$. Then $\mathrm{mult}_{P}(D)>1$.
\end{lemma}

\begin{proof}
Left to the~reader.
\end{proof}

Let $f\colon\overline{S}\to S$ be a~blowup of the~point~$P$, and let
$E$ be the~$f$-exceptional curve. Denote by $\overline{D}$ the~proper
transform of the~$\mathbb{Q}$-divisor $D$ on the~surface $\overline{S}$ via $f$.
Then the~log pair
\begin{equation}
\label{equation:log-pull-back}
\Big(\overline{S}, \overline{D}+\big(\mathrm{mult}_{P}(D)-1\big)E\Big)
\end{equation}
is called the~log pull back of the~log pair $(S,D)$ on the~surface $\overline{S}$.

\begin{lemma}
\label{lemma:blow-up}
Suppose that the~log pair $(S,D)$ is not log canonical at $P$. Then
\begin{itemize}
\item[(i)] the~$\mathbb{Q}$-divisor $\overline{D}+(\mathrm{mult}_{P}(D)-1)E$ is effective;

\item[(ii)] the~log pair \eqref{equation:log-pull-back} is not log canonical at some point $Q\in E$.
\end{itemize}
\end{lemma}

\begin{proof}
The required assertion follows from Definition~\ref{definition:lct} and Lemma~\ref{lemma:Skoda}.
\end{proof}

The following handy statement is a~very special case of a~much more general
result, which is known as \emph{Inversion of Adjunction} (see, for example, \cite[Theorem~6.29]{CoKoSm}).

\begin{lemma}[{\cite[Exercise~6.31]{CoKoSm}}]
\label{lemma:adjunction}
Suppose that $C_1$ is smooth at $P$, the~log pair $(S,D)$ is not log canonical at $P$, and $a_1\leqslant 1$.
Let $\Delta=a_2C_2+\cdots+a_rC_r$. Then $(C_1\cdot\Delta)_P>1$.
\end{lemma}

\begin{proof}
Let $m=\mathrm{mult}_P(\Delta)$. If $m>1$, then we are done, since
$$
\big(C_1\cdot\Delta\big)_{P}\geqslant m.
$$
Therefore, we may assume that $m\leqslant 1$.
This implies that the~log pair $(S,D)$ is log
canonical in a~punctured neighborhood of the~point $P\in S$.
Since the~log pair $(S,D)$ is not log canonical at~$P$,
there exists a~birational morphism $h\colon\widehat{S}\to S$ that is a~composition of $s\geqslant 1$ blowups of points dominating $P$
such that $e_s>1$, where $e_s$ is a~rational number determined by
$$
K_{\widehat{S}}+a_1\widehat{C}_1+\widehat{\Delta}+\sum_{i=1}^{s}e_iE_i\sim_{\mathbb{Q}}h^{*}\big(K_{S}+D\big),%
$$
where each $e_i$ is a~rational number, each $E_i$ is an~$h$-exceptional divisor,  $\widehat{\Delta}$ is a~proper transform on
the~surface $\widehat{S}$ of the~divisor $\Delta$, and $\widehat{C}_1$ is a~proper
transform on  $\widehat{S}$ of the~curve $C_1$.

Let $\overline{\Delta}$ and $\overline{C}_1$ be the~proper transforms on $\overline{S}$ of the~divisor $\Delta$ and the~curve $C_1$, respectively.
Then $(\overline{S}, a_1\overline{C}_1+(a_1+m-1)E+\overline{\Delta})$ is not log
canonical at some point $Q\in E$ by Lemma~\ref{lemma:blow-up}.

Let us prove the~inequality $(C_1\cdot\Delta)_{P}>1$ by induction on $s$. If $s=1$, then
$$
a_1+m-1>1,
$$
which implies that $m>2-a_1\geqslant 1$, so that $(C_1\cdot\Delta)_{P}\geqslant m>1$ as required.
Thus, we may assume that $s\geqslant 2$ and $a_1+m-1\leqslant 2$. Since
$$
\big(C_1\cdot\Delta\big)_{P}\geqslant m+\big(\overline{C}_1\cdot\overline{\Delta}\big)_{Q},
$$
it is enough to show that
$m+(\overline{C}_1\cdot\overline{\Delta})_{Q}>1$.
We may also assume that $m\leqslant 1$, since $(C_1\cdot\Delta)_{P}\geqslant m$.

If $Q\not\in\overline{C}_1$, then $(\overline{S},(a_1+m-1)E+\overline{\Delta})$ is not log canonical at
the~point $Q$, which gives
$$
m=\overline{\Delta}\cdot E\geqslant\big(\overline{\Delta}\cdot E\big)_{Q}>1
$$
by induction. The latter implies that $Q=\overline{C}_1\cap E$, since
$m\leqslant 1$. Then
$$
a_1+m-1+\big(\overline{C}_1\cdot\overline{\Delta}\big)_{Q}=\Big(\big((a_1+m-1)E+\overline{\Delta}\big)\cdot\overline{C}_1\Big)_{Q}>1%
$$
by induction. This gives $(\overline{C}\cdot\overline{\Delta})_{Q}>2-a_1-m$. Then
$$
m+\big(\overline{C}_1\cdot\overline{\Delta}\big)_{Q}>2-a_1\geqslant 1
$$
as required.
\end{proof}

\begin{corollary}
\label{corollary:log-pull-back}
In the~notation and assumptions of Lemma~\ref{lemma:blow-up},
suppose that $\mathrm{mult}_{P}(D)\leqslant 2$.
Then there exists a~unique point $Q\in E$ such that \eqref{equation:log-pull-back} is not log canonical at $Q$.
\end{corollary}

\begin{proof}
If \eqref{equation:log-pull-back} is not log canonical at two distinct
points $P_1$ and $P_2$, then
$$
2\geqslant\mathrm{mult}_{P}\left(D\right)=\overline{D}\cdot E\geqslant\big(\overline{D}\cdot E\big)_{P_1}+\big(\overline{D}\cdot E\big)_{P_2}>2
$$
by Lemma~\ref{lemma:adjunction}. Now use Lemma~\ref{lemma:blow-up}.
\end{proof}

The following result plays an essential role in the~proof of Theorem~\ref{theorem:del-Pezzo-degree-1-2-3} given in Section~\ref{subsection:del-Pezzo-no-cylinders}.
In~fact, this theorem has been discovered \cite{Ch14} in an~attempt to give a simple proof of Theorem~\ref{theorem:del-Pezzo-degree-1-2-3},
since its original proof in \cite{CPW16a} is very technical.
For other applications of Theorem~\ref{theorem:del-Pezzo-degree-1-2-3}, see \cite{CheltsovAhmadinezhadSchicho,Viswanathan}.

\begin{theorem}[{\cite{Ch14}}]
\label{theorem:Vanya}
Suppose that $(C_1\cdot C_2)_P=1$, and the~log pair $(S,D)$ is not log canonical~at~$P$.
Let $\Delta=a_3C_3+\cdots+a_rC_r$ and $m=\mathrm{mult}_P(\Delta)$.
Suppose also that $m\leqslant 1$.
Then
$$
\Big(C_1\cdot\Delta\Big)_{P}>2(1-a_{2})
$$
or
$$
\Big(C_2\cdot\Delta\Big)_{P}>2(1-a_{1}).
$$
\end{theorem}

\begin{proof}
We may assume that $a_1\leqslant 1$ and $a_2\leqslant 1$.
There is a~morphism
$h\colon\widehat{S}\to S$ that is a~composition of $s\geqslant 1$ blowups of points dominating $P$ such that $e_s>1$
for $e_s\in\mathbb{Q}$ that is determined~by
$$
K_{\widehat{S}}+a_1\widehat{C}_1+a_2\widehat{C}_2+\widehat{\Delta}+\sum_{i=1}^{r}e_iE_i=h^{*}\left(K_{S}+a_{1}C_{1}+a_{2}C_{2}+\Delta\right),%
$$
where each $e_i$ is a~rational number, each $E_i$ is an~$h$-exceptional divisor,
$\widehat{C}_1$ and $\widehat{C}_2$, are proper transforms on $\widehat{S}$ of the~curves $C_1$ and
$C_2$, respectively, and $\widehat{\Delta}$ is a~proper transform of the~divisor~$\Delta$.

Let $\overline{\Delta}$, $\overline{C}_1$, $\overline{C}_2$ be the~proper transforms on $\overline{S}$ of the~divisors $\Delta$, $C_1$ and $C_2$, respectively. Then
$$
\Big(\overline{S}, a_1\overline{C}_1+a_2\overline{C}_2+\big(a_1+a_2+m-1\big)E+\overline{\Delta}\Big)
$$
is not log canonical at some point $Q\in E$ by Lemma~\ref{lemma:blow-up}.

If $s=1$, then $a_1+a_2+m-1>1$.
If  $m>2-a_1-a_2$, then $m>2(1-a_1)$ or $m>2(1-a_2)$, because otherwise we would have
$$
2m\leqslant 4-2(a_1+a_2),
$$
which contradicts to $m>2-a_1-a_2$. Then $(\Delta\cdot C_{1})_P>2(1-a_{2})$ or $(\Delta\cdot C_{2})_P>2(1-a_{1})$ if $s=1$.

Let us  prove the~required assertion by induction on $s$.
The case $s=1$ is already done, so that we may assume that $s\geqslant 2$ and $a_1+a_2+m\leqslant 2$.
If $Q\ne E\cap\overline{C}_1$ and $Q\ne E\cap\overline{C}_2$, then
$$
m=\overline{\Delta}\cdot E>1
$$
by Lemma~\ref{lemma:adjunction}, which is impossible by assumption.
Thus, either $Q=E\cap\overline{C}_1$ or $Q=E\cap\overline{C}_2$. Without
loss of generality, we may assume that $Q=E\cap\overline{C}_1$.

By induction, we can apply the~lemma to
$(\overline{S}, a_1\overline{C}_1+\left(a_1+a_2+m-1\right)E+\overline{\Delta})$ at the~point $Q$. This implies that either
$$
\big(\overline{\Delta}\cdot\overline{C}_1\big)_{Q}>2\left(1-(a_1+a_2+m-1)\right)=4-2a_1-2a_2-2m
$$
or $\big(\overline{\Delta}\cdot E\big)_{Q}>2(1-a_1)$. In the~latter case, we have
$$
\big(\Delta\cdot C_{2}\big)_{P}\geqslant m=\overline{\Delta}\cdot E\geqslant\big(\overline{\Delta}\cdot E\big)_{Q}>2(1-a_1),
$$
which is exactly what we want.
Therefore, we may assume that $(\overline{\Delta}\cdot \overline{C}_1)_{Q}>4-2a_1-2a_2-2m$.
If $(\Delta\cdot C_{2})_{P}>2(1-a_1)$, then we are done.
Hence, we may assume $(\Delta\cdot C_{2})_{P}\leqslant 2(1-a_1)$.
Then
$$
m\leqslant\big(\Delta\cdot C_{2}\big)_P\leqslant 2(1-a_1).
$$
This gives
$$
\big(\Delta\cdot C_{1}\big)_{P}\geqslant m+\big(\overline{\Delta}\cdot \overline{C}_1\big)_{Q}>m+4-2a_1-2a_2-2m>2(1-a_2),%
$$
because $m\leqslant 2(1-a_1)$.
\end{proof}

Almost all results we have considered so far in this subsection are local (except for Remark~\ref{remark:convexity}).
Let us conclude this subsection by two global statements.
The first of them is due to Puhklikov:

\begin{lemma}[{\cite[Lemma~5.36]{CoKoSm}}]
\label{lemma:Pukhlikov}
Suppose that $S$ is a~smooth surface in $\mathbb{P}^3$, and $D$ is $\mathbb{Q}$-linearly equivalent to its hyperplane section.
Then each $a_i$ does not exceed~$1$.
\end{lemma}

\begin{proof}
Let $X$ be a~cone over the~curve $C_i$ whose vertex is a~general enough point in $\mathbb{P}^3$.
Then
$$
X\cap S=C_i+\widehat{C}_i,
$$
where $\widehat{C}_i$ is an~irreducible curve of degree $(\mathrm{deg}(S)-1)\mathrm{deg}(C_i)$.
Moreover, $\widehat{C}_i$ is not contained in the~support of the~divisor $D$,
and the~intersection
$C_i\cap\widehat{C}_i$
consists of exactly $\mathrm{deg}(\widehat{C}_i)$ points.
Then
$$
\mathrm{deg}\big(\widehat{C}_i\big)=D\cdot \widehat{C}_i\geqslant a_iC_i\cdot \widehat{C}_i\geqslant a_i\mathrm{deg}\big(\widehat{C}_i\big),
$$
which implies that $a_i\leqslant 1$.
\end{proof}

The second global result we want to mention is the~following lemma about del Pezzo surfaces of degree $2$ that have at most two ordinary double points.

\begin{lemma}
\label{lemma:double-plane}
Suppose that there is a~double cover $\tau\colon S\to\mathbb{P}^2$ branched over an~irreducible quartic curve $B$
that has at most two ordinary double points, and
$$
D\sim_{\mathbb{Q}} -K_S.
$$
Then each $a_i$ does not exceed~$1$.
Moreover, if $(S,D)$ is not log canonical at $P$, then $\tau(P)\in B$.
\end{lemma}

\begin{proof}
Write  $D=a_1C_1+\Delta$, where $\Delta=a_2C_2+\cdots+a_rC_r$.
Suppose that $a_1>1$. Let us seek for a~contradiction.
Since
$$
2=-K_{S}\cdot D=-K_{S}\cdot\left(a_1C_1+\Delta\right)=-a_1K_{S}\cdot C_1-K_{S}\cdot\Delta \geqslant -a_1K_{S}\cdot C_1>-K_{S}\cdot C_1,%
$$
we have $-K_{S}\cdot C_1=1$. Then $\tau(C_1)$ is a~line.
Hence, the~surface $S$ contains an~irreducible curve~$Z_1$ such that $C_1+Z_1\sim -K_{S}$ and $\tau(C_1)=\tau(Z_1)$.
Note that the~curves~$C_1$ and $Z_1$ are interchanged by the~biregular involution of the~surface $S$ induced by the~double cover $\tau$.
Then
$$
2=(-K_S)^2=\big(C_1+Z_1\big)^2=2C_1^2+2C_1\cdot Z_1,
$$
which implies that $C_1\cdot Z_1=1-C_1^2$. Since $C_1$ and $Z_1$ are smooth rational curves, we have
$$
C_1^{2}=Z_1^2=-1+\frac{k}{2},
$$
where $k$ is the~number of singular points of $S$ that lie on  $C_1$.
Now we write $D=a_1C_1+b_1 Z_1+\Theta$, where $b_1$ is a~non-negative rational number,
and $\Theta$ is an~effective $\mathbb{Q}$-divisor whose support does not contains the~curves $C_1$ and  $Z_1$. Then
$$
1=C_1\cdot\big(a_1C_1+b_1 Z_1+\Theta\big)\geqslant  a_1C_1^2+b_1 C_1\cdot Z_1=a_1C_1^2+b_1\big(1-C_1^2\big),
$$
and hence $1\geqslant a_1C_1^2+b_1(1-C_1^2)$. Similarly, from $Z_1\cdot D=1$, we obtain
$$
1\geqslant b_1  C_1^2+a_1\big(1-C_1^2\big).
$$
The obtained two inequalities imply that $a_1\leqslant 1$ and $b_1\leqslant 1$, because $C_1^2=-1+\frac{k}{2}$ and $k\leqslant 2$.
Since $a_1>1$ by our assumption, this is a~contradiction.

We see that $a_1\leqslant 1$. Similarly, we see that $a_i\leqslant 1$ for every $i$.

Now we suppose that the~log pair  $(S,D)$ is not log canonical at $P$. Let us show that $\tau(P)\in B$.
Suppose that $\tau(P)\not\in B$. Then $S$ is smooth at $P$. Let us seek for a contradiction.

Let $H$ be a~general curve in $|-K_{S}|$ that passes through the~point $P$. Then
$$
2=H\cdot D\geqslant\mathrm{mult}_{P}(H)\mathrm{mult}_{P}(D)\geqslant \mathrm{mult}_{P}(D),
$$
so that $\mathrm{mult}_{P}(D)\leqslant 2$.
But the~pair \eqref{equation:log-pull-back} is not log canonical at some point $Q\in E$ by Lemma~\ref{lemma:blow-up}.
Applying Lemma~\ref{lemma:Skoda} to \eqref{equation:log-pull-back}, we get $\mathrm{mult}_{P}(D)+\mathrm{mult}_{Q}(\overline{D})>2$.

Since $\tau(P)\not\in B$, there exists a~unique (possibly reducible) curve $R\in|-K_{S}|$
such that its proper transform on $\overline{S}$ passes through the~point $Q$.
Note that  $R$ is smooth at  $P$.
This enables us to assume that the~support of $D$ does not contain at least one irreducible component of  $R$ by Remark~\ref{remark:convexity}.
Denote by $\overline{R}$ the~proper transform of $R$ on the~surface $\overline{R}$.
If the~curve $R$ is irreducible, then
$$
2-\mathrm{mult}_{P}(D)=2-\mathrm{mult}_{P}(C)\mathrm{mult}_{P}(D)=\overline{R}\cdot\overline{D}\geqslant\mathrm{mult}_{Q}(\overline{R})\mathrm{mult}_{Q}(\overline{D})=\mathrm{mult}_{Q}(\overline{D}),
$$
which is impossible, since  $\mathrm{mult}_{P}(D)+\mathrm{mult}_{Q}(\overline{D})>2$.
Thus, the~curve $R$ must be reducible.

Write $R=R_1+R_2$, where $R_1$ and $R_2$ are irreducible smooth curves.
Without loss of generality we may assume that the~curve $R_1$ is not contained in $\mathrm{Supp}(D)$.
Then $P\in R_2$, because otherwise we would have
$$
1=D\cdot R_1\geqslant\mathrm{mult}_{P}(D)>1,
$$
since $\mathrm{mult}_{P}(D)>1$ by Lemma~\ref{lemma:Skoda}.
Thus, we put $D=aR_2+\Omega$, where $a$ is a~non-negative rational number and $\Omega$ is an~effective $\mathbb{Q}$-divisor whose
support does not contain the~curve $R_2$. Then
$$
1=R_1\cdot D=\left(2-\frac{1}{2}l\right)a+R_1\cdot\Omega\geqslant \left(2-\frac{1}{2}l\right)a,
$$
where $l$ is the~number of singular points of the~surface $S$ contained in $R_1$.
Denote by $\overline{R}_2$ the~proper transform on $\overline{S}$ of the~curve $R_2$ ,
and denote by $\overline{\Omega}$ the~proper transform on $\overline{S}$ of the~divisor $\Omega$.
Then the~log pair
$$
\Big(\overline{S}, a\overline{R}_2+\overline{\Omega}+\big(\mathrm{mult}_P(D)-1\big)E\Big)
$$
is not log canonical at $Q$. Note that we already proved that $a\leqslant 1$.
Thus, using Lemma~\ref{lemma:adjunction}, we~get
$$
\left(2-\frac{1}{2}l\right)a=\overline{R}_2\cdot \Big(\overline{\Omega}+\big(\mathrm{mult}_P(D)-1\big)E\Big)>1.
$$
This is a~contradiction.
\end{proof}

\newcommand{\etalchar}[1]{$^{#1}$}
\def\cprime{$'$}


\end{document}